\documentclass[reqno]{amsart}
\usepackage{amsmath,amssymb,stmaryrd}
\usepackage{amsrefs}
\usepackage{amsfonts}

\usepackage[colorlinks=true, pdfborder={0 0 0}]{hyperref}
\usepackage{cleveref}
\usepackage{upref}

\newtheorem{theorem}{Theorem}[section]
\newtheorem{defi}{Definition}[section]
\newtheorem{proposition}{Proposition}[section]
\newtheorem{lemma}{Lemma}[section]

\def \rr {\mathbb{R}}
\def \rn {\mathbb{R}^n}
\def \nn {\mathbb{N}}
\def \eps {\epsilon}
\def \crit {2^\star}
\def \Pl {{\mathcal P}_\lambda}
\def \Mmx {M - \{p_1,...,p_N\}}
\def \ua {u_\alpha}

\def \va {v_\alpha}

\def \tpl {\tilde{\psi}_\lambda}
\def \F {{\mathcal F}}

\def \na {\nu_\alpha}
\def \za {z_\alpha}
\def \xaun {x_{\alpha,1}}
\def \xai {x_{\alpha,i}}
\def \maun {\mu_{\alpha,1}}
\def \mai {\mu_{\alpha,i}}
\def \tuaun {\tilde{u}_{\alpha,1}}
\def \tuai {\tilde{u}_{\alpha,i}}
\def \maj {\mu_{\alpha,j}}
\def \xaj {x_{\alpha,j}}
\def \tvap {\tilde{v}_{\alpha,p}}
\def \nap {\nu_{\alpha,p}}
\def \naq {\nu_{\alpha,q}}
\def \yap {y_{\alpha,p}}
\def \yaq {y_{\alpha,q}}
\def \zai {z_{\alpha,i}}
\def \zaj {z_{\alpha,j}}
\def \zal {z_{\alpha,l}}
\def \mal {\mu_{\alpha,l}}

\title{Sharp multiscale control for  high order nonlinear equations}

\author{Fr\'ed\'eric Robert}
\address{Fr\'ed\'eric Robert, Institut \'Elie Cartan, Universit\'e de Lorraine, CNRS, IECL, F-54000 Nancy, France}
\email{frederic.robert@univ-lorraine.fr}

\date{September 5th, 2025}

\subjclass[2020]{Primary 35J35, Secondary 35J60, 35B44, 35J08, 58J05}
\keywords{critical elliptic equations, instability, concentration phenomena, multibump description}

\begin{document}
\begin{abstract} We analyze the behavior of families $(\ua)_{\alpha>0}$ of solutions to the high-order critical equation $P_\alpha\ua=\Delta_g^k\ua +\hbox{lot}=|\ua|^{\crit-2}\ua$ on a Riemannian manifold $M$, with a uniform bound on the Dirichlet energy. We prove a sharp pointwise   control of the $\ua$'s by a sum of bubbles uniformly with respect to $\alpha\to +\infty$, that is
$|\ua|\leq  C\Vert u_\infty \Vert_\infty +C\sum_{i=1}^NB_{i,\alpha}$ where $u_\infty \in C^{2k}(M)$ and the $(B_{i,\alpha})_\alpha$, $i=1,...,N$ are explicit standard peaks. 
\end{abstract}
\maketitle

\section{Introduction and statement of the result}
Let $(M,g)$ be a smooth compact Riemannian manifold of dimension $n\geq 3$ without boundary, and let $k\in\nn$ be such that $2\leq 2k<n$. We consider families of function $(\ua)_{\alpha>0}\in C^{2k}(M)$ that are solutions to
\begin{equation}\label{eq:ua}
P_\alpha u_\alpha=|u_\alpha|^{\crit-2}u_\alpha\hbox{ in }M,
\end{equation}
where $\crit:=\frac{2n}{n-2k}$, and $(P_\alpha)_{\alpha>0}$ is a family of   elliptic symmetric differential operators of the type
\begin{equation}\label{model:P}
P:=\Delta_g^k+\sum_{i=0}^{k-1} (-1)^i\nabla^i(A^{(i)} \nabla^i)\hbox{ for }\alpha>0
\end{equation}
where $\Delta_g:=-\hbox{div}_g\nabla$ is the Riemannian Laplacian with minus-sign convention and for all $i=0,...,k-1$, $ A^{(i)}\in C^{i}_{\chi}(M,\Lambda_{S}^{(0,2i)}(M))$ is a $(0,2i)-$tensor field of class $C^i$ on $M$ such that for any $i$ , $A^{(i)}(T,S)=A^{(i)}(S,T)$ for any $(i,0)-$tensors $S$ and $T$. In coordinates, we mean that $\nabla^i(A^{(i)} \nabla^i)=\nabla^{q_i\cdots q_1}(A_{p_1\cdots p_i,q_1\cdots q_i}\nabla^{p_1\cdots p_i})$. The model for equations like \eqref{eq:ua} is
\begin{equation}\label{eq:lim}
\Delta_\xi^k U=|U|^{\crit-2}U\hbox{ in }\rn,
\end{equation}
where $\xi$ is the Euclidean metric on $\rn$. The positive solutions $U\in C^{2k}(\rn)$ to \eqref{eq:lim} are exactly (see Wei-Xu \cite{weixu}):
\begin{equation*}
X\mapsto U_{\mu, x_0}(X):=a_{n,k}\left(\frac{\mu }{\mu^2 +|X-X_0|^2}\right)^{\frac{n-2k}{2}}\hbox{ for }X\in\rn,
\end{equation*}
where $\mu >0$ and $X_0\in\rn$ are parameters and $a_{n,k}:=\left(\Pi_{j=-k}^{k-1}(n+2j)\right)^{\frac{n-2k}{4k}}$.

\smallskip\noindent Due to the critical exponent $\crit$, solutions to \eqref{eq:ua} may blow-up, that is $\max_M|\ua|\to\infty$ as $\alpha\to +\infty$. There are several ways to describe the blow-up. In the spirit of Lions \cite{PL2}, the family of measures $(|\ua|^{\crit}\, dv_g)_\alpha$ converges to a sum of Dirac masses and a density, see Mazumdar \cite{mazumdar:jde}. In the spirit of Struwe \cite{struwe:1984}, in the associated Sobolev space, $(\ua)_\alpha$ is a weak limit plus a sum of "peaks" modeled on solutions to \eqref{eq:lim}, see Mazumdar \cite{mazumdar:cpaa}. In the present paper, we improve these descriptions and prove a pointwise control by a sum of peaks modeled on the $U_{\mu,x_0}$'s:

\begin{theorem}\label{th:estim-co:intro} Let $(M,g)$ be a compact Riemannian manifold of dimension $n$ without boundary and let $k\in\nn$ be such that $2\leq 2k<n$. We consider a family of operators $(P_\alpha)_{\alpha>0}\to P_\infty$ of type (SCC) (see Definition \ref{def:scc} below) and a family of functions $(\ua)_\alpha\in C^{2k}(M)$ such that \eqref{eq:ua} holds for all $\alpha>0$. We assume that $\lim_{\alpha\to +\infty}\max_M|\ua|=+\infty$ and that there exists $\Lambda>0$ such that $\Vert \ua\Vert_{\crit}\leq \Lambda$ for all $\alpha>0$. Then there exist $u_\infty\in C^{2k}(M)$, $N\geq 1$, there exist $(z_{\alpha, 1})_\alpha,...,(z_{\alpha, N})_\alpha\in M$ such that $\lim_{\alpha\to +\infty}|\ua(\zai)|=+\infty$ for all $i=1,...,N$, and setting $\mai:=|\ua(\zai)|^{-\frac{2}{n-2k}}$, we have that
\begin{equation}
|\ua(x)|\leq C\Vert u_\infty \Vert_\infty+C\sum_{i=1}^N\left(\frac{\mai }{\mai^{2 }+d_g(x,\zai)^{2 }}\right)^{\frac{n-2k}{2}}\label{est:opti} 
\end{equation}
for all $x\in M$ and $\alpha\to +\infty$. Moreover, the $(z_{\alpha, i})_\alpha$'s are such that:
$$\lim_{\alpha\to +\infty}u_\alpha=u_\infty\hbox{ in }C^{2k}_{loc}\left(M-\left\{\lim_{\alpha\to +\infty}\zai,\, i=1,...,N\right\}\right);$$
for all $i,j\in \{1,...,N\}$ such that $i\neq j$, then
\begin{equation*}
\hbox{ either }\left\{ \frac{d_g(\zai,\zaj)}{\mai}\to+\infty\right\}\hbox{ or }\left\{ \frac{d_g(\zai,\zaj)}{\mai}\to c_{i,j}>0\hbox{ and }\maj=o(\mai)\right\}
\end{equation*}
as $\alpha\to + \infty$ and there exists $U_i \in C^{2k}(\rn)-\{0\}$ solution to \eqref{eq:lim} such that
\begin{equation*}
\tuai :=\mai^{\frac{n-2k}{2}}\ua(\hbox{exp}_{\zai}(\mai \cdot))\to U_i\hbox{ in }C^{2k}_{loc}(\rn-S_i)
\end{equation*}
where 
\begin{equation*}
S_i:=\left\{\lim_{\alpha\to \infty}\frac{\hbox{exp}_{\zai}^{-1}(\zaj)}{\mai}/\,j\in\{1,...,N\}-\{i\}\hbox{ such that }d_g(\zai,\zaj)=O(\mai)\right\}.
\end{equation*}
Finally, $U_i\in D_k^2(\rn)$, the completion of $C^\infty_c(\rn)$ for the norm $  \Vert \cdot\Vert_{D_k^2}:=\Vert \Delta_\xi^\frac{k}{2}\cdot\Vert_2$ (when $k=2l+1$ is odd, we  set $(\Delta_\xi^{\frac{k}{2}}u)^2=|\nabla \Delta_\xi^l u|^2$).
\end{theorem}
The sum on the right-hand-side of \eqref{est:opti} correspond to the sum of peaks of the Struwe decomposition, which correspond to the Dirac masses in the Lions decomposition.

\smallskip\noindent For $k=1$, this theorem has been proved by Premoselli \cite{premoselli} for general elliptic operators, including non-coercive ones. For $k=1$ and $\ua>0$, the theorem is due to Druet-Hebey-Robert \cite{DHR} and can be adapted to sign-changing solutions (see Ghoussoub-Mazumdar-Robert \cites{gmr:memams,gmr:jde}). The result  was applied by Druet \cite{druet:jdg} to prove compactness of positive solutions to \eqref{eq:ua} far from the geometric model, and by Premoselli-V\'etois \cites{premo-vetois-jmpa,premo-vetois-aim,premo-vetois-eigen} and Premoselli-Robert \cite{PR} for sign-changing solutions. The result was also written for second-order elliptic systems by Druet-Hebey \cite{DH}. Interesting phenomena have been observed by Cheikh-Ali-Premoselli \cite{CA-P} for manifolds with boundary. See Hebey \cite{hebey.eth} for a general exposition on these methods.

\smallskip\noindent There is quite a long history of pointwise controls like in Theorem \ref{th:estim-co:intro}. The pioneer works were performed for $k=1$ and $N=1$ by Atkinson-Peletier \cite{AP}, Z.-C.Han \cite{zchan} and Hebey-Vaugon \cite{HV} for positive solutions. Still for $k=1$, when the operator $P_\alpha$ is the conformal Laplacian and the solutions are positive, Schoen \cites{schoen-montecatini,schoen1988} has introduced the method of simple blow-up points. This powerful method has been paramount to tackle the prescription of scalar curvature and to prove compactness of the solutions to the Yamabe equation: see Schoen-Zhang \cite{sz}, Chen-C.-S.Lin \cites{cs:lin:1,cs:lin:2,cs:lin:3}, Li \cites{li-95,li-96}, Li-Zhu \cite{LiZhu}, Li-Zhang \cites{LiZha3,LiZhang2005}, Druet \cite{druet:imrn}, Druet-Laurain \cite{druet:laurain}, Marques \cite{Mar} and Khuri-Marques-Schoen \cite{kms}. This technique proves isolation of the concentration points for geometric blowing-up solutions, and a bound on the $L^{\crit}-$norm. Therefore, it does not apply to more general situations where concentration points accumulate or when there is no apriori bound on the $L^{\crit}-$norm, see the examples of Robert-V\'etois \cite{RV:jdg}. 

\smallskip\noindent In the above mentioned papers, the proofs are specific to second order problems, that is $k=1$. For instance, some use the pointwise maximum principle and the Harnack inequality, which are not valid for $k>1$ in general.

\smallskip\noindent Regarding higher order, for $k=2$, partial estimates have been proved by Hebey, Wen and the author \cites{hebey:2003,HRW,HR} for specific operators and positive solutions. Li-Xiong \cite{lx} and Gong-Kim-Wei \cite{gkw}  have obtained beautiful compactness results for $5\leq n\leq 24$ (this does not hold for $n\geq 25$, see Wei-Zhao \cite{wz}) by adapting the method of simple blowup points of Khuri-Marques-Schoen \cite{kms}: the solutions are positive, the operator $P_\alpha$ is the geometric Paneitz operator and it is assumed that its Green's function is positive. As for $k=1$, an intermediate step is to prove that the blow-up points are isolated. Using the powerful method of Premoselli \cite{premoselli}, Carletti \cite{carletti} has proved a refined pointwise estimate for $k>1$ and $N=1$ when the coefficients of the operator are unbounded.

\smallskip\noindent Our objective is to develop a flexible approach that tolerates sign-changing solutions and general operators. This is natural since there are not many examples of operators with positive Green's function and for many applied problems, like the clamped plate equation in mechanics, the Green's function of the operator changes sign and so do the solutions to the problem. See Grunau-Robert \cite{gr:arma} and Grunau-Robert-Sweers \cite{grs} for qualitative estimates of the sign-change.

\smallskip\noindent The general idea here is to write the nonlinear equation \eqref{eq:ua} as the linear equation
\begin{equation}\label{model:linear}
(P_\alpha-V_\alpha)\ua=0,
\end{equation}
where $V_\alpha:=|\ua|^{\crit-2}$, and to represent $\ua$ in terms of the Green's function of the operator $P_\alpha-V_\alpha$. Although this operator has unbounded coefficients, it turns that $V_\alpha$ behaves like a small Hardy potential far from the $N$ concentration points. In the case of a sole concentration point, that is $N=1$, this was performed by the author in \cite{robert:gjms} (see Premoselli-Robert \cite{PR} for $k=1$): the most delicate part is to obtain pointwise estimates for the corresponding Green's function. In the present paper, we generalize \cite{robert:gjms} to arbitrary $N$ points. In order to tackle the potential accumulation of these points, we construct an order on the concentration points and we decompose the manifold $M$ in several subdomains on which only the maximal elements for the order have an influence. Finally, we write a representation formula on each of these domains. A large portion of the analysis is devoted to pointwise controls of the Green's function enjoying multiple Hardy potentials.

\smallskip\noindent The method that is developed in this paper can be adapted to several contexts, as long as one can write a nonlinear problem as a linear problem like \eqref{model:linear} where $V_\alpha$ is a small Hardy-type potential.

\section{Preliminary material and notations}

\begin{defi}\label{def:scc} We say that $(P_\alpha)_{\alpha>0}\to P_\infty$ is of type (SCC) if $(P_\alpha)_\alpha$ and $P_\infty$ are differential operators such that $P_\alpha:=\Delta_g^k+\sum_{i=0}^{k-1} (-1)^i\nabla^i(A^{(i)}_\alpha \nabla^i)$ for all $\alpha>0$ and $P_\infty:=\Delta_g^k+\sum_{i=0}^{k-1} (-1)^i\nabla^i(A^{(i)}_\infty \nabla^i)$ are as in \eqref{model:P} and $\theta\in (0,1)$ are such that \par
\noindent$\bullet$ for all $i=0,...,k-1$, $(A^{(i)}_\alpha)_{\alpha>0}, A^{(i)}_\infty\in C^{i,\theta}$ and $\lim_{\alpha\to \infty}A_\alpha^{(i)}=A_\infty^{(i)}$ in $C^{i,\theta}$. \par
\noindent$\bullet$  $(P_\alpha)_{\alpha}$ is uniformly coercive in the sense that there exists $c>0$ such that
$$\int_M u P_\alpha u\, dv_g=\int_M (\Delta_g^{\frac{k}{2}}u)^2\, dv_g+\sum_{i=0}^{k-1}\int_M A^{(i)}_\alpha(\nabla^iu,\nabla^i u)\, dv_g\geq c\Vert u\Vert_{H_k^2}^2$$
for all $\alpha>0$ and $u\in H_k^2(M)$, where $dv_g$ is the Riemannian element of volume and $H_k^2(M)$ is the completion of $C^\infty(M)$ for the norm $u\mapsto\Vert u\Vert_{H_k^2}:= \sum_{i=0}^{k}\Vert\nabla^iu\Vert_2$.
When $k=2l+1$ is odd, we have set $(\Delta_g^{\frac{k}{2}}u)^2=|\nabla \Delta_g^l u|^2_g$. 
\end{defi}

For any metric $g$ on a manifold $X$ and $p\in X$, we let $\hbox{exp}_p^g$ be the exponential map at $p$ with respect to the metric $g$. By assimilating isometrically the tangent space $T_pX$ to $\rn$, we will consider $\hbox{exp}_p^g:\rn\to M$. Note that the isometry depends on $p$, and we will consider a smooth family of isometries in order to have smoothness of the exponential with respect to a neighborhood of each point $p$. We will simply write $\hbox{exp}_p$ when there is no ambiguity.  In all this manuscript, $i_g(X)$ will denote the injectivity radius of $(X,g)$: note that it is positive when $X$ is compact.

\smallskip\noindent Since the exponential map is a diffeomorphism before the cut-locus, then for any $p_0\in M$, there exists a neighborhood $\Omega\subset M$ of $p_0$ and $r>0$ such that
\begin{equation}\label{lem:lip}
\frac{1}{2}|X-Y|\leq d_g(\hbox{exp}_p(X),\hbox{exp}_p(Y))\leq 2|X-Y|
\end{equation}
for all $p\in\Omega$ and $X,Y\in B_{r}(0)\subset \rn$.

\smallskip\noindent  In the sequel, $C(a,b,...)$ will denote a constant depending on $(M,g)$, $a,b,...$. The value can change from one line to another, and even in the same line. Given two sequence $(a_\alpha)_\alpha$ and $(b_\alpha)_\alpha$, we will write $a_\alpha\asymp b_\alpha$ if there existe $c_1,c_2>0$ such that $c_1 a_\alpha\leq b_\alpha\leq c_2 a_\alpha$ for all $\alpha>0$.  Even when it is not mentioned, all results in this paper are up to the extraction of sub-family $\alpha\to +\infty$.

\smallskip\noindent In all this manuscript, by "Elliptic regularity theory", we will refer to the general article \cite{ADN} by Agmon-Douglis-Nirenberg. For the convenience of the readers, we will also refer to the precise statements in \cite{robert:gjms} that are extracted from \cite{ADN}.

\section{Exhaustion of the concentration points}
We let $(P_\alpha)_{\alpha>0}\to P_\infty$ be of type (SCC) and a family $(\ua)_{\alpha>0}\in C^{2k}(M)$ be such that
\begin{equation}\label{eq:ua:pf}
P_\alpha u_\alpha=|u_\alpha|^{\crit-2}u_\alpha\hbox{ in }M,
\end{equation}
\begin{equation}\label{hyp:unbnd}
\lim_{\alpha\to +\infty}\max_M|\ua|=+\infty.
\end{equation}
and \begin{equation}\label{hyp:bnd:nrj:pf}
\Vert\ua\Vert_{\crit} \leq \Lambda \hbox{ for all }\alpha>0.
\end{equation} 
Sobolev's embedding theorem for compact manifolds yields that $H_k^2(M)\hookrightarrow L^{\crit}(M)$ continuously, so that there exists $C_S(k)>0$ such that 
\begin{equation}\label{sobo:ineq:M}
\Vert u\Vert_{\crit}\leq C_S(k)\Vert u\Vert_{H_k^2}\hbox{ for all }u\in H_k^2(M).
\end{equation}
Multiplying \eqref{eq:ua:pf} by $\ua$, integrating by parts and using the uniform coercivity of $(P_\alpha)$ and the Sobolev inequality above, we get a constant $C(\Lambda)>0$ independent of $\alpha$  such that
\begin{equation}\label{hyp:bnd:nrj:pf:bis}
\Vert \ua\Vert_{H_k^2(M)}\leq C(\Lambda)\hbox{ for all }\alpha>0.
\end{equation}
The objective here is to saturate the number of concentration points: this is the object of Theorem \ref{th:4.1} at the end of the section. This section is devoted to setting up some notations and proving the theorem through intermediate steps. 
\begin{lemma}\label{lem:1} Let $(\za)_\alpha\in M$ be such that $\lim_{\alpha\to \infty}|\ua(\za)|=+\infty$. We set $\na:=|\ua(\za)|^{-\frac{2}{n-2k}}$ for all $\alpha>0$ and we fix $\Omega\subset\rn$. Assume that for all $\omega\subset\subset \Omega$, there exists $C(\omega)>0$ such that
\begin{equation}\label{assump:01}
|\ua(\hbox{exp}_{\za}(\na x))|\leq C(\omega)|\ua(\za)|\hbox{ for all }x\in\omega\hbox{ and }\alpha>0.
\end{equation}
Then there exists $V\in C^{2k}(\Omega)$ such that $\Delta_\xi^k V=|V|^{\crit-2}V$ in $\Omega$ and
\begin{equation*}
\lim_{\alpha\to \infty}\na^{\frac{n-2k}{2}}\ua(\hbox{exp}_{\za}(\na x))=V(x)\hbox{ for all }x\in \Omega,
\end{equation*}
and this convergence holds in $C^{2k}_{loc}(\Omega)$. Moreover, $|\nabla^iV|\in L^{\frac{2n}{n-2(k-i)}}(\Omega)$ for all $i=0,...,k$.
\end{lemma}
\begin{proof} Let us set
$$\hat{u}_\alpha(x):=\na^{\frac{n-2k}{2}}\ua(\hbox{exp}_{z_\alpha}(\na x))\hbox{ for all }x\in \Omega.$$
Equation \eqref{eq:ua:pf} rewrites
\begin{equation}\label{eq:ua:prop:hat}
\Delta_{\hat{g}_\alpha}^k\hat{u}_\alpha+\sum_{j=0}^{2k-2}\na^{2k-j}\hat{B}_\alpha^j\star\nabla^j\hat{u}_\alpha=|\hat{u}_\alpha|^{\crit-2 }\hat{u}_\alpha\hbox{ in }\Omega
\end{equation}
where $\hat{g}_\alpha:=(\hbox{exp}_{z_\alpha}g)(\na\cdot)$  and  $\hbox{exp}_{p }^\star g$ is the pull-back metric of $g$ and for all $j=0,...,2k-2$, $(\hat{B}_\alpha^j)_\alpha$ is a family of $(j,0)-$tensors such that there exists $C_j>0$ such that $\Vert \hat{B}_\alpha^j\Vert_{C^{0,\theta}}\leq C_j$ for all $\alpha>0$. Let us fix $\omega\subset\subset \Omega$. It follows from the assumption \eqref{assump:01} that $|\hat{u}_\alpha(x)|\leq C(\omega)$ for all $x\in\omega$ and $\alpha>0$. It then follows from elliptic regularity (\cite{ADN} or Theorems D.2 and D.3 of \cite{robert:gjms}) that there exists $V\in C^{2k}(\Omega)$ such that $\lim_{\alpha\to \infty}\hat{u}_\alpha=V$ in $C^{2k}_{loc}(\Omega)$. Note that since $\hbox{exp}_{z_\alpha}$ is a normal chart at $z_\alpha$, and therefore, we get that $\lim_{\alpha\to\infty}\hat{g}_\alpha=\xi$ in $C^{2k}_{loc}(\rn)$. Passing to the limit $\alpha\to +\infty$ in \eqref{eq:ua:prop:hat} yields $\Delta^k_\xi V=|V|^{\crit-2}V$ in $\Omega$. We fix $i\in \{0,...,2k\}$. It follows from Sobolev's embedding theorem that $H_{k}^2(M)\hookrightarrow H_i^{\frac{2n}{n-2(k-i)}}(M)$ continuously, so there exists $C_i(M)>0$ such that  
$$\left(\int_M |\nabla^i_g \ua|_g^{\frac{2n}{n-2(k-i)}}\, dv_g\right)^{\frac{n-2(k-i)}{2n}}\leq C_i(M)\Vert \ua\Vert_{H_k^2}\hbox{ for all }\alpha>0.$$
With a change of variable, we get that 
$$\int_{\hbox{exp}_{\za}(\na \omega)}|\nabla^i_g\ua|_g^{\frac{2n}{n-2(k-i)}}\, dv_g= \int_{\omega}|\nabla^i_{\hat{g}_\alpha}\hat{u}_\alpha|_{\hat{g}_\alpha}^{\frac{2n}{n-2(k-i)}}\, dv_{\hat{g}_\alpha}= \int_{\omega}|\nabla^i V|_\xi^{\frac{2n}{n-2(k-i)}}\, dv_\xi+o(1).$$
These two identities and  \eqref{hyp:bnd:nrj:pf:bis} then yield $\int_{\omega}|\nabla^i V|_\xi^{\frac{2n}{n-2(k-i)}}\, dv_\xi\leq (C_i(M)C(\Lambda))^{\frac{2n}{n-2(k-i)}}$
for all $\omega\subset\subset \Omega$. Therefore $|\nabla^i V|\in L^{\frac{2n}{n-2(k-i)}}(\Omega)$.\end{proof}
The Euclidean Sobolev embedding yields $K(n,k)>0$ such that
\begin{equation*}\label{sobo:ineq:rn}
\left(\int_{\rn}|u|^{\crit}\, dx\right)^\frac{2}{\crit}\leq K(n,k)\int_{\rn}(\Delta_\xi^{\frac{k}{2}}u)^2\, dx\hbox{ for all }u\in D_k^2(\rn).
\end{equation*}
We then get that
\begin{lemma}\label{lem:2} Given a finite set $S\subset \rn$, possibly empty, let $V\in C^{2k}(\rn-S)$ be such that $\Delta_\xi^kV=|V|^{\crit-2}V$ in $\rn-S$ and  $|\nabla^iV|\in L^{\frac{2n}{n-2(k-i)}}(\rn-S)$ for all $i=0,...,k$. Then $V$ extends continuously to $\rn$ and $V\in D_k^2(\rn)\cap C^{2k}(\rn)$ satisfies $\Delta_\xi^kV=|V|^{\crit-2}V$ in the entire space $\rn$. Moreover, if $V\not\equiv 0$, then
\begin{equation}\label{min:nrj:U}
\int_{\rn}|V|^{\crit}\, dv_\xi\geq \frac{1}{K(n,k)^{\frac{n}{2k}}}.
\end{equation}
\end{lemma}
\begin{proof} Let us write $S=\{p_1,..,p_N\}$ where $p_1,...,p_N\in\rn$ are distinct and let $r_0, R_0>0$ be such that $|p_i-p_j|>r_0$ for all $i\neq j$ and $|p_i|<R_0$ for all $i$. Let $\eta\in C^\infty(\rr)$ be such that $\eta(t)=0$ for $t\leq 1$ and $\eta(t)=1$ for $t\geq 2$. For $0<\eps<r_0/6$ and $R>2R_0+r_0$, we define
$$\eta_{\eps,R}(x):=\left\{\begin{array}{cc}
\eta\left(\frac{|x-p_j|}{\eps}\right)&\hbox{ if }x\in B_{3\eps}(p_j)\hbox{ for some }j=1,...,N\\
1-\eta\left(\frac{|x|}{R}\right)&\hbox{ if }x\in \rn -  B_{R/2}(0)\\
1&\hbox{ otherwise.}
\end{array}\right.$$
As one checks,  $\eta_{\eps,R}V\in C^\infty_c(\rn)$ and $(\eta_{\eps,R}V)_{\eps,R}$ is a Cauchy sequence in $D_k^2(\rn)$, and taking the pointwise limit, we then get that $V\in D_k^2(\rn)$. We then get that $V$ is a weak solution of $\Delta_\xi^kV=|V|^{\crit-2}V$ in $D_k^2(\rn)$. Then $V\in C^{2k}(\rn)$ is a classical solution, see Van der Vorst \cite{vdv} or Mazumdar \cite{mazumdar:jde} for a version in the Riemannian setting. It follows form \eqref{sobo:ineq:rn} that
\begin{equation*}
\frac{1}{K(n,k)}\leq \frac{\int_{\rn}(\Delta_\xi^{k/2}V)^2\, dv_\xi}{\left(\int_{\rn}|V|^{\crit}\, dv_\xi\right)^{\frac{2}{\crit}}}=\frac{\int_{\rn}|V|^{\crit}\, dv_\xi}{\left(\int_{\rn}|V|^{\crit}\, dv_\xi\right)^{\frac{2}{\crit}}}=\left(\int_{\rn}|V|^{\crit}\, dv_\xi\right)^{\frac{2k}{n}}\end{equation*}
and then \eqref{min:nrj:U} holds.
\end{proof}
\noindent{\bf Step 1: the first concentration point.} 
For all $\alpha>0$, we let $\xaun\in M$ be such that $|\ua(\xaun)|=\max_M |\ua|$ and we set $\maun:=|\ua(\xaun)|^{-\frac{2}{n-2k}}$. It follows from \eqref{hyp:unbnd} that $\lim_{\alpha\to +\infty}\maun=0$. We apply Lemmae \ref{lem:1} and \ref{lem:2} with $\Omega=\rn$, $C(\omega)=1$ for all $\omega\subset\subset\rn$, $\za:=\xaun$ and $S=\emptyset$. We then get $U_1\in D_k^2(\rn)\cap C^{2k}(\rn)$ such that $\Delta^k_\xi U_1=|U_1|^{\crit-2}U_1$ weakly in $D_k^2(\rn)$, strongly in $C^{2k}$ and 
\begin{equation*}
\lim_{\alpha\to +\infty}\tuaun=U_1\hbox{ in }C^{2k}_{loc}(\rn)
\end{equation*}
where $\tuaun(x):=\maun^{\frac{n-2k}{2}}\ua(\hbox{exp}_{\xaun}(\maun x))$ for all $x\in B_{i_g(M)/\maun}(0)\subset\rn$. Since $|\tuaun(0)|=1$, we get that $|U_1(0)|=1$, and therefore $U_1\not\equiv 0$ so that \eqref{min:nrj:U} holds with $V:\equiv U_1$. With a change of variable, we then get that
$$\int_{B_{R\maun}(\xaun)}|\ua|^{\crit}\, dv_g=\int_{B_R(0)}|\tuaun|^{\crit}\, dv_{\tilde{g}_{\alpha,1}}$$
where $\tilde{g}_{\alpha,1}:=\hbox{exp}_{\xaun}^\star g(\maun\cdot)$. Therefore, since $\hbox{exp}_{\xaun}$ is a normal chart at $\xaun$, we get that $\lim_{\alpha\to\infty}\tilde{g}_{\alpha,1}=\xi$ in $C^{2k}_{loc}(\rn)$. Then, using \eqref{min:nrj:U} we get that
\begin{equation}\label{ineq:21}
\lim_{R\to +\infty}\lim_{\alpha\to +\infty}\int_{B_{R\maun}(\xaun)}|\ua|^{\crit}\, dv_g=\int_{\rn}|U_1|^{\crit}\, dv_\xi\geq \frac{1}{K(n,k)^{\frac{n}{2k}}}.
\end{equation}
{\bf Step 2: A first family of concentration points.} 
\begin{defi} Given $K\geq 1$, we say that $(H_{K})$ holds if for all $i=1,..,K$ there exists $(\xai)_\alpha\in M$ and there exists $U_i\in D_k^2(\rn)\cap C^{2k}(\rn)-\{0\}$ such that:
\begin{itemize}
\item $\mai:=|\ua(\xai)|^{\frac{-2}{n-2k}}\to 0$ as  $\alpha\to +\infty$;
\item For all $x\in\rn$,
\begin{equation}\label{cv:tuai}
\tuai :=\mai^{\frac{n-2k}{2}}\ua(\hbox{exp}_{\xai}(\mai \cdot))\to U_i\hbox{ in }C^{2k}_{loc}(\rn)
\end{equation}
\item We have that $\Delta_\xi^k U_i=|U_i|^{\crit-2}U_i$ strongly in $\rn$ and weakly in $D_k^2(\rn)$ and
\begin{equation*}
\int_{\rn}|U_i|^{\crit}\, dv_\xi\geq \frac{1}{K(n,k)^{\frac{n}{2k}}}.
\end{equation*}
\item For all $i,j\in \{1,...,K\}$, $i\neq j$, we have that
\begin{equation}\label{lim:d:infty}
\lim_{\alpha\to +\infty}\frac{d(\xai,\xaj)}{\mai}=+\infty.
\end{equation}
\end{itemize}
\end{defi}
It follows from Step 1 that $(H_1)$ holds. 
\begin{proposition}\label{prop:rec:1} Assume that $(H_K)$ holds for some $K\geq 1$. Assume that
\begin{equation}\label{hyp:weak:1}
\limsup_{\alpha\to +\infty}\sup_{x\in M}R_{\alpha, K}(x)^{\frac{n-2k}{2}}|\ua(x)|=+\infty,
\end{equation}
where $R_{\alpha, K}(x):=\min_{i=1,...,K}d_g(x,\xai)$. Then $(H_{K+1})$ holds.
\end{proposition}
\begin{proof} We keep the same notations as in the definition of $(H_K)$. We set
$$w_\alpha(x):=R_{\alpha, K}(x)^{\frac{n-2k}{2}}|\ua(x)|\hbox{ for all }x\in M\hbox{ and }\alpha>0.$$
It follows from \eqref{hyp:weak:1} that, up to extraction, $\lim_{\alpha\to +\infty }\sup_Mw_\alpha=+\infty$. We let $(x_{\alpha, K+1})_\alpha\in M$ be such that $w_\alpha(x_{\alpha, K+1})=\sup_Mw_\alpha\to +\infty$ as $\alpha\to +\infty$. We then get that $R_{\alpha, K}(x_{\alpha, K+1})^{\frac{n-2k}{2}}|\ua(x_{\alpha, K+1})|\to +\infty$. Setting $\mu_{\alpha, K+1}:=|\ua(x_{\alpha, K+1})|^{-\frac{2}{n-2k}}$, we get that
\begin{equation}\label{lim:N+1:1}
\lim_{\alpha\to +\infty} \mu_{\alpha, K+1}=0\hbox{ and }\lim_{\alpha\to +\infty}\frac{d_g(x_{\alpha, K+1},\xai)}{\mu_{\alpha, K+1}}=+\infty\hbox{ for all }i=1,...,K.
\end{equation}
Assume that there exists $i\in \{1,..,K\}$ such that $d_g(x_{\alpha, K+1},\xai)=O(\mai)$ as $\alpha\to +\infty$. It then follows from \eqref{cv:tuai} that $w_\alpha(x_{\alpha, K+1})=O(1)$ as $\alpha\to +\infty$, contradicting the definition of $x_{\alpha, K+1}$. Therefore,  for all $i=1,...,K$, we have that
\begin{equation*}
\lim_{\alpha\to +\infty}\frac{d_g(x_{\alpha, K+1},\xai)}{\mai}=+\infty.
\end{equation*}

\medskip\noindent We claim that for all $R>0$, for $\alpha>0$ large enough, we have that
\begin{equation}\label{ineq:28}
|\ua(\hbox{exp}_{x_{\alpha, K+1}}(\mu_{\alpha, K+1} x))|\leq 2^{\frac{n-2k}{2}} |\ua(x_{\alpha, K+1})|\hbox{ for all }x\in B_R(0)\subset \rn.
\end{equation}
We prove the claim. For $x\in B_R(0)$ and $i=1,...,K$, we have that
\begin{eqnarray*}
d_g(\xai, \hbox{exp}_{x_{\alpha, K+1}}(\mu_{\alpha, K+1} x))&\geq& d_g(\xai,  x_{\alpha, K+1})-\mu_{\alpha, K+1} |x|\\
&\geq& d_g(\xai,  x_{\alpha, K+1})\left(1-\frac{\mu_{\alpha, K+1}}{d_g(\xai,  x_{\alpha, K+1})}R\right)\\
&\geq & R_{\alpha, K}(x_{\alpha, K+1})\left(1-\frac{\mu_{\alpha, K+1}}{d_g(\xai,  x_{\alpha, K+1})}R\right)
\end{eqnarray*}
It then follows from \eqref{lim:N+1:1} that for $\alpha>0$ large enough, 
$R_{\alpha,K} (\hbox{exp}_{x_{\alpha, K+1}}(\mu_{\alpha, K+1} x))\geq \frac{1}{2} R_{\alpha, K}(x_{\alpha, K+1})$. The definition of $w_\alpha$ then yields  \eqref{ineq:28}.

\smallskip\noindent We define $\tilde{u}_{\alpha,K+1}(x):=\mu_{\alpha, K+1}^{\frac{n-2k}{2}}\ua(\hbox{exp}_{x_{\alpha, K+1}}(\mu_{\alpha, K+1} x))$ for all $x\in B_{\mu_{\alpha, K+1}^{-1}i_g(M)}(0)$. It follows from \eqref{ineq:28}, \eqref{lim:N+1:1}, Lemmae \ref{lem:1} and \ref{lem:2} that there exists $U_{K+1}\in D_k^2(\rn)\cap C^{2k}(\rn)$ such that $\lim_{\alpha\to +\infty}\tilde{u}_{\alpha,K+1}=U_{K+1}$ in $C^{2k}_{loc}(\rn)$, where we have that $\Delta_\xi^kU_{K+1}=|U_{K+1}|^{\crit-2}U_{K+1}$ strongly in $\rn$ and weakly in $D_k^2(\rn)$. Using that $|\tilde{u}_{\alpha,K+1}(0)|=1$, we then get that $|U_{K+1}(0)|=1$, so $U_{K+1}\not\equiv 0$ and so that \eqref{min:nrj:U} holds with $V:= U_{K+1}$. All these results prove that $(H_{K+1})$ holds.\end{proof}

\begin{proposition}\label{prop:3.2} There exists $K\geq 1$ such that $(H_K)$ holds and 
\begin{equation}\label{hyp:weak:2}
 R_{\alpha, K}(x)^{\frac{n-2k}{2}}|\ua(x)|\leq C\hbox{ for all }x\in M\hbox{ and }\alpha>0
\end{equation}
 where $R_{\alpha, K}(x):=\min_{i=1,...,K}d_g(x,\xai)$ for all $x\in M$ and $\alpha>0$.
\end{proposition}
\begin{proof} Let $K\geq 1$ be such that $(H_K)$ holds. Given $R>0$, it follows from \eqref{lim:d:infty} that for $\alpha>0$ large enough, we have that $B_{R\mai}(\xai)\cap B_{R\maj}(\xaj)=\emptyset$ for all $i\neq j\leq K$. As in the proof of \eqref{ineq:21}, we then get that
\begin{eqnarray*}
\int_M|\ua|^{\crit}\, dv_g&\geq &\int_{\bigcup_{i=1}^K B_{R\mai}(\xai)}|\ua|^{\crit}\, dv_g=\sum_{i=1}^K\int_{  B_{R\mai}(\xai)}|\ua|^{\crit}\, dv_g\\
&\geq & \frac{K}{K(n,k)^{\frac{n}{2k}}}-\eps(R)+o(1)
\end{eqnarray*}
as $\alpha\to +\infty$ where $\lim_{R\to +\infty}\eps(R)=0$. With \eqref{hyp:bnd:nrj:pf}, we then get that $K\leq \Lambda^{\crit} K(n,k)^{\frac{n}{2k}}$. Therefore, since in addition $(H_1)$ holds, we let $K\geq 1$ be the maximal integer such that $(H_K)$ holds. Since $(H_{K+1})$ does not hold,  it follows from Proposition \ref{prop:rec:1} that \eqref{hyp:weak:2} holds. This proves the proposition.\end{proof}
A straightforward corollary is the following:
\begin{lemma}\label{prop:lim:out} Let $u_\infty\in H_k^2(M)$ be the weak limit of $(\ua)_\alpha$ as $\alpha\to +\infty$. Then $u_\infty\in C^{2k}(M)$ is a strong solution to $P_\infty u_\infty=|u_\infty|^{\crit-2}u_\infty$ in $M$ and 
\begin{equation*}
\lim_{\alpha\to\infty}u_\alpha=u_\infty\hbox{ in }C^{2k}_{loc}\left(M-\left\{\lim_{\alpha\to \infty}\xai/\,i=1,...,K\right\}\right).
\end{equation*}
\end{lemma}
\begin{proof} The existence of the weak limit follows from the boundedness of $(\ua)_\alpha$ in $H_k^2(M)$. Passing to the limit in \eqref{eq:ua:pf} yields that $u_\infty\in H_k^2(M)$ is a weak solution to $P_\infty u_\infty=|u_\infty|^{\crit-2}u_\infty$ in $M$. Then, see Mazumdar \cite{mazumdar:jde}, $u_\infty\in C^{2k}(M)$ is a strong solution. We fix $\delta>0$ and we set $M_\delta:=M -  \bigcup_{i=1}^K\overline{B}_\delta(\lim_{\alpha\to \infty}\xai)$. It follows from \eqref{hyp:weak:2} that there exists $C(\delta)>0$ such that $|u_\alpha(x)|\leq C(\delta)$ for all $x\in M_\delta$ and $\alpha>0$. It then follows from \eqref{eq:ua} and elliptic theory  that $(\ua)_\alpha$ has a strong limit in $C^{2k}_{loc}\left(M-\{\lim_{\alpha\to \infty}\xai/\,i=1,...,K\}\right)$. By uniqueness, this limit is $u_\infty$. \end{proof}
\noindent{\bf Step 3: A second family of concentration points}
\begin{lemma}\label{lem:3} Let $(\za)_\alpha\in M$ be such that
\begin{equation}\label{eq:67}
\lim_{\alpha\to +\infty}R_{\alpha, K}(\za)^{\frac{n-2k}{2}}|\ua(\za)-u_\infty(\za)|=c_0>0.
\end{equation}
Then $\lim_{\alpha\to \infty}|\ua(\za)|=+\infty$. We set $\na:=|\ua(\za)|^{-\frac{2}{n-2k}}$ and we define
$$S_0:=\left\{\lim_{\alpha\to +\infty}\frac{\hbox{exp}_{\za}^{-1}(\xai)}{\na}/\,i\in I_0\right\}$$
where $I_0:=\{ i\in\{1,...,K\}\hbox{ such that }d_g(\za,\xai)=O(\na)\}$. Then there exists $V\in D_k^2(\rn)\cap C^{2k}(\rn)$ that is a nonzero solution to $\Delta_\xi^kV=|V|^{\crit-2}V$ strongly in $\rn$ and weakly in $D_k^2(\rn)$ and such that
\begin{equation*}
\lim_{\alpha\to\infty}\na^{\frac{n-2k}{2}}\ua(\hbox{exp}_{\za}(\na \cdot))=V\hbox{ in }C^{2k}_{loc}(\rn-S_0).
\end{equation*}
$$\hbox{Moreover, }\lim_{R\to +\infty}\lim_{\alpha\to +\infty}\int_{B_{R\na}(\za)-\bigcup_{i\in I_0}B_{R^{-1}\na}(\xai)}|\ua|^{\crit}\, dv_g\geq \frac{1}{K(n,k)^{\frac{n}{2k}}}.$$
\end{lemma}
\begin{proof} It follows from Lemma \ref{prop:lim:out} and \eqref{eq:67} that  $\lim_{\alpha\to +\infty}R_{\alpha, K}(\za)=0$. In addition, since $u_\infty$ is bounded on $M$, we get  $\lim_{\alpha\to +\infty}R_{\alpha, K}(\za)^{\frac{n-2k}{2}}|\ua(\za)|=c_0$ due to  \eqref{eq:67}. The definition of $\na$ yields $\lim_{\alpha\to +\infty}\na=0$ and 
\begin{equation}\label{eq:68}
\lim_{\alpha\to +\infty}\frac{R_{\alpha, K}(\za)}{\na}=c_0^{\frac{2}{n-2k}}.
\end{equation}
For any $i\in I_0$, we let $\theta_{\alpha, i}\in \rn$ be such that $\xai=\hbox{exp}_{\za}(\na \theta_{\alpha, i})$ for all $\alpha>0$. With the definition of $I_0$, there exists $C>0$ such that $|\theta_{\alpha,i}|\leq C$ for all $i\in I_0$ and $\alpha>0$. We then set $\theta_{\infty,i}:=\lim_{\alpha\to +\infty}\theta_{\alpha,i}$ up to extraction, so that $S_0=\{\theta_{\infty,i}/\, i\in I_0\}$. We fix $R>0$, and we take $x\in B_R(0)$. We set $i\in\{1,...,K\} -  I_0$. Then $\lim_{\alpha\to +\infty}\na^{-1}d_g(\za,\xai)=+\infty$. Therefore,
$$d_g(\xai, \hbox{exp}_{\za}(\na x))\geq d_g(\xai,\za)-R\na\geq 2\na\hbox{ as }\alpha\to +\infty.$$
We now choose $i\in I_0$. It follows from \eqref{lem:lip} that
$$d_g(\xai, \hbox{exp}_{\za}(\na x))= d_g(\hbox{exp}_{\za}(\na \theta_{\alpha, i}), \hbox{exp}_{\za}(\na x))\geq \frac{1}{2}\na |x-\theta_{\alpha,i}|.$$
Therefore, for all $R>0$, there exists $c_1>0$ such that $R_{\alpha, K}(\hbox{exp}_{\za}(\na x))\geq c_1\na\min_{i\in I_0}|x-\theta_{\alpha,i}|$ for all $x\in B_R(0)$ and $\alpha>0$ large enough depending only on $R$. It then follows from \eqref{hyp:weak:2} that
$$c_1^{\frac{n-2k}{2}}\min_{i\in I_0}|x-\theta_{\alpha,i}|^{\frac{n-2k}{2}}\na^{\frac{n-2k}{2}}|\ua(\hbox{exp}_{\za}(\na x))|\leq C\hbox{ for all }x\in B_R(0).$$
Therefore, for all $\omega\subset\subset \rn-S_0=\rn-\{\theta_{\infty,i}/\, i\in I_0\}$, there exists $C(\omega)>0$ such that $|\ua(\hbox{exp}_{\za}(\na x))|\leq C(\omega)|\ua(\za)|$ for all $x\in\omega$. We set 
$$\hat{u}_\alpha(x):=\na^{\frac{n-2k}{2}}\ua(\hbox{exp}_{z_\alpha}(\na x))\hbox{ for all }x\in B_{\na^{-1}i_g(M)}(0).$$
It then follows from Lemmae \ref{lem:1} and \ref{lem:2} that there exists $V\in D_k^2(\rn)\cap C^{2k}(\rn)$ such that $\lim_{\alpha\to +\infty}\hat{u}_\alpha=V$ in $C^{2k}_{loc}(\rn-S_0)$. It follows from \eqref{eq:68} that $|\theta_{\alpha,i}|\geq c_0^{\frac{2}{n-2k}}+o(1)$ for all $\alpha>0$ and $i\in I_0$, so that $\theta_{\infty,i}\neq 0$, and then $0\in\rn-S_0$. We then get that $|V(0)|=\lim_{\alpha\to +\infty}|\hat{u}_\alpha(0)|=1$, and then $V\not\equiv 0$ and \eqref{min:nrj:U} holds. We fix $R>0$. Setting $\hat{g}_\alpha:=(\hbox{exp}_{\za}^\star g)(\na\cdot)$, a change of variable yields
\begin{eqnarray*}
&&\int_{B_{R\na}(\za)-\bigcup_{i\in I_0}B_{R^{-1}\na}(\xai)}|\ua|^{\crit}\, dv_g\\
&&=\int_{B_{R}(0)-\bigcup_{i\in I_0}\na^{-1}\hbox{exp}_{\za}^{-1}(B_{R^{-1}\na}(\xai))}|\hat{u}_\alpha|^{\crit}\, dv_{\hat{g}_\alpha}\\
&&\geq \int_{B_{R}(0)-\bigcup_{i\in I_0} B_{2R^{-1} }(\theta_{\alpha, i})}|\hat{u}_\alpha|^{\crit}\, dv_{\hat{g}_\alpha}
\end{eqnarray*}
Therefore, using  \eqref{min:nrj:U}, we get that
\begin{eqnarray*}
&&\lim_{R\to +\infty}\lim_{\alpha\to +\infty}\int_{B_{R\na}(\za)-\bigcup_{i\in I_0}B_{R^{-1}\na}(\xai)}|\ua|^{\crit}\, dv_g\\
&&\geq \lim_{R\to +\infty}\int_{B_{R}(0)-\bigcup_{i\in I_0} B_{2R^{-1} }(\theta_{\infty, i})}|V|^{\crit}\, dv_{\xi}\geq \int_{\rn}|V|^{\crit}\, dv_{\xi}\geq \frac{1}{K(n,k)^{\frac{n}{2k}}}.
\end{eqnarray*}
The Lemma is proved.\end{proof}

\begin{defi} We say that $(\tilde{H}_{0})$ holds if $(H_K)$ holds as in Proposition \ref{prop:3.2}. Given $L\geq 1$, we say that $(\tilde{H}_{L})$ holds if for all $p=1,...,L$ there exists $(\yap)_\alpha\in M$ such that $\nap:=|\ua(\yap)|^{\frac{-2}{n-2k}}\to 0$ as  $\alpha\to +\infty$, and, denoting
$$I_p:=\{i\in\{1,...,K\}\hbox{ such that }d_g(\xai,\yap)=O(\nap)\},$$
there exists $\eps_0>0$ such that for all $i\in\{1,...,K\}$ and $p\in\{1,...,L\}$, we have that
\begin{equation}\label{ineq:dist:xai:yap}
\lim_{\alpha\to+\infty}\frac{d_g(\yap,\xai)}{\mai}=+\infty\hbox{ and }\frac{d_g(\yap,\xai)}{\nap}\geq\eps_0.
\end{equation}
Moreover, for $p,q\in \{1,...,L\}$ such that $p\neq q$, then
either 
\begin{equation}\label{ineq:dist:yap:yaq:1}
\lim_{\alpha\to +\infty}\frac{d_g(\yap,\yaq)}{\nap}=+\infty
\end{equation} 
\begin{equation}\label{ineq:dist:xai:yap:yaq:2}
\hbox{or }\left\{\lim_{\alpha\to +\infty}\frac{d_g(\yap,\yaq)}{\nap}=c_{p,q}\in (0,+\infty)\hbox{ and }\naq=o(\nap)\right\}.
\end{equation}
In addition, for any $p\in \{1,...,L\}$, there exists $V_p\in D_k^2(\rn)\cap C^{2k}(\rn)-\{0\}$ such that $\Delta_\xi^k V_p=|V_p|^{\crit-2}V_p$ strongly in $\rn$ and weakly in $D_k^2(\rn)$ and
\begin{equation}\label{cv:tvap}
\tvap :=\nap^{\frac{n-2k}{2}}\ua(\hbox{exp}_{\yap}(\nap \cdot))\to V_p\hbox{ in }C^{2k}_{loc}(\rn-S_p)
\end{equation} 
\begin{equation}\label{def:Sp}
\hbox{where }S_p:=\left\{\lim_{\alpha\to \infty}\frac{\hbox{exp}_{\yap}^{-1}(\xai)}{\nap}/\, i\in I_p\right\}.
\end{equation}
\end{defi}
In the sequel, for $i\in \{1,...,K\}$, $p\in\{1,...,L\}$, $\alpha>0$ and $R>0$, we set
$$\Omega_{i,\alpha}(R):=B_{R\mai}(\xai), $$
$$\tilde{\Omega}_{p,\alpha}(R):= B_{R\nap}(\yap)-\bigcup_{j\in I_p}B_{R^{-1}\nap}(\xaj).$$

\begin{proposition}\label{prop:4.4} Assume that $(\tilde{H}_L)$ holds for some $L\geq 0$. Then $K+L\leq \Lambda^{\crit} K(n,k)^{\frac{n}{2k}}$.
\end{proposition}
\begin{proof} It follows from \eqref{hyp:bnd:nrj:pf} that
\begin{equation}\label{eq:36}
\int_{\bigcup_i\Omega_{i,\alpha}(R)\cup \bigcup_p\tilde{\Omega}_{p,\alpha}(R) }|\ua|^{\crit}\, dv_g\leq \int_M |\ua|^{\crit}\, dv_g\leq\Lambda^{\crit}.
\end{equation}
It follows from \eqref{cv:tuai}, \eqref{cv:tvap} and \eqref{min:nrj:U} that for any $i\in\{1,...,K\}$ and $p\in\{1,...,L\}$, we have that
\begin{equation}
\left\{\begin{array}{c}
\lim_{R\to +\infty}\lim_{\alpha\to +\infty}\int_{\Omega_{i,\alpha}(R) }|\ua|^{\crit}\, dv_g\geq \frac{1}{K(n,k)^{\frac{n}{2k}}}\\
\lim_{R\to +\infty}\lim_{\alpha\to +\infty}\int_{\tilde{\Omega}_{p,\alpha}(R)}|\ua|^{\crit}\, dv_g\geq \frac{1}{K(n,k)^{\frac{n}{2k}}}\end{array}\right\}.\label{eq:37}
\end{equation}
Therefore, the conclusion of the proposition holds provided  that the integral is neglictible on the intersection of the domains. This is what we prove now.

\smallskip\noindent  For $i,j\in \{1,...,K\}$, $i\neq j$, it follows from \eqref{lim:d:infty} that for $\alpha>0$ large enough, we have that $\Omega_{i,\alpha}(R)\cap \Omega_{j,\alpha}(R)=\emptyset$. Therefore
\begin{equation}\label{lim:int:1}
\lim_{\alpha\to \infty}\int_{\Omega_{i,\alpha}(R)\cap \Omega_{j,\alpha}(R)}|\ua|^{\crit}\, dv_g=0.
\end{equation}
\smallskip\noindent For $i\in \{1,...,K\}$ and $p\in\{1,...,L\}$. Assume that, up to to extraction,
$$\Omega_{i,\alpha}(R)\cap \tilde{\Omega}_{p,\alpha}(R)\neq \emptyset\hbox{ for }\alpha\to +\infty.$$
Therefore, there exists $(a_\alpha)_\alpha\in M$ lying in the intersection, so that $d_g(a_\alpha, \xai)\leq R\mai$, $d_g(a_\alpha, \yap)\leq R\nap$ and $d_g(a_\alpha, \xaj)\geq R^{-1}\nap$ for all $j\in I_p$. The triangle inequality yields
\begin{equation}\label{eq:34}
d_g(\xai,\xaj)\geq d_g(a_\alpha,\xaj)-d_g(a_\alpha,\xai)\geq R^{-1}\nap-R\mai\hbox{ for all }j\in I_p.
\end{equation}
Another application of the triangle inequality yields $d_g(\xai,\yap)=O(\mai+\nap)$, and then, with  \eqref{ineq:dist:xai:yap}, we get 
\begin{equation}\label{eq:35}
\mai=o(\nap)\hbox{ and }d_g(\xai,\yap)=O(\nap).
\end{equation}
We set $\Theta_\alpha\in B_{\nap^{-1}i_g(M)}(0)\subset\rn$ be such that $\xai=\hbox{exp}_{\yap}(\nap \Theta_\alpha)$ for all $\alpha>0$. It follows from \eqref{eq:35} that $|\Theta_\alpha|=O(1)$ and then there exists $\Theta_\infty\in\rn$ such that $\lim_{\alpha\to +\infty}\Theta_\alpha=\Theta_\infty$. It follows from \eqref{lem:lip} that
$$\Omega_{i,\alpha}(R)\cap \tilde{\Omega}_{p,\alpha}(R)\subset B_{R\mai}(\xai)\subset \hbox{exp}_{\yap}\left(\nap B_{2R\mai/\nap}(\Theta_\alpha)\right).$$
We claim that $\Theta_\infty\not\in S_p$, where $S_p$ is as \eqref{def:Sp}. Indeed, it follows from  \eqref{eq:34} and \eqref{eq:35} that $d_g(\xai,\xaj)\geq (2R)^{-1}\nap$ for all $j\in I_p$, and then, using \eqref{lem:lip},  and passing to the limit, we get that $|\Theta_\infty-\Theta|\geq (4R)^{-1}$ for all $\Theta\in S_p$. Therefore $\Theta_\infty\not\in S_p$ and the claim is proved.

\smallskip\noindent Taking $\tvap$ as in \eqref{cv:tvap}, with a change of variable, we get that
\begin{eqnarray*}\int_{\Omega_{i,\alpha}(R)\cap \tilde{\Omega}_{p,\alpha}(R)}|\ua|^{\crit}\, dv_g&\leq& \int_{\hbox{exp}_{\yap}\left(\nap B_{2R\mai/\nap}\Theta_\alpha)\right)}|\ua|^{\crit}\, dv_g\\
&\leq& \int_{B_{2R\mai/\nap}(\Theta_\alpha)}|\tvap|^{\crit}\, dv_{\tilde{g}_{\alpha,p}}
\end{eqnarray*}
where $\tilde{g}_{\alpha,p}:=\hbox{exp}_{\yap}^{\star}g(\nap\cdot)$. It then follows from \eqref{cv:tvap}, $\Theta_\alpha\to\Theta_\infty\in\rn-S_p$ and  $\mai=o(\nap)$ that 
\begin{equation}\label{lim:int:2}
\lim_{\alpha\to \infty}\int_{\Omega_{i,\alpha}(R)\cap \tilde{\Omega}_{p,\alpha}(R)}|\ua|^{\crit}\, dv_g=0.
\end{equation}
Note that when the domain is empty, then this result is trivial.

\medskip\noindent We now choose $p,q\in \{1,...,L\}$ such that $p\neq q$. Arguing as above, we get that
\begin{equation}\label{lim:int:3}
\lim_{\alpha\to \infty}\int_{\tilde{\Omega}_{p,\alpha}(R)\cap \tilde{\Omega}_{q,\alpha}(R)}|\ua|^{\crit}\, dv_g=0.
\end{equation}

\smallskip\noindent We now can conclude the proof of Proposition \ref{prop:4.4}. It follows from \eqref{eq:36}, \eqref{eq:37}, \eqref{lim:int:1}, \eqref{lim:int:2} and \eqref{lim:int:3} that  $K+L\leq \Lambda^{\crit} K(n,k)^{\frac{n}{2k}}$. The proposition is proved.\end{proof}

\begin{proposition} Assume that $(\tilde{H}_L)$ holds for some $L\geq 0$. Assume that
\begin{equation}\label{hyp:67}
\lim_{R\to \infty}\lim_{\alpha\to \infty}\sup_{\begin{array}{c}
x\in M-\Omega_{i,\alpha}(R)\\
x\in M- \tilde{\Omega}_{p,\alpha}(R)\end{array}}\min\{R_{\alpha,K}(x),\tilde{R}_{\alpha,L}(x)\}^{\frac{n-2k}{2}}|\ua(x)-u_\infty(x)|>0,
\end{equation}
where $\tilde{R}_{\alpha,L}(x):=\min\{d_g(x, \yap)/\, p=1,...,L\}$. Then $(\tilde{H}_{L+1})$ holds.
\end{proposition}
\begin{proof} We let $(\xai)_\alpha, (\yap)_\alpha\in M$, $i=1,...,K$ and $p=1,...,L$ such that $(\tilde{H}_L)$ holds. It follows from \eqref{hyp:67} that there exists $(y_{\alpha,L+1})_\alpha\in M$ such that 
\begin{equation}\label{hyp:lim:23}
\lim_{\alpha\to +\infty}\min\{R_{\alpha,K}(y_{\alpha,L+1}),\tilde{R}_{\alpha,L}(y_{\alpha,L+1})\}^{\frac{n-2k}{2}}|\ua(y_{\alpha,L+1})-u_\infty(y_{\alpha,L+1})|>0
\end{equation}
\begin{equation}\label{eq:49}
 \hbox{with }\lim_{\alpha\to \infty}\frac{d_g(y_{\alpha,L+1},\xai)}{\mai}=+\infty\hbox{ for all }i=1,...,K
 \end{equation}
and for any $p\in \{1,...,L\}$,
\begin{equation}\label{eq:45}
\begin{array}{c}
\lim_{\alpha\to \infty}\frac{d_g(y_{\alpha,L+1},\yap)}{\nap}=c_{p,L+1}\in (0,+\infty)\cup \{+\infty\}.
\end{array}
\end{equation}
Moreover, if $c_{p,L+1}<+\infty$, there exists $j\in I_p$ such that $d_g(y_{\alpha,L+1},\xaj)=o(\nap)$. It follows from \eqref{hyp:lim:23} and \eqref{hyp:weak:2} that there exists $\eps_0>0$ such that
\begin{equation*}
\lim_{\alpha\to +\infty} R_{\alpha,K}^{\frac{n-2k}{2}}(y_{\alpha, L+1})|\ua(y_{\alpha,L+1})-u_\infty(y_{\alpha,L+1})|=\eps_0>0.
\end{equation*}
Then Lemma \ref{lem:3} yields $\nu_{\alpha,L+1}:=|\ua(y_{\alpha, L+1})|^{-\frac{2}{n-2k}}\to 0$ as $\alpha\to +\infty$. We set
$$S_{L+1}:=\left\{\lim_{\alpha\to +\infty}\frac{\hbox{exp}_{y_{\alpha, L+1}}^{-1}(\xai)}{\nu_{\alpha,L+1}}/\,i\in I_{L+1}\right\}$$
where $I_{L+1}:=\{ i\in\{1,...,K\}\hbox{ such that }d_g(y_{\alpha, L+1},\xai)=O(\nu_{\alpha, L+1})\}$. Then there exists $V_{L+1}\in D_k^2(\rn)\cap C^{2k}(\rn)$ that is a nonzero solution to $\Delta_\xi^kV_{L+1}=|V_{L+1}|^{\crit-2}V_{L+1}$ strongly in $\rn$ and weakly in $D_k^2(\rn)$ and such that
\begin{equation*}
\lim_{\alpha\to\infty}\tilde{v}_{\alpha, L+1}=V_{L+1}\hbox{ in }C^{2k}_{loc}(\rn-S_{L+1})\hbox{ where }\tilde{v}_{\alpha, L+1}:=\nu_{\alpha, L+1}^{\frac{n-2k}{2}}\ua(\hbox{exp}_{y_{\alpha, L+1}}(\nu_{\alpha, L+1} \cdot)).
\end{equation*}
It follows from \eqref{hyp:67} that for all $i=1,...,K$ and $p=1,...,L$, then for some $\eps_1>0$,
\begin{equation}\label{eq:51}
\lim_{\alpha\to +\infty}\frac{d_g(y_{\alpha, L+1},\xai)}{\nu_{\alpha, L+1}}\geq \eps_1\hbox{ and }\lim_{\alpha\to +\infty}\frac{d_g(y_{\alpha, L+1},\yap)}{\nu_{\alpha, L+1}}\geq \eps_1 
\end{equation}
for all $\alpha>0$. Then \eqref{eq:49} and \eqref{eq:51} yield  \eqref{ineq:dist:xai:yap} for $i=1,...,K$ and $p=1,...,L+1$. We are left with proving \eqref{ineq:dist:yap:yaq:1}  and \eqref{ineq:dist:xai:yap:yaq:2}  for $p,q\in\{1,...,L+1\}$, one of them being $L+1$. Assume first that
$$\lim_{\alpha\to +\infty}\frac{d_g(y_{\alpha, L+1},\yap)}{\nu_{\alpha, L+1}}=c_{p, L+1}\in (0,+\infty).$$
Then $d_g(y_{\alpha, L+1},\yap)\asymp \nu_{\alpha, L+1}$, and then,   \eqref{eq:45} yields $\nap=O(d_g(y_{\alpha, L+1},\yap))=O( \nu_{\alpha, L+1})$. Assume that $\nap\asymp \nu_{\alpha, L+1}\asymp d_g(y_{\alpha, L+1},\yap)$, and then  there exists $j\in I_p$ such that $d_g(y_{\alpha,L+1},\xaj)=o(\nap)=o(\nu_{\alpha, L+1})$, contradicting  \eqref{eq:51}. So $\nap=o( \nu_{\alpha, L+1})$, and we are done. 

\smallskip\noindent  Assume now that
$$\lim_{\alpha\to +\infty}\frac{d_g(y_{\alpha, L+1},\yap)}{\nu_{\alpha, p}}=c_{p, L+1}\in (0,+\infty).$$
Therefore $d_g(y_{\alpha, L+1},\yap)\asymp \nu_{\alpha, p}$, and then $\nu_{\alpha, L+1}=O(d_g(y_{\alpha, L+1},\yap))=O( \nu_{\alpha, p})$. Assume that $\nu_{\alpha, L+1}\asymp \nu_{\alpha, p}\asymp d_g(y_{\alpha, L+1},\yap)$. As above, we get that $\nap=o( \nu_{\alpha, L+1})$, which is a contradiction. Therefore,  $\nu_{\alpha, L+1}=o( \nu_{\alpha, p})$, and we are done.

\smallskip\noindent As one checks, this yields $(\tilde{H}_{L+1})$. \end{proof}

We are now in position to conclude this section. 
\begin{theorem}\label{th:4.1} There exists $K\geq 1$ such that $(H_K)$ holds, there exists $L\geq 0$ such that $(\tilde{H}_L)$ holds and such that 
\begin{equation}\label{lim:exhaust}
\lim_{R\to \infty}\lim_{\alpha\to \infty}\sup_{\begin{array}{c}
x\in M-\cup_i\Omega_{i,\alpha}(R)\\
x\in M- \cup_p\tilde{\Omega}_{p,\alpha}(R)\end{array}} R_{\alpha}(x)^{\frac{n-2k}{2}}|\ua(x)-u_\infty(x)|=0,
\end{equation}
where $\tilde{R}_{\alpha}(x):=\min\{d_g(x,\xai),d_g(x, \yap)/\, i=1,...,K\hbox{ and }p=1,...,L\}$.
\end{theorem}
\begin{proof} Let $L\geq 0$ be the largest integer such that $(\tilde{H}_L)$ holds: the existence follows from Proposition \ref{prop:4.4}. Since $(\tilde{H}_{L+1})$ does not hold, we get \eqref{lim:exhaust}.\end{proof}
Setting $N:=K+L$ and letting $(z_{\alpha, 1})_\alpha$,..., $(z_{\alpha, N})_\alpha\in M$ be the $\xai$'s and $\yap$'s, a consequence of Theorem \ref{th:4.1} is that 
\begin{theorem}\label{th:exhaust:def} There exists $(z_{\alpha, 1})_\alpha,...,(z_{\alpha, N})_\alpha\in M$ such that $\lim_{\alpha\to +\infty}|\ua(\zai)|=+\infty$ for all $i=1,...,N$, and setting $\mai:=|\ua(\zai)|^{-\frac{2}{n-2k}}$, we have that for all $i,j\in \{1,...,N\}$ such that $i\neq j$, then
\begin{equation}\label{eq:53}
\hbox{ either }\left\{\lim_{\alpha\to +\infty}\frac{d_g(\zai,\zaj)}{\mai}=+\infty\right\}\hbox{ or }\left\{ d_g(\zai,\zaj)\asymp\mai\hbox{ and }\maj=o(\mai)\right\}.
\end{equation}
Moreover, for all $i\in \{1,...,N\}$, there exists $U_i\in D_k^2(\rn)\cap C^{2k}(\rn)-\{0\}$ such that $\Delta_\xi^k U_i=|U_i|^{\crit-2}U_i$ strongly in $\rn$ and weakly in $D_k^2(\rn)$ and
\begin{equation}\label{cv:tuai:bis}
\tuai :=\mai^{\frac{n-2k}{2}}\ua(\hbox{exp}_{\zai}(\mai \cdot))\to U_i\hbox{ in }C^{2k}_{loc}(\rn-S_i)
\end{equation} 
\begin{equation*}
\hbox{where }S_i:=\left\{\lim_{\alpha\to \infty}\frac{\hbox{exp}_{\zai}^{-1}(\zaj)}{\mai}/\, j\in I_i\right\}.
\end{equation*}
\begin{equation}\label{def:Ii}
\hbox{and }I_i:=\{j\in\{1,...,N\}-\{i\}\hbox{ such that }d_g(\zai,\zaj)=O(\mai)\}.
\end{equation}
Finally, 
\begin{equation}\label{lim:exhaust:bis}
\lim_{R\to \infty}\lim_{\alpha\to \infty}\sup_{x\in M-\cup_i\Omega_{i,\alpha}(R)} R_{\alpha}(x)^{\frac{n-2k}{2}}|\ua(x)-u_\infty(x)|=0,
\end{equation}
where $R_{\alpha}(x):=\min\{d_g(x,\zai)/\, i=1,...,N\}$ and for $i\leq N$, $\alpha,R>0$ we set
$$\Omega_{i,\alpha}(R):=B_{R\mai}(\zai) -\bigcup_{j\in I_i}B_{R^{-1}\mai}(\zaj).$$
\end{theorem}

\section{Strong pointwise control: proof of Theorem \ref{th:estim-co:intro}}\label{sec:strong}

\subsection{Ordering the concentration points}
For $i,j\in \{1,...,N\}$, we say that 
\begin{equation*}
i\preceq j\hbox{ if }\{\mai=O(\maj)\hbox{ and }d_g(\zai,\zaj)=O(\maj)\}.
\end{equation*}
We claim that this is an order. Reflexivity and transitivity are obvious. Assume that $i\preceq j$ and $j\preceq i$ for $i,j\leq N$. Therefore $\mai\asymp\maj$ and $d_g(\zai,\zaj)=O(\mai)$: this contradicts \eqref{eq:53} when $i\neq j$, so $i=j$. This proves the claim.

\smallskip\noindent The following remarks rely on \eqref{eq:53}: for any $i,j\in\{1,...,N\}$, we have that
\begin{eqnarray*}
&&I_i=\{j\in\{1,...,N\}\hbox{ such that }j\preceq i\hbox{ and }j\neq i\}\\
&&\{j\preceq i\hbox{ and }j\neq i\}\Rightarrow \{\maj=o(\mai)\}.
\end{eqnarray*}
Up to extraction, we assume that for all $i,j,l\in \{1,...,N\}$,
$$\lim_{\alpha\to +\infty}\frac{\mai}{\maj}\hbox{ and } \lim_{\alpha\to +\infty}\frac{d_g(\zai,\zaj)}{\mal}\hbox{ exist in }[0,+\infty].$$
Extracting again, we let $C_1,C_2>0$ be such that for all $i,j,l\in\{1,..,N\}$:
\begin{eqnarray}
&&C_1\geq \frac{d_g(\zai,\zaj)}{\mal}\hbox{ if }d_g(\zai,\zaj)=O(\mal);\label{lim:d:1}\\
&& \frac{d_g(\zai,\zaj)}{\mal}\geq \frac{1}{C_1}\hbox{ if }\mal=O(d_g(\zai,\zaj));\nonumber
\end{eqnarray}
\begin{equation}\label{lim:m:1}
C_2\geq\frac{\maj}{\mai}\hbox{ for all } i, j\in \{1,...,N\}\hbox{ such that }\maj=O(\mai).
\end{equation}
In particular, 
\begin{equation*}
\lim_{\alpha\to +\infty}\frac{\maj}{\mai}\neq 0 \; \Rightarrow \; \frac{\maj}{\mai}\geq \frac{1}{C_2}.
\end{equation*}

\begin{lemma}\label{step:34} We set $v_\alpha:=\ua-u_\infty $. There exists $(V_\alpha)_\alpha, (f_\alpha)_\alpha\in L^\infty(M)$ such that
\begin{equation}\label{eq:P:V:f}
P_\alpha v_\alpha= V_\alpha v_\alpha+f_\alpha
\end{equation}
where there exists $C>0$ such that $\Vert f_\alpha\Vert_\infty\leq C\Vert u_\infty \Vert_\infty^{\crit-1}+C\Vert u_\infty \Vert_{C^{2k}}$ and for  any $\delta>0$, there exists $R_\delta>0$ such that for all $R>R_\delta$ we have that
\begin{equation}\label{ineq:45}
R_\alpha(x)^{2k}|V_\alpha(x)|\leq \delta\hbox{ for all }x\in M-\cup_i\Omega_{i,\alpha}(R)\hbox{ for all }\alpha>0.
\end{equation}
\end{lemma}
\begin{proof} We define $v_\alpha:=\ua-u_\infty $. We  have that \eqref{eq:P:V:f} holds with
$$V_\alpha:=\frac{{\bf 1}_{|u_\infty |<\frac{1}{2}|v_\alpha|}|u_\infty +v_\alpha|^{\crit-2}(u_\infty +v_\alpha)}{v_\alpha}$$
$$\hbox{and }f_\alpha:={\bf 1}_{\frac{1}{2}|v_\alpha|\leq |u_\infty |}|u_\infty +v_\alpha|^{\crit-2}(u_\infty +v_\alpha)-P_\alpha u_\infty .$$
So $|V_\alpha|\leq C|\va|^{\crit-2}$ and $\Vert f_\alpha\Vert_\infty\leq C\Vert u_\infty \Vert_\infty^{\crit-1}+C\Vert u_\infty \Vert_{C^{2k}}$. Using Theorem \ref{th:exhaust:def}, we get  \eqref{ineq:45}.\end{proof}

\medskip\noindent We let $(\za)_\alpha\in M$ be an arbitrary family. 

\subsection{The case when $(\za)_\alpha$ remains far from the $(\zai)_\alpha$'s}Assume first that
\begin{equation}\label{choice:1}
\lim_{\alpha\to +\infty}\frac{d_g(\za,\zai)}{\mai}=+\infty\hbox{ for all }i=1,...,N.
\end{equation}
Let $J\subset \{1,...,N\}$ be the set of maximal elements  for the order $\preceq$. We claim that
\begin{equation}\label{lim:34}
\lim_{\alpha\to +\infty}\frac{d_g(\zai,\zaj)}{\mai}=+\infty\hbox{ for all }i\neq j\in J.
\end{equation}
We argue by contradiction and assume that for some $i\neq j\in J$, we have that $d_g(\zai,\zaj)=O(\mai)$. It then follows from \eqref{eq:53} that $\maj=o(\mai)$ and then $j\preceq i$, which contradicts maximality since $j\neq i$. This proves \eqref{lim:34}.

\smallskip\noindent We define 
\begin{equation}\label{def:omega:J}
\Omega_{\alpha, J}(R):=\bigcup_{i\in J}B_{R\mai}(\zai).
\end{equation}
It follows from \eqref{lim:34} that the balls involved in this definition are all disjoint. For $RC_2>2C_1$,  $R>2C_1$, we claim that
\begin{equation}\left\{\begin{array}{c}
R_\alpha(x)\geq \frac{1}{2}R_{\alpha, J}(x):=\min\{d_g(x,\zai)/\, i\in J\}\\
\hbox{ and }x\in M-\bigcup_{i}\Omega_{\alpha, i}(R)\end{array}\right\}\label{res:89}
\end{equation}
for all $x\in M-\Omega_{\alpha, J}(2R)$.

\smallskip\noindent We prove the claim. We fix $x\in M-\Omega_{\alpha, J}(2R)$ and we fix $i\in\{1,...,N\}$. Let $j\in J$ be a maximal element such that $i\preceq j$. Then $\mai=O(\maj)$ and $d_g(\zai,\zaj)=O(\maj)$. It then follows from \eqref{lim:d:1} and \eqref{lim:m:1} that $\mai\leq C_2\maj$ and $d_g(\zai,\zaj)\leq C_1\maj$. Since $d_g(x,\zaj)>2RC_2\maj$, we have that
\begin{eqnarray*}
d_g(x,\zai)&\geq& d_g(x,\zaj)-d_g(\zai,\zaj)\geq d_g(x,\zaj)-C_1\maj\\
&\geq& d_g(x,\zaj)\left(1-\frac{C_1}{2RC_2}\right)\geq \frac{1}{2} d_g(x,\zaj).
\end{eqnarray*}
On the one hand, we get that $d_g(x,\zai)\geq RC_2\maj\geq R\mai$ for all $i\in\{1,...,N\}$, so $x\in M-\bigcup_{i}\Omega_{\alpha, i}(R)$. On the other hand, letting $i\in\{1,...,N\}$ be such that $R_\alpha(x)=d_g(x,\zai)$, we get that $R_\alpha(x)\geq \frac{1}{2}R_{\alpha, J}(x)$, which proves the claim.

\medskip\noindent A consequence of \eqref{ineq:45} and the claim is that for any $\lambda>0$, there exists $R_\lambda>0$ such that for all $R>R_\lambda$ we have that
\begin{equation}\label{ineq:46}
R_{\alpha,J}(x)^{2k}|V_\alpha(x)|< \lambda\hbox{ for all }x\in M-\Omega_{\alpha,J}(R)\hbox{ for all }\alpha>0.
\end{equation}
In particular, $R_{\alpha,J}(x)^{2k}|{\bf 1}_{M-\Omega_{\alpha,J}(R)}V_\alpha(x)|< \lambda$ for all $x\in M$ and all $\alpha>0$. The analysis requires the pointwise controls of the Green's function of Theorems \ref{th:Green:main} and \ref{th:Green:pointwise:BIS} of Section \ref{sec:green1}. We fix $$0<\gamma<\min\left\{\frac{n-2k}{2},\frac{k}{\crit-1}\right\}.$$
We let $R>0$ be such that \eqref{ineq:46} holds for $\lambda:=\lambda_\gamma$, where $\lambda_\gamma$ is given by Theorem \ref{th:Green:pointwise:BIS}. Let $\hat{G}_\alpha$ be the Green's function for the operator $P_\alpha-  {\bf 1}_{M-\Omega_{\alpha,J}(R)}V_\alpha$ given by Theorem \ref{th:Green:main}. For any $l=0,...,2k-1$. Inequality \eqref{ineq:funda} of Theorem \ref{th:Green:pointwise:BIS}  yields
\begin{equation}\label{bnd:G:1}
d_g(x,y)^{n-2k+l}|\nabla_y^l \hat{G}_\alpha(x,y)|\leq C \max\left\{1, \left(\frac{d_g(x,y)}{R_{\alpha,J}(x)}\right)^\gamma \left(\frac{d_g(x,y)}{R_{\alpha,J}(y)}\right)^{\gamma+l} \right\}
\end{equation}
for all $x\neq y\in M-\{\zai/\, i\in J\}$. It follows from the symmetry of $P_\alpha$ and the form of the coefficients of $P_\alpha$ that there exist tensors $\bar{A_\alpha}_{lm}$ such that, integrating by parts, for any smooth domain ${\mathcal O}\subset M$, we have that for all $u,v\in C^{2k}(\overline{{\mathcal O}})$,
\begin{equation}\label{ipp:P}
\int_{\mathcal O} (P_\alpha u)v\, dv_g=\int_\Omega u(P_\alpha v)\, dv_g+\sum_{l+m<2k}\int_{\partial{\mathcal O}}\bar{A_\alpha}_{lm}\star\nabla^lu\star \nabla^m v\, d\sigma_g.
\end{equation}
Equation \eqref{eq:P:V:f} reads $(P_\alpha-{\bf 1}_{M-\Omega_{\alpha,J}(R)}V_\alpha)v_\alpha=f_\alpha$ on $M -   \Omega_{\alpha,J}(2R)$. For any $z\in  M-\Omega_{\alpha,J}(3R)$, Green's representation formula yields
\begin{eqnarray*}
v_\alpha(z)&=&\int_{M -  \bigcup_{i\in J}B_{2R\mai}(\zai)}\hat{G}_\alpha(z,\cdot) f_\alpha\, dv_g\\
&&+\sum_{l+m<2k}\int_{\partial(M -  \bigcup_{i\in J}B_{2R\mai}(\zai))}\bar{A_\alpha}_{lm}\star\nabla^l_y\hat{G}_\alpha(z,y)\star \nabla^m v_\alpha(y)\, d\sigma_g(y)\\
&=&\int_{M-\Omega_{\alpha,J}(2R)}\hat{G}_\alpha(z,\cdot) f_\alpha\, dv_g\\
&&-\sum_{i\in J}\sum_{l+m<2k}\int_{\partial B_{2R\mai}(\zai)}\bar{A_\alpha}_{lm}\star\nabla^l_y\hat{G}_\alpha(z,y)\star \nabla^m v_\alpha (y)\, d\sigma_g(y).
\end{eqnarray*}
 We deal with the interior integral. Using the pointwise estimates \eqref{bnd:G:1} and Giraud's Lemma with $0<\gamma<\frac{n-2k}{2}$, we get that
\begin{eqnarray*}
&&\left|\int_{M-\Omega_{\alpha,J}(2R)}\hat{G}_\alpha(z,\cdot) f_\alpha\, dv_g\right|\leq  \Vert f_\alpha\Vert_\infty\int_{M-\Omega_{\alpha,J}(2R)}|\hat{G}_\alpha(z,\cdot)|\, dv_g\\
&&\leq C(\gamma)\Vert f_\alpha\Vert_\infty\int_{M-\Omega_{\alpha,J}(2R)}\max\left(1, \frac{d_g(z,y)^2}{R_{\alpha, J}(z)R_{\alpha, J}(y)} \right)^\gamma d_g(z,y)^{2k-n}\, dv_g(y)\\
&&\leq C(\gamma)\Vert f_\alpha\Vert_\infty\int_{M-\Omega_{\alpha,J}(2R)}\left(1 +\sum_{i\in J}\frac{d_g(z,y)^{2\gamma}}{R_{\alpha, J}(z)^\gamma d_g(y,\zai)^\gamma} \right) d_g(z,y)^{2k-n}\, dv_g(y)\\
&&\leq C(\gamma)\Vert f_\alpha\Vert_\infty \int_{M}d_g(z,y)^{2k-n}\, dv_g(y) \\
&& +C(\gamma)\Vert f_\alpha\Vert_\infty\frac{1}{R_{\alpha, J}(z)^\gamma}\sum_{i\in J}\int_{M}d_g(y,\zai)^{n-\gamma-n} d_g(z,y)^{2\gamma+2k-n}\, dv_g(y) \\
&&\leq C(\gamma)\Vert f_\alpha\Vert_\infty\left(1+\frac{1}{R_{\alpha, J}(z)^\gamma}\right)\leq C\frac{\Vert u_\infty \Vert_\infty^{\crit-1} +\Vert u_\infty\Vert_{C^{2k}}}{R_{\alpha, J}(z)^\gamma}.
\end{eqnarray*}
We deal with the boundary term. It follows from  \eqref{cv:tuai:bis}, $u_\infty\in C^{2k}(M)$ and $R>2C_1$ that there exists $C(R)>0$ such that for all $m<2k$,
$$|\nabla^m v_\alpha(y)|\leq C(R)\mai^{-\frac{n-2k}{2}-m}\hbox{ for all }y\in \partial B_{2R\mai}(\zai).$$
Since $d(z,\zai)>3R\mai$, we get that $d_g(z,y)\asymp d_g(z,\zai)>3R\mai$ for all $y\in \partial B_{2R\mai}(\zai)$. Moreover, with \eqref{lim:34}, we get that $R_{\alpha, J}(y)=d_g(y,\zai)=2R\mai$ for all $y\in \partial B_{2R\mai}(\zai)$. Therefore, 
$$\left(\frac{d_g(z,y)}{R_{\alpha,J}(z)}\right)^\gamma \left(\frac{d_g(z,y)}{R_{\alpha,J}(y)}\right)^{\gamma+l} \asymp \left(\frac{d_g(z,\zai)}{R_{\alpha,J}(z)}\right)^\gamma \left(\frac{d_g(z,\zai)}{\mai}\right)^{\gamma+l}\geq  \left(3R\right)^{\gamma+l}.$$
It then follows from \eqref{bnd:G:1} that for all $l<2k$,
\begin{eqnarray*}
|\nabla^l_y\hat{G}_\alpha(z,y)|&\leq& C(\gamma) d(y,z)^{2k-n-l}\frac{d(y,z)^{2\gamma+l}}{R_{\alpha, J}(y)^{\gamma+l}R_{\alpha, J}(z)^\gamma}\\
&\leq & C(\gamma) \mai^{-\gamma-l}d_g(z,\zai)^{2k-n+2\gamma}R_{\alpha, J}(z)^{-\gamma}
\end{eqnarray*}
and then, for any $i\in J$, we have that
\begin{eqnarray}
&&\left|\sum_{l+m<2k}\int_{\partial B_{2R\mai}(\zai)}\bar{A_\alpha}_{lm}\star\nabla^l_y\hat{G}_\alpha(z,y)\star \nabla^m v_\alpha (y)\, d\sigma_g(y)\right|\label{ineq:78}\\
&&\leq C(\gamma) \sum_{l+m<2k}\int_{\partial B_{2R\mai}(\zai)} \mai^{-\gamma-l}d_g(z,\zai)^{2k-n+2\gamma}R_{\alpha, J}(z)^{-\gamma}\mai^{-\frac{n-2k}{2}-m}\,  d\sigma_g(y)\nonumber\\
&&\leq C(\gamma,R) \sum_{l+m<2k} \mai^{\frac{n-2k}{2}-\gamma }\mai^{2k-1-(m+l)}d_g(z,\zai)^{2k-n+2\gamma}R_{\alpha, J}(z)^{-\gamma}\nonumber
\end{eqnarray}
for all $z\in M-\Omega_{\alpha,J}(3R)$ and $i\in J$. Summing these inequalities yields
\begin{eqnarray*}
|v_\alpha(z)|&\leq &C\frac{\Vert u_\infty \Vert_\infty^{\crit-1} +\Vert u_\infty\Vert_{C^{2k}}}{R_{\alpha, J}(z)^\gamma}+C\sum_{i\in J}\frac{\mai^{\frac{n-2k}{2}-\gamma }d_g(z,\zai)^{2k-n+2\gamma}}{R_{\alpha, J}(z)^{\gamma}}
\end{eqnarray*}
and, coming back to the definition of $\va$, we get that
\begin{eqnarray}\label{ineq:69}
|\ua(z)|&\leq &C\frac{C(u_\infty) }{R_{\alpha, J}(z)^\gamma}+C\sum_{i\in J}\frac{\mai^{\frac{n-2k}{2}-\gamma }d_g(z,\zai)^{2k-n+2\gamma}}{R_{\alpha, J}(z)^{\gamma}}
\end{eqnarray}
for all $z\in M-\Omega_{\alpha,J}(3R)$, where $C(u_\infty):=\Vert u_\infty \Vert_\infty^{\crit-1}+\Vert u_\infty \Vert_{C^{2k}}$.

\medskip\noindent  We let $G_\alpha$ be the Green's function for the operator $P_\alpha$ with for $\alpha>0$ or $\alpha=\infty$. Since $P_\alpha\ua=|\ua|^{\crit-2}\ua$, with \eqref{ipp:P}, we have that
\begin{eqnarray*}
\ua(x)&=&\int_{M-\Omega_{\alpha,J}(3R)} G_\alpha(x,y)|\ua|^{\crit-2}\ua(y)\, dv_g(y)\\
&&+ \sum_{l+m<2k}\int_{\partial(M-\Omega_{\alpha,J}(3R))}\bar{A_\alpha}_{lm}\star\nabla^l_yG_\alpha(x,y)\star \nabla^m u_\alpha(y)\, d\sigma_g(y)
\end{eqnarray*}
for all $x\in M$. Standard estimates on the Green's function (see \cite{robert:gjms}) yield
\begin{equation}\label{est:G:34}
|\nabla_y^lG_\alpha(x,y)|\leq C d_g(x,y)^{2k-n-l}\hbox{ for all }x,y\in M,\, x\neq y\hbox{ and }l\leq 2k-1.
\end{equation}
Therefore, with these pointwise controls, for any $z\in M$ and $\delta>0$, we get that
\begin{eqnarray*}
&&|\ua(z)|\leq \left|\int_{M-{\mathcal B}_\delta(\alpha)} G_\alpha(z,y)|\ua|^{\crit-2}\ua(y)\, dv_g(y)\right|\\
&&+C\int_{{\mathcal B}_\delta(\alpha)-\Omega_{\alpha, J}(3R)}d_g(z,y)^{2k-n}|\ua(y)|^{\crit-1}\, dv_g(y)+C\sum_{i\in J}\int_{\partial B_{3R\mai}(\zai)}
\end{eqnarray*}
where ${\mathcal B}_\delta(\alpha):=\bigcup_{i=1}^NB_\delta(\zai)$
We take  $z\in M-\Omega_{\alpha,J}(4R)$. We fix $i\in J$. With \eqref{est:G:34}, the convergence \eqref{cv:tuai:bis} and $R>C_1$, we get 
\begin{eqnarray}
&&\left|\sum_{l+m<2k}\int_{\partial B_{3R\mai}(\zai)}\bar{A_\alpha}_{lm}\star\nabla^l_yG_\alpha(z,y)\star \nabla^m u_\alpha (y)\, d\sigma_g(y)\right|\nonumber\\
&&\leq C(\gamma) \sum_{l+m<2k}\int_{\partial B_{3R\mai}(\zai)} \mai^{-\frac{n-2k}{2}-m}d_g(z,\zai)^{2k-n-l}\,  d\sigma_g(y)\nonumber\\
&&\leq C(\gamma,R) \sum_{l+m<2k}\mai^{n-1} \mai^{-\frac{n-2k}{2}-m} d_g(z,\zai)^{2k-n-l}\nonumber\\
&&\leq C\mai^{\frac{n-2k}{2}}d_g(z,\zai)^{2k-n}\label{ineq:546}
\end{eqnarray}
for all $i\in J$. Therefore, summing these inequalities and using \eqref{ineq:69} yields
\begin{eqnarray}
|\ua(z)|&\leq  &C\sum_{i\in J} \mai^{\frac{n-2k}{2}}d_g(z,\zai)^{2k-n} +\left|I_{1,\delta}\right|+I_{2,\delta}(\alpha) +\sum_{i\in J}I_i(\alpha)\label{ineq:86}
\end{eqnarray}
$$\hbox{where }I_{1,\delta}(\alpha):= \int_{M-{\mathcal B}_\delta(\alpha)} G_\alpha(z,y)|\ua|^{\crit-2}\ua(y)\, dv_g(y) $$
$$I_{2,\delta}(\alpha):=C\int_{{\mathcal B}_\delta(\alpha)-\Omega_{\alpha, J}(3R)}d_g(z,y)^{2k-n}\left(\frac{C( u_\infty )^{\crit-1} }{R_{\alpha, J}(y)^{\gamma(\crit-1)}}\right)\, dv_g(y)$$
$$I_i(\alpha):=C\int_{M-\Omega_{\alpha, J}(3R)}d_g(z,y)^{2k-n}\left(\frac{\mai^{(\frac{n-2k}{2}-\gamma )(\crit-1)}d_g(y,\zai)^{(\crit-1)(2k-n+2\gamma)}}{R_{\alpha, J}(y)^{\gamma(\crit-1)}}\right)\, dv_g(y).$$
We define ${\mathcal B}_\delta(\infty):=\bigcup_{i=1}^NB_\delta(\lim_{\alpha\to \infty}\zai)$. Since $\lim_{\alpha\to\infty}G_\alpha(x,y)=G_\infty(x,y)$ for all $x\neq y\in M$,  Lemma \ref{prop:lim:out} and \eqref{est:G:34} yield that for any $z\in M$ we get that
\begin{eqnarray}
&& \lim_{\alpha\to \infty}\left|I_{1,\delta}(\alpha)-u_\infty(z)\right|\nonumber\\
&&= \left|\int_{M-{\mathcal B}_\delta(\infty)} G_\infty(z,\cdot)|u_\infty|^{\crit-2}u_\infty\, dv_g-\int_{M} G_\infty(z,\cdot)|u_\infty|^{\crit-2}u_\infty\, dv_g\right|\nonumber\\
&&=\left|\int_{{\mathcal B}_\delta(\infty)} G_\infty(z,y)|u_\infty|^{\crit-2}u_\infty(y)\, dv_g(y)\right|\leq C\delta^{2k}\label{ineq:345}
\end{eqnarray}
where $C$ is independent of $z\in M$. When $u_\infty\equiv 0$, we can even be more precise: using \eqref{ineq:69} and $C(u_\infty)=0$ when $u_\infty\equiv 0$, we get that
\begin{eqnarray}
\left|I_{1,\delta}(\alpha)\right|&\leq&  C(\delta)\left|\int_{M-{\mathcal B}_\delta(\alpha)} d_g(z,y)^{2k-n}( \sum_{i\in J}\mai^{(\frac{n-2k}{2}-\gamma )(\crit-1)} ))\, dv_g(y)\right|\nonumber\\
&\leq& C(\delta) \sum_{i\in J}\mai^{\frac{n-2k}{2}}\hbox{ when }u_\infty\equiv 0\label{ineq:87}
\end{eqnarray}
since $(\crit-1)\gamma<2k$. Now, noting that 
\begin{equation}\label{ineq:R:debut}
\frac{1}{R_{\alpha, J}(z)^{\gamma(\crit-1)}}\leq \sum_{j\in J}\frac{1}{d_g(z,\zaj)^{\gamma(\crit-1)}},
\end{equation}
we get with Giraud's Lemma that
\begin{eqnarray}
I_{2,\delta}(\alpha)&\leq &C\times C( u_\infty )^{\crit-1}\sum_{i,j\in J}\int_{B_\delta(\zai) }d_g(z,y)^{2k-n}d_g(y,\zaj)^{-\gamma(\crit-1)}\, dv_g(y)\nonumber\\
&\leq & C\times C( u_\infty )^{\crit-1}  \delta^{2k-\gamma(\crit-1)}\label{ineq:456}
\end{eqnarray}
since $\gamma(\crit-1)<2k$. Putting together \eqref{ineq:345}, \eqref{ineq:87} and \eqref{ineq:456}, we then get that
\begin{equation}\label{ineq:sup:1}
\left|I_{1,\delta}(\alpha)\right|+I_{2,\delta}(\alpha)\leq C(\delta) \sum_{i\in J}\mai^{\frac{n-2k}{2}}\hbox{ when }u_\infty\equiv 0
\end{equation}
and 
\begin{equation*}
\left|I_{1,\delta}(\alpha)\right|+I_{2,\delta}(\alpha)\leq \Vert u_\infty\Vert_\infty+o(1)+C\times C( u_\infty )^{\crit-1}  \delta^{2k-\gamma(\crit-1)}\hbox{ as }\alpha\to \infty,
\end{equation*}
since $2k-\gamma(\crit-1)>0$. For all $\tau>0$, for $\delta=\delta(u_\infty,\tau)$ small enough, we have that
\begin{equation}\label{ineq:sup:2}
\left|I_{1,\delta}(\alpha)\right|+I_{2,\delta}(\alpha)\leq (1+\tau)\Vert u_\infty\Vert_\infty\hbox{ as }\alpha\to \infty\hbox{ when }u_\infty\not\equiv 0.
\end{equation}
Using again \eqref{ineq:R:debut}, we get that
\begin{equation}\label{ineq:III}
I_i(\alpha)\leq C\sum_{j\in J}\mai^{ (\frac{n-2k}{2}-\gamma )(\crit-1)}I_{i,j}(\alpha)
\end{equation}
where  for $i,j\in J$, we have set
\begin{equation}\label{def:I:ij}
I_{i,j}(\alpha):=\int_{M-B_{3R\mai}(\zai)- B_{3R\maj}(\zaj)}F_\alpha(y)\, dv_g(y)
\end{equation}
$$\hbox{where }F_\alpha(y):=d_g(z,y)^{a-n}d_g(y,\zai)^{-b-n}d_g(y,\zaj)^{-\eps}$$
with $a=2k$, $b=2k-2\gamma(\crit-1)>0$ and $\eps=\gamma(\crit-1)$. We split the integral
\begin{equation*}
I_{i,j}(\alpha)\leq \int_{D_\alpha^1}F_\alpha \, dv_g+ \int_{D_\alpha^2}F_\alpha \, dv_g+ \int_{D_\alpha^3}F_\alpha \, dv_g
\end{equation*}
with
$$D_\alpha^1:=\left(M-B_{3R\mai}(\zai) \right)\cap \left\{d_g(y,\zaj)\geq\frac{1}{2}d_g(y,\zai)\right\}\cap \left\{d_g(z,y)\geq\frac{1}{2}d_g(z,\zai)\right\}$$
$$D_\alpha^2:=\left(M-B_{3R\mai}(\zai) \right)\cap \left\{d_g(y,\zaj)\geq\frac{1}{2}d_g(y,\zai)\right\}\cap \left\{d_g(z,y)<\frac{1}{2}d_g(z,\zai)\right\}$$

$$D_\alpha^3:=\left(M-B_{3R\mai}(\zai)- B_{3R\maj}(\zaj)\right)\cap \left\{d_g(y,\zaj)<\frac{1}{2}d_g(y,\zai)\right\}.$$
We have that
\begin{eqnarray*} \int_{D_\alpha^1}F_\alpha \, dv_g&\leq &  Cd_g(z,\zai)^{a-n}\int_{M-B_{R\mai}(\zai)}d_g(y,\zai)^{-b-n-\eps}\, dv_g(y)\\
&\leq &Cd_g(z,\zai)^{a-n}\mai^{-b-\eps} .
\end{eqnarray*}
Since $d_g(y,\zai)\geq d_g(z,\zai)- d_g(z,y)>\frac{1}{2}d_g(z,\zai)$ for all $y\in D_\alpha^2$, we get that
\begin{eqnarray*} \int_{D_\alpha^2}F_\alpha \, dv_g&\leq &  2^{b+n+\eps}d_g(z,\zai)^{-b-n-\eps}\int_{\left\{d_g(z,y)<\frac{1}{2}d_g(z,\zai)\right\}}d_g(z,y)^{a-n}\, dv_g(y)\\
&\leq &Cd_g(z,\zai)^{a-b-n-\eps}\leq C d_g(z,\zai)^{a-n}\mai^{-b-\eps} \left(\frac{\mai}{d_g(z,\zai)}\right)^{b+\eps}\\
&\leq &  C d_g(z,\zai)^{a-n}\mai^{-b-\eps}
\end{eqnarray*}
since $z\in M-B_{4R\mai}(\zai)$.

\smallskip\noindent For any $y\in D_\alpha^3$, $|d_g(y,\zai)-d_g(\zaj,\zai)|\leq d_g(y,\zaj)<\frac{1}{2}d_g(y,\zai)$, and then we get $\frac{2}{3}d_g(\zai,\zaj)<d_g(\zai,y)<2d_g(\zai,\zaj)$, and then
\begin{eqnarray*} \int_{D_\alpha^3}F_\alpha \, dv_g&\leq & C d_g(\zai,\zaj)^{-b-n} \int_{D_\alpha} d_g(z,y)^{a-n}d_g(y,\zaj)^{-\eps}\, dv_g(y)
\end{eqnarray*}
where $D_\alpha:=\left(M- B_{3R\maj}(\zaj)\right)\cap\{d_g(y,\zaj)< d_g(\zai,\zaj)\}$. Assume that $d_g(z,\zaj)\geq 2 d_g(\zai,\zaj)$. We then get that
\begin{equation}\label{eq:123}
d_g(z, \zai)\leq d_g(z,\zaj)+d(\zaj,\zai)\leq \frac{3}{2}d_g(z,\zaj).
\end{equation}
For all $y\in M$ such that $d_g(y,\zaj)< d_g(\zai,\zaj)$, we have that 
\begin{eqnarray*}
d_g(z,y)&\geq& d_g(z,\zaj)-d_g(y,\zaj)>d_g(z,\zaj)- d_g(\zai,\zaj)\\
&\geq& d_g(z,\zaj)-\frac{1}{2}d_g(z,\zaj)=\frac{1}{2}d_g(z,\zaj).
\end{eqnarray*}
Therefore, using \eqref{eq:123} and  \eqref{lim:34}, we get that 
\begin{eqnarray} 
&&\int_{D_\alpha^3}F_\alpha \, dv_g\nonumber\\
&&\leq  C d_g(\zai,\zaj)^{-b-n} d_g(z,\zaj)^{a-n}\int_{ \{d_g(y,\zaj)< d_g(\zai,\zaj)\}} d_g(y,\zaj)^{-\eps}\, dv_g(y)\nonumber\\
&&\leq C d_g(\zai,\zaj)^{-b-\eps} d_g(z,\zaj)^{a-n}\leq C  d_g(\zai,\zaj)^{-b-\eps} d_g(z,\zai)^{a-n}\nonumber\\
&&\leq  C   \left(\frac{\mai}{d_g(\zai,\zaj)}\right)^{b+\eps}\mai^{-b-\eps} d_g(z,\zai)^{a-n}\leq C\mai^{-b-\eps} d_g(z,\zai)^{a-n}\label{eq:145}
\end{eqnarray}
Assume that $d_g(z,\zaj)<2 d_g(\zai,\zaj)$. The triangle inequality yields
\begin{equation}\label{eq:124}
d_g(z,\zai)\leq d_g(z,\zaj)+d_g(\zaj,\zai)\leq 3 d_g(\zai,\zaj).
\end{equation}
We let $Z\in B_2(0)\subset\rn$ be such that $z=\hbox{exp}_{\zaj}(d_g(\zai,\zaj)Z)$. Performing the change of variable $y=\hbox{exp}_{\zaj}(d_g(\zai,\zaj)Y)$, comparing the distances via \eqref{lem:lip}, using $\eps<n$,  with Giraud's Lemma, \eqref{eq:124} and arguing as in \eqref{eq:145}, we get that
\begin{eqnarray*} \int_{D_\alpha^3}F_\alpha \, dv_g&\leq & C d_g(\zai,\zaj)^{a-b-\eps-n} \int_{\{|Y|< 1 \}} |Y-Z|^{a-n}|Y|^{-\eps}\, dY\\
&\leq & C d_g(\zai,\zaj)^{a-b-\eps-n} \leq C  d_g(\zai,\zaj)^{-b-\eps}  d_g(\zai,\zaj)^{a -n}\\
&\leq & C\mai^{-b-\eps} d_g(z,\zai)^{a-n}.
\end{eqnarray*}
Putting all these identities in \eqref{ineq:III} yields
\begin{equation}\label{ineq:95}
I_i(\alpha)\leq C \mai^{ (\frac{n-2k}{2}-\gamma )(\crit-1)}\mai^{-b-\eps} d_g(z,\zai)^{a-n}\leq C \mai^{ \frac{n-2k}{2}}d_g(z,\zai)^{a-n}.
\end{equation}
Then, putting together \eqref{ineq:sup:1}, \eqref{ineq:sup:2} and \eqref{ineq:95} into \eqref{ineq:86}, we get that for all $\tau>0$, there exists $C_\tau>0$ such that
\begin{eqnarray}\label{ineq:part:1}
|\ua(z)|&\leq& (1+\tau)\Vert u_\infty\Vert_\infty+C_\tau\sum_{i\in J} \mai^{\frac{n-2k}{2}}d_g(z,\zai)^{2k-n} 
\end{eqnarray}
 for all $ z\in M-\Omega_{\alpha,J}(4R)$. It follows from \eqref{choice:1} that $\za\in M-\Omega_{\alpha,J}(4R)$ for $\alpha\to +\infty$. Therefore \eqref{ineq:part:1} holds for $z:=z_\alpha$. 
\subsection{The case when the $(\za)_\alpha$ approaches some of the $(\zai)_\alpha$}
We now assume that there exists $i\in\{1,...,N\}$ such that $d_g(\za,\zai)=O(\mai)$ and we define
\begin{equation}\label{choice:2}
I:=\{i\in\{1,...,N\}\hbox{ such that }d_g(\za,\zai)=O(\mai)\}\neq \emptyset.
\end{equation}
We claim that $I$ is fully ordered. Indeed, for $i,j\in I$ such that $d_g(\za,\zai)=O(\mai)$ and $d_g(\za,\zaj)=O(\maj)$. We then get that $d_g(\zai,\zaj)=O(\mai+\maj)$. Without loss of generality, we assume that $\mai=O(\maj)$, then $d_g(\zai,\zaj)=O( \maj)$ and then $i\preceq j$. This proves the claim.

\smallskip\noindent We set $l:=\min I$ for  $\preceq$.  We first assume that there exists $\eps_0>0$ such that
\begin{equation}\label{choice:3}
d_g(\za,\zai)\geq \eps_0\mal\hbox{ for all }i\in I_l,
\end{equation}
where $I_l=\{j\in\{1,...,N\}-\{l\}\hbox{ such that }d_g(\zal,\zaj)=O(\mal)\}$ is defined in \eqref{def:Ii}.  Since $d_g(\za,\zal)=O(\mal)$, it   follows from \eqref{cv:tuai:bis} that
\begin{equation}\label{ineq:110}
|\ua(\za)|\leq C\mal^{-\frac{n-2k}{2}}\leq C\left(\frac{\mal}{\mal^2+d_g(\za,\zal)^2}\right)^{\frac{n-2k}{2}}.
\end{equation}
We now assume that  there exists $i\in I_l$ such that $d_g(\za,\zai)=o(\mal)$. Therefore
\begin{equation*}
J\neq \emptyset\hbox{ where }J:=\{i\in I_l\hbox{ such that }d_g(\za,\zai)=o(\mal)\}
\end{equation*}
and we let $J_0$ be the set of maximal elements of $J$ for the order $\preceq$. Note that it follows from \eqref{eq:53} that for all $i\in J\subset I_l$, we have that $i\preceq l$, $i\neq l$ and $\mai=o(\mal)$. We  fix $i_0\in J_0$. We define
$$D_{J_0,\alpha}(R):=B_{\frac{1}{R}\mal}(z_{\alpha, i_0})-\bigcup_{i\in J_0}B_{R\mai}(\zai).$$
Arguing as to prove of \eqref{res:89}, we get that for $R>\max\{2C_1, 2C_1 C_2^{-1}\}$, we have that 
\begin{equation}\left\{\begin{array}{c}
R_\alpha(x)\geq \frac{1}{2}\min\{d_g(x,\zai)/\, i\in J_0\}\\
\hbox{ and }x\in M-\bigcup_{i=1}^N\Omega_{\alpha, i}(R)\end{array}\right\}\label{res:90}
\end{equation}
for all $x\in D_{J_0,\alpha}(2R)$. It then follows from \eqref{lim:exhaust:bis} and \eqref{res:90} that for any $\lambda>0$, there exists $R>0$ such that
\begin{equation}\label{ineq:46:bis}
R_{\alpha, J_0}(x)^{\frac{n-2k}{2}}|\ua(x)|\leq \lambda\hbox{ for all }x\in D_{J_0,\alpha}(R)
\end{equation}
where $R_{\alpha, J_0}(x):=\min\{d_g(x,\zai)/\, i\in J_0\}$ for all $x\in M$. The argument is now similar to the one used in the proof of \eqref{ineq:69}.  We fix $$0<\gamma<\min\left\{\frac{n-2k}{2},\frac{k}{\crit-1}\right\}.$$
We let $R>\max\{2C_1,2C_1 C_2^{-1}\}$ be such that \eqref{ineq:46:bis} holds for $\lambda:=\lambda_\gamma$, where $\lambda_\gamma$ is given in Theorem \ref{th:Green:pointwise:BIS}. Let $\bar{G}_\alpha$ be the Green's function for the operator $P_\alpha-  {\bf 1}_{D_{J_0,\alpha}(R)}|\ua|^{\crit-2}$ given by Theorem \ref{th:Green:main}. The pointwise estimates \eqref{ineq:funda} of Theorem \ref{th:Green:pointwise:BIS} then yield for any $l=0,...,2k-1$
\begin{equation}\label{bnd:G:2}
d_g(x,y)^{n-2k+l}|\nabla_y^l \bar{G}_\alpha(x,y)|\leq C \max\left\{1, \left(\frac{d_g(x,y)}{R_{\alpha,J_0}(x)}\right)^\gamma \left(\frac{d_g(x,y)}{R_{\alpha,J_0}(y)}\right)^{\gamma+l} \right\}
\end{equation}
for all $x\neq y\in M-\{\zai/\, i\in J_0\}$. We use again \eqref{ipp:P}. Since for any $z\in  D_{J_0,\alpha}(3R)$,$(P_\alpha-{\bf 1}_{D_{J_0,\alpha}(R)}|\ua|^{\crit-2})\ua=0$ on $D_{J_0,\alpha}(2R)$ and since the balls involved in the definition $D_{J_0,\alpha}(2R)$ are all disjoint, Green's representation formula yields
\begin{eqnarray*}
&&\ua(z)\\
&&= \sum_{l+m<2k}\int_{\partial D_{J_0,\alpha}(2R)}\bar{A_\alpha}_{lm}\star\nabla^l_y\bar{G}_\alpha(z,y)\star \nabla^m \ua(y)\, d\sigma_g(y)\\
&&= \sum_{l+m<2k}\int_{\partial B_{\frac{1}{2R}\mal}(z_{\alpha, i_0}) }\bar{A_\alpha}_{lm}\star\nabla^l_y\bar{G}_\alpha(z,y)\star \nabla^m \ua(y)\, d\sigma_g(y)\\&&-\sum_{i\in J_0}\sum_{l+m<2k}\int_{\partial B_{2R\mai}(\zai)}\bar{A_\alpha}_{lm}\star\nabla^l_y\bar{G}_\alpha(z,y)\star \nabla^m \ua (y)\, d\sigma_g(y).
\end{eqnarray*}
The second integral of the right-hand-side is estimated as in \eqref{ineq:78}, so  we get that 
\begin{eqnarray*}
&&\left|\sum_{i\in J_0}\sum_{l+m<2k}\int_{\partial B_{2R\mai}(\zai)}\bar{A_\alpha}_{lm}\star\nabla^l_y\bar{G}_\alpha(z,y)\star \nabla^m \ua (y)\, d\sigma_g(y)\right|\\
&&\leq C\sum_{i\in J_0}\frac{\mai^{\frac{n-2k}{2}-\gamma }d_g(z,\zai)^{2k-n+2\gamma}}{R_{\alpha, J_0}(z)^{\gamma}}\end{eqnarray*}
for all $z\in D_{J_0,\alpha}(3R)$. We deal with the first integral. Noting that $i_0\in J_0\subset I_l$ and $2R>C_1$ as in \eqref{lim:d:1}, it follows from  \eqref{cv:tuai:bis} that there exists  $C(R)>0$ such that
$$|\nabla^m \ua(y)|\leq C(R)\mal^{-\frac{n-2k}{2}-m}\hbox{ for all }m<2k\hbox{ and }y\in   \partial   B_{\frac{1}{2R}\mal}(z_{\alpha, i_0}).$$
Since $d(z,z_{\alpha, i_0})<\frac{\mal}{3R}$, we get that $d_g(z,y)\asymp d_g(z_{\alpha, i_0},y)\asymp\mal$ for all $y\in  \partial   B_{\frac{1}{2R}\mal}(z_{\alpha, i_0})$. Moreover, we have that
$$|d_g(y, \zai)-d_g(y,z_{\alpha, i_0} )|\leq d_g(\zai, z_{\alpha, i_0})=o(\mal)$$
for all $i\in J_0\subset J$ and $y\in  \partial   B_{\frac{1}{2R}\mal}(z_{\alpha, i_0})$. Therefore, $R_{\alpha, J_0}(y)\asymp \mal$ for all $y\in  \partial   B_{\frac{1}{2R}\mal}(z_{\alpha, i_0})$. We then get that $ R_{\alpha, J_0}(y)\leq C d_g(z,y)$ for all $y\in  \partial   B_{\frac{1}{2R}\mal}(z_{\alpha, i_0})$. Moreover, $R_{\alpha, J_0}(z)\leq d_g(z, z_{\alpha, i_0})\leq \frac{\mal}{3R}$. It then follows from \eqref{bnd:G:2} that 
\begin{eqnarray*}
|\nabla^l_y\bar{G}_\alpha(z,y)|&\leq& C(\gamma) d(y,z)^{2k-n-l}\frac{d(y,z)^{2\gamma+l}}{R_{\alpha, J_0}(y)^{\gamma+l}R_{\alpha, J_0}(z)^\gamma}\leq  C(\gamma) \frac{\mal^{2k-n-l+\gamma}}{ R_{\alpha, J_0}(z)^\gamma}
\end{eqnarray*}
for all $l<2k$ and $y\in  \partial   B_{\frac{1}{2R}\mal}(z_{\alpha, i_0})$. Therefore
\begin{eqnarray}
&& \sum_{l+m<2k}\int_{\partial B_{\frac{1}{2R}\mal}(z_{\alpha, i_0}) }\bar{A_\alpha}_{lm}\star\nabla^l_y\bar{G}_\alpha(z,y)\star \nabla^m \ua(y)\, d\sigma_g(y)\nonumber\\
&&\leq C  \frac{1}{ R_{\alpha, J_0}(z)^\gamma} \mal^{ -\frac{n-2k}{2}+ 2k-1-m-l+\gamma}\leq C   \frac{1}{ R_{\alpha, J_0}(z)^\gamma} \mal^{ -\frac{n-2k}{2} +\gamma}.\label{ineq:120}
\end{eqnarray}
So that, summing these inequalities yields
\begin{equation*}
|\ua(z)|\leq  C\frac{1}{ R_{\alpha, J_0}(z)^\gamma} \mal^{ -\frac{n-2k}{2} +\gamma}+ C\sum_{i\in J_0}\frac{\mai^{\frac{n-2k}{2}-\gamma }d_g(z,\zai)^{2k-n+2\gamma}}{R_{\alpha, J_0}(z)^{\gamma}}
\end{equation*}
for all $z\in D_{\alpha, J_0}(3R)$. We now argue as in the proof of \eqref{ineq:part:1} and we let $G_\alpha$ be the Green's function for the operator $P_\alpha$. With \eqref{ipp:P}, we have that
\begin{eqnarray*}
\ua(x)&=&\int_{D_{\alpha, J_0}(3R)} G_\alpha(x,y)|\ua|^{\crit-2}\ua(y)\, dv_g(y)\\
&&+ \sum_{l+m<2k}\int_{\partial D_{\alpha, J_0}(3R)}\bar{A_\alpha}_{lm}\star\nabla^l_yG_\alpha(x,y)\star \nabla^m u_\alpha\, d\sigma_g
\end{eqnarray*}
for all $x\in M$. With \eqref{est:G:34}, for any $z\in D_{\alpha, J_0}(4R)$, we get that
\begin{eqnarray*}
|\ua(z)|&\leq& C\int_{D_{\alpha, J_0}(3R)}d_g(z,y)^{2k-n}|\ua(y)|^{\crit-1}\, dv_g(y)\\
& &+C\int_{\partial B_{\frac{1}{3R}\mal}(z_{\alpha, i_0})}+C\sum_{i\in J_0}\int_{\partial B_{3R\mai}(\zai)}.
\end{eqnarray*}
Arguing as in the proof of \eqref{ineq:120}, we get that $\int_{\partial B_{\frac{1}{3R}\mal}(z_{\alpha, i_0})}=O(\mal^{ -\frac{n-2k}{2}  })$. So, as in \eqref{ineq:546}, we get that
\begin{eqnarray}
|\ua(z)|&\leq& C\mal^{ -\frac{n-2k}{2}  }+C\sum_{i\in J_0} \mai^{\frac{n-2k}{2}}d_g(z,\zai)^{2k-n} + \tilde{I}_0(\alpha) +\sum_{i\in J_0}\tilde{I}_i(\alpha)\label{final:1}
\end{eqnarray}
for all $z\in D_{\alpha, J_0}(4R)$, where
$$\tilde{I}_0(\alpha):=C\int_{D_{\alpha, J_0}(3R)}d_g(z,y)^{2k-n}\left(\frac{\mal^{ -\frac{n+2k}{2} +\gamma(\crit-1)}  }{R_{\alpha, J_0}(y)^{\gamma(\crit-1)}}\right)\, dv_g(y)$$
$$\tilde{I}_i(\alpha):=C\int_{D_{\alpha, J_0}(3R)}d_g(z,y)^{2k-n} \frac{\mai^{(\frac{n-2k}{2}-\gamma )(\crit-1)}d_g(y,\zai)^{(\crit-1)(2k-n+2\gamma)}}{R_{\alpha, J_0}(y)^{\gamma(\crit-1)}} \, dv_g(y).$$
Now, using \eqref{ineq:R:debut}, we get that
\begin{eqnarray*}
&&\tilde{I}_0(\alpha)\leq C\mal^{ -\frac{n+2k}{2} +\gamma(\crit-1)}  \sum_{j\in J_0}\int_{D_{\alpha, J_0}(3R) }d_g(z,y)^{2k-n}d_g(y,\zaj)^{-\gamma(\crit-1)}\, dv_g(y)\\
&&\leq C\mal^{ -\frac{n+2k}{2} +\gamma(\crit-1)}  \sum_{j\in J_0}\int_{B_{\frac{1}{3R}\mal}(z_{\alpha, i_0}) }d_g(z,y)^{2k-n}d_g(y,\zaj)^{-\gamma(\crit-1)}\, dv_g(y).
\end{eqnarray*}
For any $j\in J_0$, we have that $d_g(\zaj, z_{\alpha, i_0})=o(\mal)$, we then let $Z_{\alpha, j}\in B_1(0)\subset \rn$ be such that $\zaj:=\hbox{exp}_{z_{\alpha, i_0}}(\mal Z_{\alpha, j})$ with $\lim_{\alpha\to \infty}Z_{\alpha, j}=0$. Since $z\in B_{\frac{1}{4R}\mal}(z_{\alpha, i_0})$, we let $Z\in B_{1/(4R)}(0)$ be such that $z=\hbox{exp}_{z_{\alpha, i_0}}(\mal Z )$. The change of variable $y:=\hbox{exp}_{z_{\alpha, i_0}}(\mal Y)$, \eqref{lem:lip}, $\gamma(\crit-2)<2k$ and  Giraud's Lemma yield
\begin{eqnarray}
&&\tilde{I}_0(\alpha)\leq  \nonumber\\
&&\leq C\mal^{ -\frac{n-2k}{2} }  \sum_{j\in J_0}\int_{B_{\frac{1}{3R}}(0) } |Z-Y|^{2k-n}| Y- Z_{\alpha, j}|^{-\gamma(\crit-1)} \, dY\leq C\mal^{ -\frac{n-2k}{2} }. \label{final:2}
\end{eqnarray}
Using again \eqref{ineq:R:debut}, we get that $\tilde{I}_i(\alpha)\leq C\sum_{j\in J_0}\mai^{ (\frac{n-2k}{2}-\gamma )(\crit-1)}I_{i,j}(\alpha)$ where the $I_{i,j}(\alpha)$ have been defined in \eqref{def:I:ij}. Therefore, with the computations made to prove \eqref{ineq:95} and the inequalities \eqref{final:1} and \eqref{final:2},  we get
\begin{eqnarray*}
|\ua(z)|&\leq& C\mal^{ -\frac{n-2k}{2}  }+C\sum_{i\in J_0} \mai^{\frac{n-2k}{2}}d_g(z,\zai)^{2k-n}  \end{eqnarray*}
for all $z\in D_{\alpha, J_0}(4R)$. The definition of $D_{\alpha, J_0}(4R)$ then yields
\begin{eqnarray}
|\ua(z)|&\leq&  C\sum_{i\in J_0\cup\{l\}} \left(\frac{\mai}{\mai^2+d_g(z,\zai)^2}\right)^{\frac{n-2k}{2}}   \label{case:2}
\end{eqnarray}
for all  $z\in D_{\alpha, J_0}(4R)$. Since $\za\in D_{\alpha, J_0}(4R)$ for $\alpha\to +\infty$, we get that \eqref{case:2} holds for $z:=z_\alpha$. 

\smallskip\noindent We are in position to conclude. If \eqref{choice:1} holds, then we have proved \eqref{ineq:part:1}. If  \eqref{choice:2}  and \eqref{choice:3} hold, we have proved \eqref{ineq:110}. If  \eqref{choice:2} holds and \eqref{choice:3} does not hold, we have proved  \eqref{case:2} holds for $z:=z_\alpha$. These cases prove that for all $\tau>0$, there exists $C_\tau>0$ such that 
\begin{eqnarray*}
|\ua(\za)|&\leq& (1+\tau)\Vert u_\infty \Vert_\infty+C_\tau\sum_{i=1}^N \left(\frac{\mai}{\mai^2+d_g(\za,\zai)^2}\right)^{\frac{n-2k}{2}} 
\end{eqnarray*}
for any family $(\za)_\alpha\in M$. This proves Theorem \ref{th:estim-co:intro}.

\section{Green's function with Hardy potential: construction and first estimates}\label{sec:green1}
\begin{defi}\label{def:p} We say that an operator $P$ is of type $O_{k,L}$ if
$$\bullet\; P:=\Delta_g^k+\sum_{i=0}^{k-1} (-1)^i\nabla^i(A^{(i)} \nabla^i)\hbox{ for }\alpha>0$$
where for all $i=0,...,k-1$, $ A^{(i)}\in C^{i}_{\chi}(M,\Lambda_{S}^{(0,2i)}(M))$ is a family of $C^{i}-$field of $(0,2i)-$tensors on $M$ such that:\par
$\bullet$ For any $i$ , $A^{(i)}(T,S)=A^{(i)}(S,T)$ for any $(i,0)-$tensors $S$ and $T$. \par
$\bullet$ $\Vert A^{(i)}\Vert_{C^{i}}\leq L$ for all $i=0,...,k-1$\par
$\bullet$ $P$ is  coercive and 
$$\int_M u P u\, dv_g=\int_M (\Delta_g^{\frac{k}{2}}u)^2\, dv_g+\sum_{i=0}^{k-1}\int_M A^{(i)} (\nabla^iu,\nabla^i u)\, dv_g\geq \frac{1}{L}\Vert u\Vert_{H_k^2}^2$$
for all $u\in H_k^2(M)$.\end{defi}
For $N\geq 1$  and $\lambda>0$, we define
\begin{equation*}
\Pl(N):=\left\{\begin{array}{cc}
(\F, V)\in M^N\times L^1(M)\hbox{ such that}\\
|V(x)|\leq \lambda R_{\mathcal F}(x)^{-2k}\hbox{ for all }x\in \Mmx\end{array}\right\}.
\end{equation*}
where for $\F=\{p_1,...,p_N\}$, we have let
$$R_{\mathcal F}(x):=\min\{d_g(x,p_i)/\, i=1,...,N\}.$$
This section is devoted to the proof of the following theorems:
\begin{theorem}\label{th:Green:main} Let $(M,g)$ be a compact Riemannian manifold of dimension $n$ without boundary. Fix $k\in\nn$ such that $2\leq 2k<n$, $L>0$, $\lambda>0$ and  $N\geq 1$. We consider an operator $P\in O_{k,L}$ and a   $(\mathcal F,V)\in \Pl(N)$ such that $P-V$ is coercive. We set    ${\mathcal F}:=\{p_1,...,p_N\}\subset M$. Then there exists $\lambda_0(k,L)>0$ such that for $\lambda<\lambda_0(k,L)$, there exists $G:(M -  \F)\times (M -  \F) - \{(z,z)/z\in M -  \F\}\to \rr$ such that
\begin{itemize}
\item For all $x\in M -  \F$, $G(x,\cdot)\in L^q(M)$ for all $1\leq q<\frac{n}{n-2k}$;
\item For all $x\in M -  \F$, $G(x,\cdot)\in L^{\crit}_{loc}(M - \{x\})$;
\item For all $f\in L^{\frac{2n}{n+2k}}(M)\cap L^p_{loc}(M -  \F)$, $p>\frac{n}{2k}$, we let $\varphi\in H_k^2(M)$ such that $P\varphi=f$ in the weak sense. Then $\varphi\in C^0(M -  \F)$ and
$$\varphi(x)=\int_M  G(x,\cdot) f\, dv_g\hbox{ for all }x\in M -  \F.$$
\end{itemize}
Moreover, such a function $G$ is unique. It is the Green's function for $P-V$. In addition, for all $x\in M -  \F$, $G(x,\cdot)\in H_{2k}^p(M - \{x, \F\})$ for all $1<p<\infty$.
\end{theorem}
In addition, we get the following pointwise control:
\begin{theorem}\label{th:Green:pointwise:BIS} Let $(M,g)$ be a compact Riemannian manifold of dimension $n$ without boundary. Fix $k\in\nn$ such that $2\leq 2k<n$, $L>0$, $\lambda>0$ and   $N\geq 1$.  We consider an operator $P\in O_{k,L}$. Then for any $\gamma\in (0,n-2k)$, there exists $\lambda_\gamma>0$ such that all $\lambda\leq \lambda_\gamma$ and $(\F,V)\in \Pl(N)$ such that  
$$R_{\mathcal F}(x)^{2k}|V(x)|\leq \lambda\hbox{ for all }x\in \Mmx,\, \F:=\{p_1,...,p_N\}\subset M,$$
the Green's function $G$ of $P-V$ defined as in Theorem \ref{th:Green:main} satisfies the following pointwise estimates: for any $l_1, l_1\leq 2k-1$, then 
\begin{eqnarray}
&&d_g(x,y)^{n-2k+l_1+l_2}|\nabla_x^{l_1}\nabla_y^{l_2}G(x,y)|\nonumber\\
&&\leq C(\gamma,L,\lambda, k,N) \max\left\{1, \left(\frac{d_g(x,y)}{ R_{\mathcal F}(x) }\right)^{\gamma+l_1}\left(\frac{d_g(x,y)}{ R_{\mathcal F}(y) }\right)^{\gamma+l_2}\right\}\label{ineq:funda}
\end{eqnarray}
for all $x,y\in \Mmx$ such that $x\neq y$, where $C(\gamma,L,\lambda, k,N)$ depends only on $(M,g)$, $L$, $\lambda$ and $N$. Note that these estimates are uniform with respect to  $\mathcal F$.
\end{theorem}
We fix $k\in\nn$ such that $2\leq 2k<n$ and $L>0$. We consider an operator $P\in O_{k,L}$ (see Definition \ref{def:p}) and $N\geq 1$. Here and in the sequel,
\begin{equation*}
 B_{r}(\F):=\bigcup_{p\in\F} B_{r}(p)\hbox{ for all }\F\subset M^N\hbox{ and }r>0.
 \end{equation*}
The Hardy inequality yields $C_H(k)>0$  such that for all $p\in M$,
\begin{equation*}
\int_M\frac{u^2\, dv_g}{d_g(x,p)^{2k}} \leq C_H(k)\Vert u\Vert_{H_k^2}^2\hbox{ for all }u\in H_k^2(M)\hbox{ and all }p\in M.
\end{equation*}
Note that the constant is independent of the choice of $p\in M$. This is essentially a consequence of the Euclidean Hardy inequality by Mitidieri \cite{mitidieri}, see Robert \cite{robert:gjms} for the adaptation to the Riemannian setting.

\smallskip\noindent We  let $\eta\in C^\infty(\rr)$ be such that $\eta(t)=0$ for $t\leq 1$ and $\eta(t)=1$ for $t\geq 2$. As in \cite{robert:gjms}, there exists $\lambda_0=\lambda_0(k, L)$ such for all $(\F,V_0)\in \Pl(N) $ with $0<\lambda<\lambda_0$, then 
\begin{equation*}
\int_M (Pu-V_0 u)u\, dv_g\geq \frac{1}{2L}\Vert u\Vert_{H_k^2}^2\hbox{ for all }u\in H_k^2(M)
\end{equation*}
and defining $V_\eps(x):=\eta(R_\F(x)/\eps) V_0(x)$ for all $\eps>0$ and a.e. $x\in M$, we get that
\begin{equation}\label{hyp:Ve}
\left\{\begin{array}{cc}
\lim_{\eps\to 0}V_\eps(x)=V_0(x)&\hbox{ for a.e. }x\in \Mmx\\
(\F,V_\eps)\in \Pl(N)&\hbox{ for all }\eps>0\\
P-V_\eps&\hbox{ is uniformly coercive for all }\eps>0
\end{array}\right\}
\end{equation}
in the sense that
\begin{equation}\label{coer:PV}
\int_M (Pu-V_\eps u)u\, dv_g\geq \frac{1}{2L}\Vert u\Vert_{H_k^2}^2\hbox{ for all }u\in H_k^2(M)\hbox{ and }\eps>0.
\end{equation}
For any $\eps>0$,  we let $G_\eps$ be the Green's function for the operator $P-V_\eps$. Since $V_\eps\in L^\infty(M)$, the existence of $G_\eps$ follows from Theorem C.1 of Robert \cite{robert:gjms}.

\medskip\noindent{\bf Step 1: First pointwise control.} We choose $f\in C^{0}(M)$ and we fix $\eps>0$. Since $P-V_\eps$ is coercive, variational methods yield a unique  $\varphi_\eps\in H_k^2(M)$ such that 
$$(P-V_\eps)\varphi_\eps=f\hbox{ in }M\hbox{ in the weak sense.}$$
It follows from the second point of Definition   \ref{def:p} of $P\in O_{k,L}$ that $A^{(i)}\in C^{i}$ for all $i=0,...,2k-1$, it follows from elliptic regularity (\cite{ADN} or Theorem D.4 of \cite{robert:gjms}) and Sobolev's embedding theorem that $\varphi_\eps\in C^{2k-1}(M)\cap H_{2k}^p(M)$ for all $p>1$. The coercivity hypothesis \eqref{coer:PV} yields
\begin{eqnarray*}
\frac{1}{2L}\Vert \varphi_\eps\Vert_{H_k^2}^2\leq \int_M (P\varphi_\eps-V_\eps\varphi_\eps)\varphi_\eps\, dv_g=\int_M f\varphi_\eps\, dv_g\leq \Vert f\Vert_{\frac{2n}{n+2k}}\Vert\varphi_\eps\Vert_{\frac{2n}{n-2k}}.
\end{eqnarray*}
The Sobolev  inequality \eqref{sobo:ineq:M} yields
\begin{equation}\label{eq:14}
C_S(k)^{-1}\Vert\varphi_\eps\Vert_{\crit}\leq \Vert\varphi_\eps \Vert_{H_k^2}\leq 2LC_S(k)\Vert f\Vert_{\frac{2n}{n+2k}}
\end{equation}
for all $f\in C^{0}(M)$. We fix $p>1$ such that
$$\frac{n}{2k}<p<\frac{n}{2k-1}\hbox{ and }\theta_p:=2k-\frac{n}{p}\in (0,1).$$
We fix $\delta>0$. Since $(\F,V_\eps)\in \Pl(N)$ for all $\eps>0$ and $P\in O_{k,L}$, it follows from
regularity theory (\cite{ADN} or Theorem D.2 of \cite{robert:gjms}) that
\begin{eqnarray}
&&\Vert\varphi_\eps\Vert_{C^{0,\theta_p}(M -  B_\delta(\F))}\leq  C(p,\delta,k)\Vert \varphi_\eps\Vert_{H_{2k}^p(M -   B_\delta(\F))}\label{case:F:F0}\\
&&\leq C(p,\delta,k, L,\lambda_0) \left(\Vert f\Vert_{L^p(M -   B_{\delta/2}(\F))}+\Vert \varphi_\eps\Vert_{L^{\crit}(M -   B_{\delta/2}(\F))}\right).\nonumber
\end{eqnarray}
Indeed, this is a direct consequence of regularity theory when $p_1,...,p_N$ are in small neighborhoods of some fixed distinct points, and the general case goes by contradiction and compactness. With \eqref{eq:14} and noting that $\frac{n}{2k}>\frac{2n}{n+2k}$, we get that
\begin{equation*}
\Vert\varphi_\eps\Vert_{C^{0,\theta_p}(M -   B_\delta(\F))}\leq  C(p,\delta,k, L,\lambda_0) 
\Vert f\Vert_{L^p(M)}.
\end{equation*}
Since $\varphi_\eps\in H_{2k}^p(M)$ for all $p>1$, Green's representation formula  yields
\begin{equation}\label{eq:green:56}
\varphi_\eps(x)=\int_M G_\eps(x,y)f(y)\, dv_g(y)\hbox{ for all }x\in M- \{p_1,...,p_N\},
\end{equation}
and then when $R_{\F}(x)>\delta$, we get that
$$\left|\int_M G_\eps(x,y)f(y)\, dv_g(y)\right|\leq C(p,\delta,k, L,\lambda) 
\Vert f\Vert_{L^p(M)}$$
for all $f\in C^{0}(M)$ and $p\in \left(\frac{n}{2k},\frac{n}{2k-1}\right)$. Via duality, we then deduce that 
\begin{equation}\label{eq:16}
\Vert G_\eps(x,\cdot)\Vert_{L^q(M)}\leq C(q,\delta, k, L,\lambda_0)\hbox{ for all }q\in \left(1,\frac{n}{n-2k}\right)\hbox{ and }R_{\F}(x)>\delta. 
\end{equation}
We   fix $p\in M$ and $\gamma\in (0,n-2k)$. We let $\lambda_0=\lambda(\gamma,\crit, L,N,p,\delta)$ be given by Lemma \ref{lem:bis} and we take $0<\lambda<\lambda_0$. We take $f\in C^0(M)$ such that $f\equiv 0$ in $B_{\delta}(p)$, so that $(P-V_\eps)\varphi_\eps=0$ in $B_{\delta/2}(p)$. Since $(\F,V_\eps)\in \Pl(N)$ for all $\eps>0$ and $P\in O_{k,L}$, given $l_1\in\{0,...,2k-1\}$, the regularity Lemma \ref{lem:bis} yields $\lambda_\gamma>0$ such that for $\lambda<\lambda_\gamma$, we have that
\begin{eqnarray*}
R_{\F}(x)^{\gamma+l_1}|\nabla_x^{l_1}\varphi_\eps(x)|\leq C(p,\delta,k, L,\lambda_0) \Vert \varphi_\eps\Vert_{L^{\crit}(B_{\delta}(p))},
\end{eqnarray*}
for all $x\in B_{\delta/2}(p)- \{p_1,...,p_N\}$. With \eqref{eq:14}, we get that
\begin{equation*}
R_{\F}(x)^{\gamma+l_1}|\nabla_x^{l_1}\varphi_\eps(x)|\leq  C(p,\delta,k, L,\lambda_0) 
\Vert f\Vert_{L^{\frac{2n}{n+2k}}(M)}. 
\end{equation*}
With Green's representation \eqref{eq:green:56}, we then get that
$$\left|R_{\F}(x)^{\gamma+l_1}\int_M \nabla_x^{l_1}G_\eps(x,y)f(y)\, dv_g(y)\right|\leq C(p,\delta,k, L,\lambda) 
\Vert f\Vert_{L^{\frac{2n}{n+2k}}(M)}$$
for all $f\in C^{0}(M)$ vanishing in $B_{\delta}(p)$ and  $x\in B_{\delta/2}(p)- \{p_1,...,p_N\}$. Duality yields
\begin{equation}\label{eq:16:crit}
\Vert R_{\F}(x)^{\gamma+l_1} \nabla_x^{l_1} G_\eps(x,\cdot)\Vert_{L^{\crit}(M -  B_{\delta}(p))}\leq C(\delta, k, L,\lambda_0)
\end{equation}
for all $x\in B_{\delta/2}(p)- \{p_1,...,p_N\}$. For all $\eps>0$, we have that 
\begin{equation}\label{eq:Ge}
P G_\eps(x,\cdot)-V_\eps G_\eps(x,\cdot)=0\hbox{ in }M-\{x\}.
\end{equation}
It follows from the regularity Lemma \ref{lem:bis} (note that $V_\eps\in L^\infty(M)$) that there exists $C(\delta, k, L,\gamma,\lambda_0)>0$ such that for all $l_2=0,...,2k-1$,
\begin{equation*}
R_{\F}(y)^{\gamma+l_2} R_{\F}(x)^{\gamma+l_1} |\nabla_y^{l_2}\nabla_x^{l_1} G_\eps(x,y)|\leq C(\delta, k, L,\gamma,\lambda_0)
\end{equation*}
for all $x\in B_{\delta/2}(p)$ and $y\in M -  B_{2\delta}(p)$, $x\neq y$ and $x,y\not\in  \{p_1,...,p_N\}$. Via a finite covering,  there exists $\lambda_\gamma>0$ such that for $\lambda<\lambda_\gamma$, we have that for any $\delta>0$, there exists $C(\delta, k, L,\gamma,\lambda_0)>0$ such that 
\begin{equation}\label{ineq:G:1}
R_{\F}(y)^{\gamma+l_2} R_{\F}(x)^{\gamma+l_1} |\nabla_y^{l_2}\nabla_x^{l_1} G_\eps(x,y)|\leq C(\delta, k, L,\gamma,\lambda_0)
\end{equation}
for all $x,y\in M-\{p_1,...,p_N\}$ such that $d_g(x,y)\geq\delta$. 

\medskip\noindent{\bf Step 2: passing to the limit $\eps\to 0$ and Green's function for $P-V_0$.}\par
\noindent We fix $x\in M-\{p_1,...,p_N\}$. With \eqref{eq:Ge} and \eqref{ineq:G:1}, Ascoli's theorem yields $G_0(x,\cdot)\in C^{2k-1}(M-\{x,p_1,...,p_N\})$ such that, up to extraction,
\begin{equation}\label{lim:Ge}
\lim_{\eps\to 0}G_\eps(x,\cdot)=G_0(x,\cdot)\hbox{ in }C^{2k-1}_{loc}(M-\{x,p_1,...,p_N\}).
\end{equation}
By elliptic regularity again, we also get that
\begin{equation*}
\lim_{\eps\to 0}G_\eps(x,\cdot)=G_0(x,\cdot)\hbox{ in }H^p_{2k,loc}(M-\{x,p_1,...,p_N\})\hbox{ for all }p>1.
\end{equation*}
Via the monotone convergence theorem, passing to the limit in \eqref{eq:16}, we get that 
\begin{equation}\label{eq:17}
\Vert G_0(x,\cdot)\Vert_{L^q(M)}\leq C(q,\delta, k, L,\lambda_0)\hbox{ for all }q\in \left(1,\frac{n}{n-2k}\right)\hbox{ and }R_\F(x)>\delta,
\end{equation}
and then $G_0(x,\cdot)\in L^q(M)$ for all $q\in \left(1,\frac{n}{n-2k}\right)$ and $x \in M- \{p_1,...,p_N\}$. Similarly, using \eqref{eq:16:crit}, we get that
\begin{equation*}
\Vert G_0(x,\cdot)\Vert_{L^{\crit}(M -  B_{\delta}(x))}\leq C(\delta, k, L,\lambda_0)\hbox{ when }R_\F(x)>\delta. 
\end{equation*}
So that $G_0(x,\cdot)\in L^{\crit}_{loc}(M - \{x\})$ for all $x\in M-\{p_1,...,p_N\}$. Passing to the limit $\eps\to 0$ in \eqref{ineq:G:1} yields
\begin{equation}\label{ineq:G:2}
R_{\F}(y)^{\gamma+l_2} R_{\F}(x)^{\gamma+l_1} |\nabla_y^{l_2}\nabla_x^{l_1} G_0(x,y)|\leq C(\delta, k, L,\lambda_0)
\end{equation}
for all $x,y\in M-\{p_1,...,p_N\}$ such that $d_g(x,y)\geq\delta$.

\medskip\noindent{\bf Step 3: Representation formula.} We fix $f\in L^{\frac{2n}{n+2k}}(M)\cap L^p_{loc}(\Mmx)$, $p>\frac{n}{2k}>1$. Via the coercivity of $P-V_\eps$ and $P-V_0$ and using that $V_\eps\in L^\infty(M)$ for all $\eps>0$, it follows from variational methods (\cite{ADN} or Theorem D.4 of \cite{robert:gjms}) that there exists  $\varphi_\eps\in H_k^2(M)\cap H_{2k}^{\frac{2n}{n+2k}}(M)$ and $\varphi_0 \in H_k^2(M)$ such that 
\begin{equation}\label{eq:phieps:phi0}
(P-V_0)\varphi_0=f\hbox{ and }(P-V_\eps)\varphi_\eps=f\hbox{ in }M.
\end{equation}
As one checks (see for instance \cite{robert:gjms} for details), we have that
\begin{equation}\label{lim:phi:eps}
\lim_{\eps\to 0}\varphi_\eps=\varphi_0\hbox{ in }H_k^2(M).
\end{equation}
Since $\varphi_\eps\in H_{2k}^{\frac{2n}{n+2k}}(M)$ is a solution to \eqref{eq:phieps:phi0}, $P\in O_{k,L}$, $V_\eps\in \Pl(N)$ and $f\in L^p_{loc}(M -  \{p_1,...,p_N\})$, $p>\frac{n}{2k}$, it follows from regularity theory (\cite{ADN} or Theorems D.1 and D.2 of \cite{robert:gjms}) that $\varphi_\eps\in H_{2k,loc}^p(M - \{p_1,...,p_N\})$ and that for any $\delta>0$, using \eqref{lim:phi:eps} and as in \eqref{case:F:F0}, we get that
\begin{eqnarray*}
\Vert \varphi_\eps\Vert_{H_{2k}^p(M -  B_r(\F))}&\leq& C( r, k,L,\lambda_0)\left(\Vert f\Vert_{L^p(M -  B_{r/2}(\F))}+\Vert \varphi_\eps\Vert_{L^{\crit}(M -  B_{r/2}(\F))}\right)\\
&\leq& C( r, k,L,\lambda_0, f).
\end{eqnarray*}
Since $p>n/2k$, it follows from Sobolev's embedding theorem that $\varphi_\eps\in C^0(M - \{p_1,...,p_N\})$ and that
\begin{equation*}
\Vert \varphi_\eps\Vert_{C^0(M -  B_r(\F))}\leq C(k,r)\Vert \varphi_\eps\Vert_{H_{2k}^p(M -  B_r(\F))}\leq C( r, k,L,\lambda_0,f)
\end{equation*}
for all $\eps>0$. As one checks, (see again \cite{robert:gjms} for details, we then get that $\varphi_0\in C^0(M - \{p_1,...,p_N\})$ and 
\begin{equation}\label{lim:C0}
\lim_{\eps\to 0}\varphi_\eps=\varphi_0\hbox{ in }C^0_{loc}(\Mmx).
\end{equation}
With \eqref{eq:16}, \eqref{lim:Ge}, \eqref{eq:17}, \eqref{eq:16:crit}, \eqref{lim:C0}, passing to the limit in \eqref{eq:green:56} yields
$$\varphi_0(x)=\int_{M}G_0(x,\cdot)f\, dv_g.$$
This yields the existence of a Green's function for $P-V_0$ in Theorem \ref{th:Green:main}. Uniqueness goes as in \cite{robert:gjms}. This ends the proof of Theorem \ref{th:Green:main}.

\section{Asymptotics for the Green's function close to the singularity}\label{sec:green2}
This section is devoted to the proof of infinitesimal versions of \eqref{ineq:G:1} and \eqref{ineq:G:2} when  $x,y$ are close to the singular set $\F$.

\begin{theorem}\label{th:G:close} Let $(M,g)$ be a compact Riemannian manifold of dimension $n$. Fix $k\in\nn$ such that $2\leq 2k<n$, $L>0$, $\lambda>0$ and $N\geq 1$ an integer. Fix an operator $P$ of type $O_{k,L}$ (see Definition \ref{def:p}), $(\F,V_0)\in \Pl(N)$ where $\F:=\{p_1,...,p_N\}$ and a family $(V_\eps)$ as in \eqref{hyp:Ve}. For $\lambda>0$ sufficiently small, let $G_0$ (resp. $G_\eps$) be the Green's function for $P-V_0$ (resp. $P-V_\eps$). 

\smallskip\noindent Let us fix $U, V$ two open subsets of $\rn$ such that $ \overline{U}\cap\overline{V}=\emptyset$. We let $\mu_0:=\mu_0(M,g,U,V)>0$ be such that $|\mu X|<i_g(M)/4$ for all $0<\mu<\mu_0$ and $X\in U\cup V$.  

\smallskip\noindent We fix $i\in\{1,...,N\}$. We let $J_i:=\{j\in \{1,...,N\}/\, d_g(p_j,p_i)<i_g(M)/2\}$ (note that $i\in I_i$). For $j\in J_i$, we define $\tilde{p}_{j}:=\frac{\hbox{exp}_{p_i}^{-1}(p_j)}{\mu}$ and 
$$\tilde{R}_{i,\F}(X):=\inf_{j\in J_i}|X-\tilde{p}_j|\hbox{ for all }X\in\rn.$$
We fix $\gamma\in (0, n-2k)$. Then there exists $\lambda=\lambda(\gamma)>0$, there exists $C=C(U,V, M,g,\lambda, k, L,N,\gamma)>0$ such that for any $0\leq l_1,l_2\leq 2k-1$, 
\begin{equation}\label{ineq:G:close}
\left| \tilde{R}_{i,\F}(X)^{\gamma+l_1} \tilde{R}_{i,\F}(Y)^{\gamma+l_2} \mu^{n-2k+l_1+l_2}\nabla_X^{l_1}\nabla_Y^{l_2}G_\eps( \hbox{exp}_{p_i}(\mu X),\hbox{exp}_{p_i}(\mu Y) )\right|\leq C
\end{equation}
for all $X\in U-\{\tilde{p}_j/\, j\in J_i\}$, $Y\in V-\{\tilde{p}_j/\, j\in J_i\}$, $\mu\in (0,\mu_0)$ and $\eps\geq 0$.
\end{theorem}

\noindent{\it Proof of Theorem \ref{th:G:close}.} We first take $\eps>0$ small in order to have that $V_\eps\in L^\infty(M)$: we will pass to the limit $\eps\to 0$ at the end of the argument. We first set $U',V'$ two open subsets of $\rn$ such that
\begin{equation*}
 U\subset\subset U'\subset\subset \rn \, ,\, V\subset\subset V'\subset\subset \rn\hbox{ and }\overline{U'}\cap\overline{V'}=\emptyset,
\end{equation*}
and $\mu_1:=\mu_1(M,g,U',V')$ be such that $|\mu X|<i_g(M)/4$ for all $0<\mu<\mu_1$ and $X\in U'\cup V'$. Note that for $\mu\in [\mu_1,\mu_0)$, \eqref{ineq:G:close} for $\eps>0$ is a consequence of \eqref{ineq:G:1}. We fix $f\in C^\infty_c(V')$ and for any $0<\mu<\mu_1$, we set 
$$f_\mu(x):=\frac{1}{\mu^{\frac{n+2k}{2}}}f\left(\frac{\hbox{exp}_{p_i}^{-1}(x)}{\mu}\right)\hbox{ for all }x\in M.$$
As one checks, $f_\mu\in C^\infty_c(\hbox{exp}_{p_i}(\mu V'))\subset C^\infty(M)$. Since $V_\eps\in L^\infty(M)$, it follows from elliptic regularity (\cite{ADN} or Theorem D.4 of \cite{robert:gjms})  that there exists $\varphi_{\mu,\eps}\in H_{2k}^q(M)$ for all $q>1$ such that
\begin{equation}\label{eq:phi:alpha}
P\varphi_{\mu,\eps}-V_\eps\varphi_{\mu,\eps}=f_\mu\hbox{ in }M.
\end{equation}
It follows from Sobolev's embedding theorem that $\varphi_{\mu,\eps}\in C^{2k-1}(M)$. We define
\begin{equation*}
\tilde{\varphi}_{\mu,\eps}(X):= \mu^{\frac{n-2k}{2}}\varphi_{\mu,\eps}\left(\hbox{exp}_{p_i}(\mu X)\right)\hbox{ for all }X\in \rn ,\, |\mu X|<i_g(M).
\end{equation*}
A change of variable and upper-bounds for the metric yield
\begin{eqnarray*}
\Vert f_\mu\Vert_{\frac{2n}{n+2k}}^{\frac{2n}{n+2k}}&=&\int_M |f_\mu(x)|^{\frac{2n}{n+2k}}\, dv_g=\int_{\hbox{exp}_{p_i}(\mu V')} |f_\mu(x)|^{\frac{2n}{n+2k}}\, dv_g\\
&=& \int_{V'} |f(X)|^{\frac{2n}{n+2k}}\, dv_{g_\mu}(X)\leq C\int_{V'} |f(X)|^{\frac{2n}{n+2k}}\, dX,
\end{eqnarray*}
where $\hbox{exp}_p^\star g$ denotes the pull-back metric of $g$ and  $g_\mu:=(\hbox{exp}_{p_i}^\star g)(\mu\cdot)$. Therefore
\begin{equation}\label{eq:31}
\Vert f_\mu\Vert_{L^{\frac{2n}{n+2k}}(M)}\leq C(k)\Vert f\Vert_{L^{\frac{2n}{n+2k}}(V')},
\end{equation}
where $C(k)$ depends only on $(M,g)$ and $k$. With \eqref{coer:PV}, \eqref{eq:phi:alpha} and the Sobolev inequality \eqref{sobo:ineq:M}, we get that
\begin{eqnarray*}
&&\frac{1}{2L}\Vert \varphi_{\mu,\eps}\Vert_{H_k^2(M)}^2\leq \int_M \varphi_{\mu,\eps}(P-V_\eps)\varphi_{\mu,\eps}\, dv_g=\int_M f_\mu\varphi_{\mu,\eps} \, dv_g\\
&&\leq  \Vert f_\mu\Vert_{L^{\frac{2n}{n+2k}}(M)} \Vert \varphi_{\mu,\eps}\Vert_{L^{\frac{2n}{n-2k}}(M)}\leq C_S(k) \Vert f_\mu\Vert_{L^{\frac{2n}{n+2k}}(M)} \Vert \varphi_{\mu,\eps}\Vert_{H_k^2(M)}.
\end{eqnarray*}
Therefore, using again the Sobolev inequality \eqref{sobo:ineq:M} and \eqref{eq:31}, we get that
\begin{equation}\label{eq:33}
\Vert \varphi_{\mu,\eps}\Vert_{L^{\frac{2n}{n-2k}}(M)}\leq C(k,L) \Vert f\Vert_{L^{\frac{2n}{n+2k}}(V')}.
\end{equation}
With $g_\mu$ as above, equation \eqref{eq:phi:alpha} rewrites
\begin{equation}\label{eq:tpl}
\Delta_{g_\mu}^k\tilde{\varphi}_{\mu,\eps}+\sum_{l=0}^{2k-2}\mu^{2k-l}B_l(\hbox{exp}_{p_i}(\mu \cdot))\star\nabla_{g_\mu}^l\tilde{\varphi}_{\mu,\eps}-\mu^{2k}V_\eps(\hbox{exp}_{p_i}(\mu X))\tilde{\varphi}_{\mu,\eps}=f
\end{equation}
weakly locally in $\rn$ where the $B_l$'s are $(l,0)-$tensors that are bounded in $L^\infty$ due to Definition \ref{def:p}. Since $V_\eps$ satisfies \eqref{hyp:Ve}, using \eqref{lem:lip}, for $\mu<\tilde{\mu}_1$ small, we have that 
\begin{equation*}
\left|\mu^{2k}V_\eps(\hbox{exp}_{p_i}(\mu X))\right|\leq 2^{2k}\lambda \tilde{R}_{i,\F}(X)^{-2k}\hbox{ for all }X\in U'-\{\tilde{p}_j/\, j\in J_i\}.
\end{equation*}
Since $f(X)=0$ for all $X\in U'$, $U\subset\subset U'$ and $\tilde{\varphi}_{\mu,\eps}\in H_{2k,loc}^q(\rn)$, it follows from the regularity Lemma \ref{lem:bis} that there exists $\lambda=\lambda(\gamma)>0$, there exists $C( L, \gamma,U, U')>0$ such that for any $0\leq l_1\leq 2k-1$, we have that
\begin{equation}\label{eq:44}
\tilde{R}_{i,\F}(X)^{\gamma+l_1} \left|\nabla^{l_1}\tilde{\varphi}_{\mu,\eps}(X)\right|\leq 
C( L,\gamma, U,U')\Vert \tilde{\varphi}_{\mu,\eps} \Vert_{L^{\crit}(U')}
\end{equation}
for all $X\in U-\{\tilde{p}_j/\, j\in J_i\}$. Arguing as in the proof of \eqref{eq:31}, we have that
\begin{equation}\label{eq:32}
\Vert \tilde{\varphi}_{\mu,\eps}\Vert_{L^{\crit}(U')}\leq C(k)\Vert  \varphi_{\mu,\eps}\Vert_{L^{\frac{2n}{n-2k}}(M)}.
\end{equation}
For $\mu>0$, we define
\begin{equation}\label{def:G:t:close}
\tilde{G}_{\mu,\eps}(X,Y):=\mu^{n-2k}G_\eps(\hbox{exp}_{p_i}(\mu X),\hbox{exp}_{p_i}(\mu Y) )\hbox{ for }(X,Y)\in U'\times V',
\end{equation}
$X,Y\not\in\{\tilde{p}_j/\, j\in J_i\}$.  Green's representation formula for $G_\eps$, $\eps>0$, and \eqref{eq:phi:alpha} yield
$$\varphi_{\mu,\eps}\left(\hbox{exp}_{p_i}(\mu X)\right)=\int_M G_\eps\left(\hbox{exp}_{p_i}(\mu X),y\right)f_\mu(y)\, dv_g(y)$$
for all $X\in U-\{\tilde{p}_j/\, j\in J_i\}$. 
With a change of variable, we then get that
\begin{equation}\label{ineq:G:b}
\tilde{\varphi}_{\mu,\eps}(X)=\int_{V'} \tilde{G}_{\mu,\eps}(X,Y) f(Y)\, dv_{g_\mu}
\end{equation}
for all $X\in U-\{\tilde{p}_j/\, j\in J_i\}$. Putting together   \eqref{eq:33}, \eqref{eq:44} and \eqref{eq:32}  and \eqref{ineq:G:b}, we get that
\begin{equation*}
\left|\tilde{R}_{i,\F}(X)^{\gamma+l_1}\int_{U'} \nabla_X^{l_1}\tilde{G}_{\mu,\eps}(X,Y) f(Y)\, dv_{g_\mu}
\right|\leq C( L, \delta, \lambda,\gamma,U,U',V') \Vert f\Vert_{L^{\frac{2n}{n+2k}}(V')}\end{equation*}
for all $f\in C^\infty_c(V')$ and $ X\in U-\{\tilde{p}_j/\, j\in J_i\}$. Duality arguments yield
\begin{equation}\label{ineq:35}
\Vert \tilde{R}_{i,\F}(X)^{\gamma+l_1}\nabla_X^{l_1} \tilde{G}_{\mu,\eps}(X,\cdot)\Vert_{L^{\crit}(V')}\leq C( L, \delta, \lambda,\gamma,U,U',V')
\end{equation}
for $X\in U-\{\tilde{p}_j/\, j\in J_i\}$. Since $G_\eps(x,\cdot)$ is a solution to $(P-V_\eps)G_\eps(x,\cdot)=0$ in $M-\{x, p_1,...,p_N\}$, as in \eqref{eq:tpl},
we get that
\begin{eqnarray*}
&&\Delta_{g_\mu}^k\tilde{G}_{\mu,\eps}(X,\cdot)+\sum_{l=0}^{2k-2}\mu^{2k-l}B_l(\hbox{exp}_{p_i}(\mu \cdot))\star\nabla_{g_\mu}^l\tilde{G}_{\mu,\eps}(X,\cdot)\\
&&-\mu^{2k}V_\eps(\hbox{exp}_{p_i}(\mu \cdot))\tilde{G}_{\mu,\eps}(X,\cdot)=0\hbox{ weakly in }V',
\end{eqnarray*}
and $\tilde{G}_{\mu,\eps}(X,\cdot)\in H_{2k,loc}^q(V')$ for some $q>1$ since $X\not\in V'$. As above, we have that $\left|\mu^{2k}V_\eps(\hbox{exp}_{p_i}(\mu Y))\right|\leq C \lambda \tilde{R}_{i,\F}(Y)^{-2k}$ for all $Y\in V'-\{\tilde{p}_j/\, j\in J_i\}$. The regularity Lemma \ref{lem:bis} then yields that there exists $C=C(k,L, \lambda,U,V,U',V')$ such that for any $0\leq l_2\leq 2k-1$, we have that
\begin{equation*}
\tilde{R}_{i,\F}(Y)^{\gamma+l_2} \tilde{R}_{i,\F}(X)^{\gamma+l_1} |\nabla_{Y}^{l_2}\nabla_{X}^{l_1}\tilde{G}_{\mu,\eps}(X,Y)|\leq C\Vert \tilde{R}_{i,\F}(X)^{\gamma+l_1} \nabla_X^{l_1}\tilde{G}_{\mu,\eps}(X,\cdot)\Vert_{L^{\crit}(V')}
\end{equation*}
for all $Y\in V-\{\tilde{p}_j/\, j\in J_i\}$ and $X\in U-\{\tilde{p}_j/\, j\in J_i\}$. The conclusion \eqref{ineq:G:close} for $\eps>0$ of Theorem \ref{th:G:close} then follows from this inequality, \eqref{ineq:35}, Definition \eqref{def:G:t:close} of $\tilde{G}_{\mu,\eps}$. The case $\eps=0$ then follows from \eqref{lim:Ge}. This proves Theorem \ref{th:G:close}.

\section{Asymptotics for the Green's function far from the singularity}\label{sec:green3}
This section is devoted to the proof of an infinitesimal version of \eqref{ineq:G:1} and \eqref{ineq:G:2} when $x,y$ are close to each other and far from the singularity set  $\F$. For any bounded domain $\Omega\subset \rn$, we let $R_\Omega>0$ the smallest real number such that $\Omega\subset B_{R_\Omega}(0)$.

\begin{theorem}\label{th:G:far} We fix $p\in M-\{p_1,...,p_N\}$ and $U,V$ two open subsets of $\rn$ such that $U\subset\subset \rn$, $V\subset\subset \rn$ and $\overline{U}\cap\overline{V}=\emptyset$. We define
\begin{equation}\label{hyp:alpha}
\mu<\mu_p:=\frac{\min\{i_g(M),R_{\F}(p)\}}{8R_U+8R_V+2}.
\end{equation}
Then for all $\gamma\in (0,n-2k)$ and $\tau\in (0,1)$, there exists $\lambda=\lambda(\gamma)>0$, there exists a constant $C=C(U,V, M,g,\lambda, k, L,N,\gamma)>0$ such that for any $0\leq l_1,l_2\leq 2k-1$, we have that
\begin{equation}\label{ineq:G:far}
\left|  \mu^{n-2k+l_1+l_2}\nabla_X^{l_1}\nabla_Y^{l_2}G_\eps(\hbox{exp}_{p}(\mu X),\hbox{exp}_{p}(\mu Y) )\right|\leq C
\end{equation}
for all $X\in U $ and $Y\in V$ and $0<\mu<\tau\mu_p$ and $\eps\geq 0$ small enough.
\end{theorem}

\noindent{\it Proof of Theorem \ref{th:G:far}.} As in the proof of Theorem \ref{th:G:close}, we take $\eps>0$. We first set $U',V'$ two open subsets of $\rn$ such that
\begin{equation*}
U\subset\subset U'\subset\subset \rn \, ,\, V\subset\subset V'\subset\subset \rn\, ,\, \tau<\frac{8R_U+8R_V+2}{8R_{U'}+8R_{V'}+2} \hbox{ and }\overline{U'}\cap\overline{V'}=\emptyset.
\end{equation*}
We take $0<\mu<\mu_p$. We fix $f\in C^\infty_c(V')$ and for any $0<\mu<\tau\mu_p$, we set 
$$f_\mu(x):=\frac{1}{\mu^{\frac{n+2k}{2}}}f\left(\frac{\hbox{exp}_{p}^{-1}(x)}{\mu}\right)\hbox{ for all }x\in M.$$
As one checks, $f_\mu\in C^\infty_c(\hbox{exp}_{p}(\mu V'))$ and $\hbox{exp}_{p}(\mu V')\subset\subset M-\{p_1,...,p_N\}$. It follows from elliptic regularity (\cite{ADN} or Theorem D.4 of \cite{robert:gjms}) that there exists $\varphi_{\mu,\eps}\in H_{2k}^q(M)$ for all $q>1$ such that
\begin{equation}\label{eq:phi:alpha:far}
P\varphi_{\mu,\eps}-V_\eps\varphi_{\mu,\eps}=f_\mu\hbox{ in }M.
\end{equation}
It follows from Sobolev's embedding theorem that $\varphi_{\mu,\eps}\in C^{2k-1}(M)$. We define
\begin{equation}\label{def:tilde:phi:far}
\tilde{\varphi}_{\mu,\eps}(X):= \mu^{\frac{n-2k}{2}}\varphi_{\mu,\eps}\left(\hbox{exp}_{p}(\mu X)\right)\hbox{ for all }X\in \rn,\, |\mu X|<i_g(M).
\end{equation}
As in \eqref{eq:31}, a change of variable yields $\Vert f_\mu\Vert_{L^{\frac{2n}{n+2k}}(M)}\leq C(k)\Vert f\Vert_{L^{\frac{2n}{n+2k}}(V')}$. As for \eqref{eq:33}, we also get that
\begin{equation}\label{eq:33:far}
\Vert \varphi_{\mu,\eps}\Vert_{L^{\frac{2n}{n-2k}}(M)}\leq C(k,L) \Vert f\Vert_{L^{\frac{2n}{n+2k}}(V')}. 
\end{equation}
Taking $g_\mu:=(\hbox{exp}_{p}^\star g)(\mu\cdot)$, equation \eqref{eq:phi:alpha:far} rewrites
\begin{equation*}
\Delta_{g_\mu}^k\tilde{\varphi}_{\mu,\eps}+\sum_{l=0}^{2k-2}\mu^{2k-l}B_l(\hbox{exp}_{p}(\mu \cdot))\star\nabla_{g_\mu}^l\tilde{\varphi}_{\mu,\eps}-\mu^{2k}V_\eps(\hbox{exp}_{p}(\mu X))\tilde{\varphi}_{\mu,\eps}=f
\end{equation*}
weakly in $\rn$ where the $B_l$'s are $(l,0)-$tensors that are bounded in $L^\infty$ due to Definition \ref{def:p}. Since $V_\eps$ satisfies \eqref{hyp:Ve}, we have that
\begin{equation*}
\left|\mu^{2k}V_\eps(\hbox{exp}_{p}(\mu X))\right|\leq \lambda \mu^{2k}\min\{d_g(p_i,\hbox{exp}_{p}(\mu X) \}^{-2k}\hbox{ for all }X\in U'.
\end{equation*}
With \eqref{hyp:alpha}, we have that
$$d_g(p_i,(\hbox{exp}_{p}(\mu X) )\geq d_g(p_i,p)-\mu|X|\geq \frac{1}{2}d_g(p_i,p)\geq \frac{1}{2}R_{\F}(p)$$
for all $X\in V'$, and therefore, we get that 
\begin{equation*}
\left|\mu^{2k}V_\eps(\hbox{exp}_{p}(\mu X))\right|\leq \lambda\left(\frac{R_{\F}(p)}{\mu}\right)^{-2k} \leq C(\lambda_0)\hbox{ for all }X\in U' .
\end{equation*}
Since $f(X)=0$ for all $X\in U'$, it follows from standard regularity theory (\cite{ADN} or Theorem D.2 of \cite{robert:gjms}) that there exists $C( k,L,  U,U',V')>0$ such that
\begin{equation}\label{eq:64}
  \left|\nabla^{l_1}\tilde{\varphi}_{\mu,\eps}(X)\right|\leq C\Vert \tilde{\varphi}_{\mu,\eps} \Vert_{L^{\crit}(U')}\hbox{ for all }X\in U.
\end{equation}
Arguing as in the proof of \eqref{eq:31}, we have that
\begin{equation}\label{eq:32:far}
\Vert \tilde{\varphi}_{\mu,\eps}\Vert_{L^{\frac{2n}{n-2k}}(U')}\leq C(k)\Vert  \varphi_{\mu,\eps}\Vert_{L^{\frac{2n}{n-2k}}(M)}.
\end{equation}
Putting together \eqref{def:tilde:phi:far}, \eqref{eq:64}, \eqref{eq:32:far} and \eqref{eq:33:far} we get that
\begin{equation*}
  \left| \nabla^{l_1}\tilde{\varphi}_{\mu,\eps}(X) \right|\leq C( k,L, U,U',V',\lambda) \Vert f\Vert_{L^{\frac{2n}{n+2k}}(V')} \hbox{ for all }X\in U.
\end{equation*}
We now just follow verbatim the proof of Theorem \ref{th:G:close} above to get the conclusion \eqref{ineq:G:far}  of Theorem \ref{th:G:far}.

\section{Proof of Theorem \ref{th:Green:pointwise:BIS}}\label{sec:green4}
We prove here the pointwise estimate of Theorem \ref{th:Green:pointwise:BIS}. For the reader's convenience, we only consider the case $l_1=l_2=0$. We fix $\gamma\in (0, n-2k)$ and $\lambda>0$ given by \eqref{ineq:G:1}, Theorems \ref{th:G:close} and \ref{th:G:far}. We argue by contradiction and we assume that there is a family of operators $(P_l)_{l\in\nn}\in O_{k,L}$, a family of potentials $(\F_l,V_l)_{l\in\nn}\in {\mathcal P}_{\lambda_\gamma}(N)$, sequences $(x_l), (y_l)\in M-\{p_1,..,p_N\}$ such that $x_l\neq y_l$ for all $l\in\nn$ and
\begin{equation}\label{hyp:absurd}
\lim_{l\to +\infty}\frac{d_g(x_l,y_l)^{n-2k}|G_l(x_l,y_l)|}{\max\left(1, \frac{d_g(x_l,y_l)^2}{R_{\F_l}(x_l)R_{\F_l}(y_l)} \right)^\gamma }=+\infty,
\end{equation}
where $G_l$ is the Green's function of $P_l-V_l$ for all $l\in\nn$. We distinguish 3 cases:

\smallskip\noindent{\it Case 1:} $d_g(x_l,y_l)\not\to 0$ as $l\to +\infty$. In this case, noting that  \eqref{hyp:absurd} rewrites  $\lim_{l\to +\infty}R_{\F_l}(x_l)^\gamma R_{\F_l}(y_l)^\gamma |G_l(x_l,y_l)| =+\infty$, we then get a contradiction to \eqref{ineq:G:1}.

\smallskip\noindent{\it Case 2:} $d_g(x_l, y_l)=o(R_{\F_l}(x_l))$ as $l\to +\infty$. The triangle inequality yields
\begin{equation}\label{ineq:tri}
|R_{\F_l}(x)-R_{\F_l}(y)|\leq d_g(x,y)\hbox{ for all }x,y\in M.
\end{equation}
We then get that $d_g(x_l,y_l)=o(R_{\F_l}(y_l))$. Therefore
\eqref{hyp:absurd} rewrites
\begin{equation}\label{hyp:absurd:1}
\lim_{l\to +\infty} d_g(x_l,y_l)^{n-2k}|G_l(x_l,y_l)|=+\infty.
\end{equation}
We let $Y_l\in\rn$ be such that $y_l:=\hbox{exp}_{x_l}(d_g(x_l, y_l)Y_l)$. In particular, $|Y_l|=1$, so, up to a subsequence, there exists $Y_\infty\in\rn$ such that $\lim_{l\to +\infty}Y_l=Y_\infty$ with $|Y_\infty|=1$
We apply Theorem \ref{th:G:far} with $p:=x_l$, $\mu:=d_g(x_l, y_l)$, $U=B_{1/3}(0)$, $V=B_{1/3}(Y_\infty)$: for $l\in\nn$ large enough, taking $X=0$ and $Y=Y_l$ in \eqref{ineq:G:far}, we get that
\begin{eqnarray*}
&&d_g(x_l, y_l)^{n-2k}|G_l(x_l,y_l )|\\
&&=d_g(x_l, y_l)^{n-2k}|G_l(\hbox{exp}_{x_l}(d_g(x_l, y_l)\cdot 0),\hbox{exp}_{x_l}(d_g(x_l, y_l)\cdot Y_l) )|\leq C
\end{eqnarray*}
which contradicts \eqref{hyp:absurd:1}. This ends Case 2.

\medskip\noindent{\it Case 3:} $R_{\F}(x_l)=O(d_g(x_l, y_l))$ and $d_g(x_l, y_l)\to 0$ as $l\to +\infty$. The triangle inequality \eqref{ineq:tri} then yields $R_{\F}(y_l)=O(d_g(x_l,y_l))$. Therefore \eqref{hyp:absurd} rewrites
\begin{equation}\label{hyp:absurd:2}
\lim_{l\to +\infty}   d_g(y_l, x_l)^{n-2k-2\gamma}R_{\F}(x_l)^\gamma R_{\F}(y_l)^{\gamma}|G_l(x_l,y_l)| =+\infty.
\end{equation}
We let $i\in\{1,...,N\}$ be such that $R_{\F}(x_l)=d_g(x_l, p_i)$ for all $l\in\nn$: this is possible up to extraction since $N$ is fixed.  We then get that $d_g(x_l, p_i)=O(d_g(x_l, y_l))$ and then $d_g(y_l, p_i)\leq d_g(y_l, x_l)+ d_g(x_l, p_i)=O(d_g(x_l, y_l))$. We let $X_l, Y_l\in\rn$ be such that  $x_l:=\hbox{exp}_{p_i}(d_g(x_l, y_l)X_l)$ and  $y_l:=\hbox{exp}_{p_i}(d_g(x_l, y_l)Y_l)$. In particular, $ |X_l|=O(1)$ and $|Y_l|=O(1)$ as $l\to +\infty$ and with \eqref{lem:lip}, we get that
$$d_g(x_l, y_l)=d_g(\hbox{exp}_{p_i}(d_g(x_l, y_l)X_l),\hbox{exp}_{p_i}(d_g(x_l, y_l)Y_l))\leq 2d_g(x_l, y_l) |X_l-Y_l|,$$
and then $|X_l-Y_l|\geq 1/2$ as $l\to +\infty$. So, up to a subsequence, there exists $X_\infty,Y_\infty\in\rn$ such that  $\lim_{l\to +\infty}X_l=X_\infty$ and $\lim_{l\to +\infty}Y_l=Y_\infty$ with $|X_\infty-Y_\infty|\geq 1/2$. We apply Theorem \ref{th:G:close} with  $\mu:=d_g(x_l, y_l)$, $U=B_{1/6}(X_\infty)$ and $V=B_{1/6}(Y_\infty)$. We also adopt the notations of Theorem \ref{th:G:close}. So, for $l\in\nn$ large enough, taking $X=X_l$ and $Y=Y_l$ in \eqref{ineq:G:close}, we get that
$$\tilde{R}_{i, \F_l}(X_l)^\gamma \tilde{R}_{i, \F_l}(Y_l)^\gamma d_g(x_l,y_l)^{n-2k}|G_l(\hbox{exp}_{p_i}(d_g(x_l, y_l)X_l),\hbox{exp}_{p_i}(d_g(x_l, y_l)Y_l))|\leq C,$$
and, coming back to the definitions of $X_l$, $Y_l$ and $\tilde{R}_{i, \F_l}$, we get that
$$ R_{ \F_l}(x_l)^\gamma R_{ \F_l}(x_l)^\gamma  d_g(x_l,y_l)^{n-2k-2\gamma}|G_l(x_l,y_l)|\leq C,$$
which contradicts \eqref{hyp:absurd:2}. This ends Case 3.

\medskip\noindent Therefore, in all 3 cases, we have obtained a contradiction with \eqref{hyp:absurd}. This proves Theorem \ref{th:Green:pointwise:BIS} in the case $l_1=l_2=0$. The proof is identical for $l_1,l_2\leq 2k-1$. 

\section{A regularity Lemma}\label{sec:lemma}
\begin{defi}\label{def:P:loc} Let $(X,g)$ be a Riemannian manifold of dimension $n$ and let us fix a subdomain $\Omega\subset Int(X)$, $k\in\nn$ such that $2\leq 2k<n$ and $ L>0$. We say that an operator $P$ is of type $O_{k,L}(\Omega)$ if 
\begin{itemize}
\item[(i)] $P:=\Delta_g^k+\sum_{i=0}^{k-1} (-1)^i\nabla^i(A^{(i)} \nabla^i)$, where for all $i=0,...,k-1$, $ A^{(i)}\in C^{i}_{\chi}(M,\Lambda_{S}^{(0,2i)}(\Omega)$ is a family $C^{i}-$field of $(0,2i)-$tensors on $\Omega$;
\item[(ii)] For any $i$ , $A^{(i)}(T,S)=A^{(i)}(S,T)$ for any $(i,0)-$tensors $S$ and $T$; \par
\item[(iii)]  $\Vert A^{(i)}\Vert_{C^{i}}\leq L$ for all $i=0,...,k-1$.\par
\end{itemize}
\end{defi}
\begin{lemma}\label{lem:bis} Let $(X,g)$ be a Riemannian manifold  of dimension $n$ and $k\in\nn$ be such that $2\leq 2k<n$ and $L>0$. Fix $p>1$, $N\in\nn$. We fix  a domain $\Omega\subset Int(X)$ and a subdomain $\omega\subset\subset \Omega$. We consider  $P\in O_{k,L}(\Omega)$.
Then for all $0<\gamma<n-2k$, there exists $\lambda=\lambda(\gamma,p, L,N,\Omega, \omega)> 0$ and $C_0=C_0(X,g,\gamma,p, L,N,\Omega,\omega)>0$ such that for any $p_1,...,p_N\in X$ and $V\in L^1(\Omega)$ such that
$$|V(x)|\leq \lambda R(x)^{-2k}\hbox{ for all }x\in \Omega-\{p_1,...,p_N\},$$
where $R(x):=\min\{d_g(x,p_i)/\, i=1,...,N\}$, 
then for any $\varphi\in H_k^2(\Omega)\cap H_{2k,loc}^s(\Omega-\{p_i/\, i=1,...,N\})$ (for some $s>1$)  such that
$$P\varphi-V\cdot\varphi=0\hbox{ weakly in }H_k^2(\Omega),$$
and $V\in L^\infty(\Omega)$, we have that
\begin{equation}\label{ineq:lem:bis}
R(x)^{\gamma}|\varphi(x)|\leq C_0\cdot \Vert\varphi\Vert_{L^p(\Omega)}\hbox{ for all }x\in \omega-\{p_i/\, i=1,...,N\},
\end{equation}
and for any $0< l<2k$, there exists $C_l=C_l(X,g,\gamma,p, L,N,\Omega,\omega)>0$ such that
\begin{equation}\label{ineq:lem:l:bis}
R(x)^{\gamma+l}|\nabla^l\varphi(x)|\leq C_l\cdot \Vert\varphi\Vert_{L^p(\Omega)}\hbox{ for all }x\in \omega-\{p_i/\, i=1,...,N\}.
\end{equation}
\end{lemma}
The sequel of this section is devoted to the proof of this Lemma, first in a particular case and then in the general case. We will often refer to \cite{robert:gjms} where the case $N=1$ was considered.

\medskip\noindent{\bf Case 1: Proof of the Lemma for $l=0$ and under an extra assumption.} We take $l=0$ and we assume that there exists  $U\subset\subset\Omega$ such that
\begin{equation}\label{hyp:lemma:1}
p_1,...,p_N\in U.
\end{equation}
We let $W$ be an open subset such that $\omega, U\subset\subset W\subset\subset \Omega$. We prove the result by contradiction and we assume that \eqref{ineq:lem:bis} does not hold. Then, as in \cite{robert:gjms}, for all $\lambda>0$ small enough,  there exists $P_\lambda \in O_{k,L}(\Omega)$, $V_\lambda\in L^1(\Omega)$, $p_{i,\lambda}\in U$ for $i=1,...,N$,   $R_\lambda(x):=\min\{d_g(x,p_{i,\lambda})/\, i=1,..,N\}$ and  $\psi_\lambda\in H_k^2(\Omega)$  such that
\begin{equation}\label{eq:75}
\left\{\begin{array}{l}
(P_\lambda-V_\lambda)\psi_\lambda=0 \hbox{ weakly in }H_k^2(\Omega)\cap H_{2k,loc}^s(\Omega-\{p_{i,\lambda}/\, i=1,..,N\})\\
\Vert\psi_\lambda\Vert_{L^p(\Omega)}=1\\
|V_{\lambda}(x)|\leq \lambda R_\lambda(x)^{-2k}\hbox{ for all }x\in \Omega-\{p_{i,\lambda}/\, i=1,..,N\}\\
V_{\lambda}\in L^\infty(\Omega)\\
\sup_{x\in \overline{W}} R_\lambda(x)^{\gamma}|\psi_\lambda(x)|>\frac{1}{\lambda}\to +\infty\hbox{ as }\lambda\to 0.
\end{array}\right\}
\end{equation} 
Without loss of generality, we assume that there are $p_{i,0},..., p_{N,0}\in \overline{U}\subset W$ such that $\lim_{\lambda\to 0}p_{i,\lambda}=p_{i,0}$. Since $V_\lambda\in L^\infty(\Omega)$, by regularity theory (\cite{ADN} or Theorem  D.1 of \cite{robert:gjms}), we get that $\psi_\lambda\in C^0(\overline{W})$. Therefore, there exists $x_\lambda\in \overline{W}$ such that
\begin{equation}\label{lim:xlambda}
R_\lambda(x_\lambda)^{\gamma}|\psi_\lambda(x_\lambda)|=\sup_{x\in \overline{W}} R_\lambda(x)^{\gamma}|\psi_\lambda(x)|>\frac{1}{\lambda}\to +\infty
\end{equation}
as $\lambda\to 0$. Up to extraction, we assume that for all $i=1,..,N$, $\lim_{\lambda\to 0}p_{i,\lambda}=p_i\in \overline{U}\subset W$.

\medskip\noindent{\bf Step 1.1:} we claim that there exists $i_0\in 1,...,N$ such that $\lim_{\lambda\to 0}x_\lambda=p_{i_0}$ and $R_\lambda(x_\lambda)=d_g(x_\lambda, p_{i_0,\lambda})$ for $\lambda\to 0$. 

\smallskip\noindent We prove the claim. Up to extraction, there exists $i_0\in \{1,...,N\}$ such that $R_\lambda(x_\lambda)=d_g(x_\lambda, p_{i_0,\lambda})$ for $\lambda\to 0$. For any $r>0$, we have that $|V_\lambda(x)|\leq \lambda r^{-2k}$ for all $x\in \Omega- \cup_{i=1}^NB_{2r}(p_{i,0})$. So, with regularity theory (\cite{ADN} or Theorems D.1 and D.2 of \cite{robert:gjms}), we get that for all $q>1$, then $\Vert \psi_\lambda\Vert_{H_{2k}^q(W- \cup_{i=1}^N B_{3r}(p_{i,0}))}\leq C(r, q, L,p)\Vert\psi_\lambda\Vert_{L^p(\Omega)}=C(r,q, L,p)$. Taking $q>\frac{n}{2k}$, we get that $|\psi_\lambda(x)|\leq C(r,q, L,p)$ for all $x\in W- \cup_{i=1}^NB_{3r}(p_{i,0})$. Since $R_\lambda(x_\lambda)$ is bounded, it follows from \eqref{lim:xlambda} that $\lim_{\lambda\to 0}d_g(x_\lambda, p_{i_0,\lambda})=\lim_{\lambda\to 0}R_\lambda(x_\lambda)=0$. The claim is proved.

\medskip\noindent{\bf Step 1.2: Convergence after rescaling.} We set $r_\lambda:=d_g(x_\lambda,p_{i_0,\lambda})>0$.  Since $p_{i_0,\lambda}\in U\subset\subset W$, there exists $\delta>0$ independent of the $p_{i,\lambda}$'s and $\lambda$ such that $B_\delta(p_{i_0,\lambda})\subset W$. We define
\begin{equation*}
\tilde{\psi}_\lambda(X):=\frac{\psi_\lambda(\hbox{exp}_{p_{i_0,\lambda}}(r_\lambda X))}{\psi_\lambda(x_\lambda)}\hbox{ for }X\in\rn \hbox{ such that }|X|<\frac{\delta}{r_\lambda}.
\end{equation*}
We define $\tilde{x}_\lambda\in \rn$ such that $x_\lambda=\hbox{exp}_{p_{i_0,\lambda}}(r_\lambda \tilde{x}_\lambda)$. In particular $|\tilde{x}_\lambda|=1$. We set
$$I:=\{i\in \{1,...,N\}\hbox{ such that }d_g(p_{i,\lambda},p_{i_0,\lambda})=O(r_\lambda)\}.$$
For any $i\in I$, we define $\tilde{p}_{i,\lambda}\in\rn$ such that $p_{i,\lambda}:=\hbox{exp}_{p_{i_0,\lambda}}(r_\lambda \tilde{p}_{i,\lambda})$. Note that  $i_0\in I$ and $\tilde{p}_{i_0,\lambda}=0$. We define $\tilde{R}_{I,\lambda}(x):=\min_{i\in I}|x-\tilde{p}_{i,\lambda}|$ for all $x\in \rn$. We fix $K>0$. As one checks, using \eqref{lem:lip}, there exists $C=C(K)>0$ such that
\begin{equation}\label{ineq:R}
C^{-1} r_\lambda \tilde{R}_{I,\lambda}(x)\leq R_\lambda( \hbox{exp}_{p_{i_0,\lambda}}(r_\lambda x))\leq C r_\lambda \tilde{R}_{I,\lambda}(x)
\end{equation}
for all $x\in B_K(0)$. Therefore, for all $x\in B_K(0)$, we have that
\begin{eqnarray*}
R_\lambda^\gamma ( \hbox{exp}_{p_{i_0,\lambda}}(r_\lambda x))|\psi_\lambda(\hbox{exp}_{p_{i_0,\lambda}}(r_\lambda x))|\leq R_\lambda^\gamma(x_\lambda) |\psi_\lambda(x_\lambda)|
\end{eqnarray*}
\begin{equation}\label{bnd:tpl}
\hbox{so that }\tilde{R}_{I,\lambda}^\gamma(x) |\tpl(x)|\leq C^{-1}\hbox{ for all }x\in B_K(0) \hbox{ and }\tpl(\tilde{x}_\lambda)=1.
\end{equation}
Defining $g_\lambda:=(\hbox{exp}_{p_{i_0,\lambda}}^\star g)(r_\lambda\cdot)$ where $\hbox{exp}_{p }^\star g$ is the pull-back metric of $g$, the equation satisfied by $\tpl$ in \eqref{eq:75} rewrites
\begin{equation}\label{eq:tpl:bis}
\Delta_{g_\lambda}^k\tpl+\sum_{l=0}^{2k-2}r_\lambda^{2k-l}(B_l)_{\lambda}(\hbox{exp}_{p_{i_0,\lambda}}(r_\lambda \cdot))\star\nabla_{g_\lambda}^l\tpl-r_\lambda^{2k}V_\lambda(\hbox{exp}_{p_{i_0,\lambda}}(r_\lambda X))\tpl=0
\end{equation}
weakly in $B_{K}(0)$ for some bounded coefficient $((B_l)_\lambda)$'s. Note that
\begin{equation}\label{bnd:Vl}
|r_\lambda^{2k}V_\lambda(\hbox{exp}_{p_{i_0,\lambda}}(r_\lambda x))|\leq \lambda C \tilde{R}_{I,\lambda}(x)^{-2k}
\end{equation}
for all $\lambda>0$ and $X\in B_K(0)-\{\tilde{p}_{i,\lambda}/\, i\in I\}$. Up to extraction, we set $\tilde{p}_{i,0}:=\lim_{\lambda\to 0} \tilde{p}_{i,\lambda}$ for all $i\in I$. With equation \eqref{eq:tpl:bis}, the bounds \eqref{bnd:tpl}, \eqref{bnd:Vl} and the bounds of the coefficients $(B_l)_\lambda$, it follows from regularity theory (\cite{ADN} or Theorems D.1 and D.2 of \cite{robert:gjms}) that, up to extraction, there exists $\tilde{\psi}\in C^{2k}(\rn-\{\tilde{p}_{i,0}/\, i\in I\})$ such that $\tpl\to \tilde{\psi}$ in $C^{2k-1}_{loc}(\rn-\{\tilde{p}_{i,0}/\, i\in I\})$ as $\lambda\to 0$ and $ \Delta^k_\xi\tilde{\psi}=0$ in $\rn-\{\tilde{p}_{i,0}/\, i\in I\}$. We define $\tilde{x}_0:=\lim_{\lambda\to 0}\tilde{x}_\lambda$, so that, with \eqref{ineq:R}, we get that  $\tilde{R}_{I,\lambda}(\tilde{x}_\lambda)\geq c$, and then, passing to the limit, we get $\tilde{x}_0\not\in \{\tilde{p}_{i,0}/\, i\in I\}$. Finally, passing to the limit in \eqref{bnd:tpl}, we get that
 \begin{equation}\label{eq:psi}
\left\{\begin{array}{l}
\tilde{\psi}\in C^{2k}(\rn-\{\tilde{p}_{i,0}/\, i\in I\})\\
\Delta^k_\xi\tilde{\psi}=0 \hbox{  in }\rn-\{\tilde{p}_{i,0}/\, i\in I\}\\
|\tilde{\psi}(\tilde{x}_0)|=1\hbox{ with }\tilde{x}_0\in\rn- \{\tilde{p}_{i,0}/\, i\in I\}\\
\tilde{R}_{I,0}(x)^\gamma|\tilde{\psi}(x)|\leq C\hbox{ for all }x\in \rn-\{\tilde{p}_{i,0}/\, i\in I\}
\end{array}\right\}
\end{equation} 
where $\tilde{R}_{I,0}(x):=\min_{i\in I}|x-\tilde{p}_{i,0}|$. By standard elliptic theory (\cite{ADN} or Theorem D.2 of \cite{robert:gjms}), for any $l=1,...,2k$, there exists $C_l>0$ such that
\begin{equation*} 
\tilde{R}_{I,0}(x)^{\gamma+l}|\nabla^l\tilde{\psi}(x)|\leq C_l  \hbox{ for all }x\in \rn- \{\tilde{p}_{i,0}/\, i\in I\}.
\end{equation*}
\noindent{\bf Step 1.3: Contradiction via Green's formula.} Let us consider the Poisson kernel of $\Delta^k$ at $\tilde{x}_0$, namely
\begin{equation}\label{def:gamma}
\Gamma_{\tilde{x}_0}(x):=C_{n,k}|x-\tilde{x}_0|^{2k-n}\hbox{ for all }x\in \rn-\{\tilde{x}_0\},
\end{equation}
where $C_{n,k}^{-1}:=(n-2)\omega_{n-1}\Pi_{i=1}^{k-1}(n-2k+2(i-1))(2k-2i)$. Up to taking a subset of $I$, we assume that $\tilde{p}_{i,0}\neq\tilde{p}_{j,0}$ for all $i\neq j$ in $I$ (note that $i_0\in I$). Let us choose $R>3+\max\{|\tilde{p}_{i,0}|/\, i\in I\}$ and $0<\eps<\frac{1}{2}\min_{i\neq j\in I}\{|\tilde{p}_{i,0}-\tilde{p}_{j,0}|\}$ and define
$$\Omega_{R,\eps}:=B_R(0) - \left( B_{\eps}(\tilde{x}_0)\cup\bigcup_{i\in I}B_{R^{-1}}(\tilde{p}_{i,0})\right).$$
Integrating by parts, we have that
\begin{equation*}
\int_{\Omega_{R,\eps}}(\Delta^k_\xi\Gamma_{\tilde{x}_0})\tilde{\psi}\, dx=\int_{\Omega_{R,\eps}}\Gamma_{\tilde{x}_0}(\Delta^k_\xi\tilde{\psi})\, dx+\int_{\partial\Omega_{R,\eps}}\sum_{j=0}^{k-1}\tilde{B}_j(\Gamma_{\tilde{x}_0},\tilde{\psi})\, d\sigma
\end{equation*}
$$\hbox{where }\tilde{B}_j(\Gamma_{\tilde{x}_0},\tilde{\psi}):=-\partial_\nu\Delta_\xi^{k-j-1}\Gamma_{\tilde{x}_0}\Delta_\xi^j\tilde{\psi}+\Delta_\xi^{k-j-1}\Gamma_{\tilde{x}_0}\partial_\nu\Delta_\xi^j\tilde{\psi}.$$
The integrals on $\Omega_{R,\eps}$ vanish since $\Delta^k_\xi\Gamma_{\tilde{x}_0}=\Delta^k_\xi\tilde{\psi}=0$. As in \cite{robert:gjms}, the boundary terms involving $R$ go to $0$ as $R\to +\infty$, and when $\eps\to 0$, the terms also go to $0$ when $j\neq 0$. Finally, we get
$$\int_{\partial B_{\eps}(\tilde{x}_0)}\partial_\nu\Delta_\xi^{k-1}\Gamma_{\tilde{x}_0} \tilde{\psi}\, d\sigma=o(1)\hbox{ as }\eps\to 0.$$
The definition \eqref{def:gamma} yields $-\partial_\nu\Delta_\xi^{k-1}\Gamma_{\tilde{x}_0}(x)=\omega_{n-1}^{-1}|x-\tilde{x}_0|^{1-n}$ for all $x\neq \tilde{x}_0$. So that, passing to the limit, we get that $\tilde{\psi}(\tilde{x}_0)=0$, contradicting \eqref{eq:psi}. So \eqref{eq:75} does not hold. This proves \eqref{ineq:lem:bis} for $l=0$ under the assumption \eqref{hyp:lemma:1}. 

\medskip\noindent{\bf Case 2: Proof of the Lemma.} We   now     prove \eqref{ineq:lem:l:bis}. We argue by contradiction. We fix $N$, $\omega$, $\Omega$, $\gamma\in (0,n-2k)$, $p>1$ and $l\in \{0,..,2k-1\}$. We assume that \eqref{ineq:lem:l:bis} does not hold. Then, as above, we get the existence of   $p_{1,\lambda},...,p_{N,\lambda}\in \Omega$, $\psi_\lambda\in H_k^2(\Omega)$ and $x_\lambda\in \omega$ such that
\begin{equation}\label{eq:75:bis}
\left\{\begin{array}{l}
(P_\lambda-V_\lambda)\psi_\lambda=0 \hbox{ weakly in }H_k^2(\Omega)\cap H_{2k,loc}^s(\Omega-\{p_{i,\lambda}/\, i=1,..,N\})\\
\Vert\psi_\lambda\Vert_{L^p(\Omega)}=1\\
|V_{\lambda}(x)|\leq \lambda R_\lambda(x)^{-2k}\hbox{ for all }x\in \Omega-\{p_{i,\lambda}/\, i=1,..,N\}\\
V_{\lambda}\in L^\infty(\Omega)\\
\lim_{\lambda\to 0} R_\lambda(x_\lambda)^{\gamma+l}|\nabla^l\psi_\lambda(x_\lambda)|=+\infty.
\end{array}\right\}
\end{equation} 
Since $(P_\lambda-V_\lambda)\psi_\lambda=0$, elliptic theory yields uniform boundedness of $\Vert\psi_\lambda\Vert_{C^{2k-1}}$ outside the $p_{i,\lambda}$'s. The last assertion of \eqref{eq:75:bis} then yields $\lim_{\lambda\to 0}R_\lambda(x_\lambda)=0$. 
We let $i_0\in\{1,...,N\}$ be such that, up to extraction, $R_\lambda(x_\lambda)=d_g(x_\lambda, p_{i_0,\lambda})$ for all $\lambda\to 0$. We let $x_0\in \overline{\omega}$ such that $\lim_{\lambda\to 0}x_\lambda=\lim_{\lambda\to 0}p_{i_0,\lambda}=x_0$. We set $I_0:=\{i=1,...,N\hbox{ such that }\lim_{\lambda\to 0}p_{i,\lambda}=x_0\}$ and $\delta>0$ such that $B_{4\delta}(x_0)\subset \Omega$ and $d_g(p_{i,\lambda},x_0)\geq 3\delta$ for all $i\not\in I_0$ and all $\lambda>0$. It follows from \eqref{eq:75:bis} that
\begin{equation*}
|V_{\lambda}(x)|\leq \lambda \tilde{R}_\lambda(x)^{-2k}\hbox{ for all }x\in B_{2\delta}(x_0)-\{p_{i,\lambda}/\, i\in I_0\}
\end{equation*}
where $\tilde{R}_\lambda(x):=\min\{d_g(x, p_{i,\lambda})/\, i\in I_0\}$. Up to taking $\lambda$ smaller, we assume that $x_\lambda\in B_{\delta}(x_0)$ and $p_{i,\lambda}\in B_\delta(x_0)\subset\subset B_{2\delta}(x_0)$ for all $i\in I_0$ and $\lambda>0$. It  follows from the proof of \eqref{ineq:lem:bis} under the assumption \eqref{hyp:lemma:1} that there exists $C>0$ such that 
\begin{equation}\label{ineq:R:C0}
\tilde{R}_\lambda(x )^{\gamma }| \psi_\lambda(x )|\leq C\hbox{ for all }x\in B_{\delta}(x_0)-\{p_{i,\lambda}/\, i\in I_0\}.
\end{equation}
We set $r_\lambda:=\tilde{R}_\lambda(x_\lambda)=R_\lambda(x_\lambda)=d_g(x_\lambda, p_{i_0,\lambda})$ and for all $X\in B_{r_\lambda^{-1}\delta}(0)\subset\rn$, we define $\tilde{\psi}_\lambda(X):=r_\lambda^\gamma \psi_\lambda\left(\hbox{exp}_{p_{i_0,\lambda}}(r_\lambda X)\right)$. We let $\tilde{x}_\lambda\in B_2(0)$ be such that $x_\lambda:=\hbox{exp}_{p_{i_0,\lambda}}(r_\lambda \tilde{x}_\lambda)$. Taking the same notations as in Step 1.2, \eqref{ineq:R:C0} reads $\tilde{R}_{\lambda, I}(X)^{\gamma }| \tilde{\psi}_\lambda(X )|\leq C$ for all $|X|<r_\lambda^{-1}\delta$. As in Step 1.2, $\tilde{\psi}_\lambda$ satisfies an elliptic PDEs with coefficients that are bounded outside the $\{\tilde{p}_{i,\lambda}/\, i\in I\subset I_0\}$. Here again, elliptic theory yields an upper bound on the $C^{2k-1}-$norm far from $\{\tilde{p}_{i,\lambda}/\, i\in I\}$. The last assertion of \eqref{eq:75:bis} reads $\lim_{\lambda\to 0}|\nabla^l\tilde{\psi}_\lambda(\tilde{x}_\lambda)|=+\infty$ with $|\tilde{x}_\lambda-\tilde{p}_{i,\lambda}|\geq \epsilon_0>0$ for all $\lambda\to 0$ and $i\in I$. This contradicts the boundedness of the $C^{2k-1}-$norm.  This proves \eqref{ineq:lem:l:bis} and ends the proof of Lemma \ref{lem:bis}.


\begin{thebibliography}{12}
\bib{ADN}{article}{
   author={Agmon, S.},
   author={Douglis, A.},
   author={Nirenberg, L.},
   title={Estimates near the boundary for solutions of elliptic partial
   differential equations satisfying general boundary conditions. I},
   journal={Comm. Pure Appl. Math.},
   volume={12},
   date={1959},
   pages={623--727},
}

\bib{AP}{article}{
   author={Atkinson, Frederick V.},
   author={Peletier, Lambertus A.},
   title={Elliptic equations with nearly critical growth},
   journal={J. Differential Equations},
   volume={70},
   date={1987},
   number={3},
   pages={349--365},
  
}
 
\bib{carletti}{article}{
   author={Carletti, Lorenzo},
   title={Attaining the optimal constant for higher-order Sobolev
   inequalities on manifolds via asymptotic analysis},
   journal={J. Lond. Math. Soc. (2)},
   volume={111},
   date={2025},
   number={5},
 }     

\bib{CA-P}{article}{
 author={Cheikh-Ali, Hussein},
 author={Premoselli, Bruno},
 journal={Analysis and PDEs, to appear},
 title={Compactness results for Sign-Changing Solutions of critical nonlinear elliptic equations of low energy},
 date={2024},
}

\bib{cs:lin:1}{article}{
   author={Chen, Chiun-Chuan},
   author={Lin, Chang-Shou},
   title={Estimates of the conformal scalar curvature equation via the
   method of moving planes},
   journal={Comm. Pure Appl. Math.},
   volume={50},
   date={1997},
   number={10},
   pages={971--1017},
}

\bib{cs:lin:2}{article}{
   author={Chen, Chiun-Chuan},
   author={Lin, Chang-Shou},
   title={Estimate of the conformal scalar curvature equation via the method
   of moving planes. II},
   journal={J. Differential Geom.},
   volume={49},
   date={1998},
   number={1},
   pages={115--178},
}


       \bib{druet:jdg}{article}{
   author={Druet, Olivier},
   title={From one bubble to several bubbles: the low-dimensional case},
   journal={J. Differential Geom.},
   volume={63},
   date={2003},
   number={3},
   pages={399--473},
}

\bib{druet:imrn}{article}{
   author={Druet, Olivier},
   title={Compactness for Yamabe metrics in low dimensions},
   journal={Int. Math. Res. Not.},
   date={2004},
   number={23},
   pages={1143--1191},
}


\bib{DH}{article}{
   author={Druet, Olivier},
   author={Hebey, Emmanuel},
   title={Stability for strongly coupled critical elliptic systems in a
   fully inhomogeneous medium},
   journal={Anal. PDE},
   volume={2},
   date={2009},
   number={3},
   pages={305--359},
   }

\bib{DHR}{book}{
   author={Druet, Olivier},
   author={Hebey, Emmanuel},
   author={Robert, Fr\'{e}d\'{e}ric},
   title={Blow-up theory for elliptic PDEs in Riemannian geometry},
   series={Mathematical Notes},
   volume={45},
   publisher={Princeton University Press, Princeton, NJ},
   date={2004},
   pages={viii+218},
   
}

\bib{druet:laurain}{article}{
   author={Druet, Olivier},
   author={Laurain, Paul},
   title={Stability of the Poho\v zaev obstruction in dimension 3},
   journal={J. Eur. Math. Soc. (JEMS)},
   volume={12},
   date={2010},
   number={5},
   pages={1117--1149},
}

\bib{gmr:jde}{article}{
   author={Ghoussoub, Nassif},
   author={Mazumdar, Saikat},
   author={Robert, Fr\'{e}d\'{e}ric},
   title={The Hardy-Schr\"{o}dinger operator on the Poincar\'{e} ball:
   compactness, multiplicity, and stability of the Pohozaev obstruction},
   journal={J. Differential Equations},
   volume={320},
   date={2022},
   pages={510--557},
 }

\bib{gmr:memams}{article}{
   author={Ghoussoub, Nassif},
   author={Mazumdar, Saikat},
   author={Robert, Fr\'{e}d\'{e}ric},
   title={Multiplicity and stability of the Pohozaev obstructions for Hardy-Schr\"odinger equations with boundary singularity},
   journal={Memoirs of the American Mathematical Society},
   volume={285},
   number={1415},
   date={2023},
   pages={510--557},
 }



\bib{gkw}{misc}{
 author={Gong, Liuwei},
 author={Kim, Seunghyeok},
 author={Wei, Juncheng},
 review={arXiv:2502.14237},
 title={Compactness of the $Q$-curvature problem},
 date={2025},
}

 
   \bib{gr:arma}{article}{
   author={Grunau, Hans-Christoph},
   author={Robert, Fr\'ed\'eric},
   title={Positivity and almost positivity of biharmonic Green's functions
   under Dirichlet boundary conditions},
   journal={Arch. Ration. Mech. Anal.},
   volume={195},
   date={2010},
   number={3},
   pages={865--898},

}

   \bib{grs}{article}{
   author={Grunau, Hans-Christoph},
   author={Robert, Fr\'ed\'eric},
   author={Sweers, Guido},
   title={Optimal estimates from below for biharmonic Green functions},
   journal={Proc. Amer. Math. Soc.},
   volume={139},
   date={2011},
   number={6},
   pages={2151--2161},
}
 
 
 \bib{zchan}{article}{
   author={Han, Zheng-Chao},
   title={Asymptotic approach to singular solutions for nonlinear elliptic
   equations involving critical Sobolev exponent},
   journal={Ann. Inst. H. Poincar\'{e} C Anal. Non Lin\'{e}aire},
   volume={8},
   date={1991},
   number={2},
   pages={159--174},
   }
   
   
   \bib{hebey:2003}{article}{
   author={Hebey, Emmanuel},
   title={Sharp Sobolev inequalities of second order},
   journal={J. Geom. Anal.},
   volume={13},
   date={2003},
   number={1},
   pages={145--162},
   }

\bib{hebey.eth}{book}{
   author={Hebey, Emmanuel},
   title={Compactness and stability for nonlinear elliptic equations},
   series={Zurich Lectures in Advanced Mathematics},
   publisher={European Mathematical Society (EMS), Z\"{u}rich},
   date={2014},
   pages={x+291},
 }



   
   \bib{HR}{article}{
   author={Hebey, Emmanuel},
   author={Robert, Fr\'{e}d\'{e}ric},
   title={Asymptotic analysis for fourth order Paneitz equations with
   critical growth},
   journal={Adv. Calc. Var.},
   volume={4},
   date={2011},
   number={3},
   pages={229--275},
   
}




		
\bib{HRW}{article}{
   author={Hebey, Emmanuel},
   author={Robert, Fr\'{e}d\'{e}ric},
   author={Wen, Yuliang},
   title={Compactness and global estimates for a fourth order equation of
   critical Sobolev growth arising from conformal geometry},
   journal={Commun. Contemp. Math.},
   volume={8},
   date={2006},
   number={1},
   pages={9--65},
  
}

\bib{HV}{article}{
   author={Hebey, Emmanuel},
   author={Vaugon, Michel},
   title={The best constant problem in the Sobolev embedding theorem for
   complete Riemannian manifolds},
   journal={Duke Math. J.},
   volume={79},
   date={1995},
   number={1},
   pages={235--279},
   }


\bib{kms}{article}{
   author={Khuri, Marcus A.},
   author={Marques, Fernando C.},
   author={Schoen, Richard M.},
   title={A compactness theorem for the Yamabe problem},
   journal={J. Differential Geom.},
   volume={81},
   date={2009},
   number={1},
   pages={143--196},
}



\bib{li-95}{article}{
   author={Li, YanYan},
   title={Prescribing scalar curvature on $S^n$ and related problems. I},
   journal={J. Differential Equations},
   volume={120},
   date={1995},
   number={2},
   pages={319--410},
}



\bib{li-96}{article}{
   author={Li, YanYan},
   title={Prescribing scalar curvature on $S^n$ and related problems. II.
   Existence and compactness},
   journal={Comm. Pure Appl. Math.},
   volume={49},
   date={1996},
   number={6},
   pages={541--597},
}



 
\bib{lx}{article}{
   author={Li, YanYan},
   author={Xiong, Jingang},
   title={Compactness of conformal metrics with constant $Q$-curvature. I},
   journal={Adv. Math.},
   volume={345},
   date={2019},
   pages={116--160},
  }


\bib{LiZhang2005}{article}{
   author={Li, YanYan},
   author={Zhang, Lei},
   title={Compactness of solutions to the Yamabe problem. II},
   journal={Calc. Var. Partial Differential Equations},
   volume={24},
   date={2005},
   number={2},
   pages={185--237},
}


\bib{LiZha3}{article}{
   author={Li, YanYan},
   author={Zhang, Lei},
   title={Compactness of solutions to the Yamabe problem. III},
   journal={J. Funct. Anal.},
   volume={245},
   date={2007},
   number={2},
   pages={438--474},
}



\bib{LiZhu}{article}{
   author={Li, YanYan},
   author={Zhu, Meijun J.},
   title={Yamabe type equations on three-dimensional Riemannian manifolds},
   journal={Commun. Contemp. Math.},
   volume={1},
   date={1999},
   number={1},
   pages={1--50},
}



\bib{cs:lin:3}{article}{
   author={Lin, Chang-Shou},
   title={Estimates of the scalar curvature equation via the method of
   moving planes. III},
   journal={Comm. Pure Appl. Math.},
   volume={53},
   date={2000},
   number={5},
   pages={611--646},
 }




\bib{PL2}{article}{
 author={Lions, Pierre-Louis},
  title={The concentration-compactness principle in the calculus of variations. The limit case. II},
 journal={Revista Matem{\'a}tica Iberoamericana},
 volume={1},
 number={2},
 pages={45--121},
 date={1985},
 publisher={EMS Press, Berlin},
 eprint={https://eudml.org/doc/39321},
}


\bib{Mar}{article}{
   author={Marques, Fernando C.},
   title={A priori estimates for the Yamabe problem in the non-locally conformally flat case},
   journal={J. Differential Geom.},
   volume={71},
   date={2005},
   number={2},
   pages={315--346},
}


\bib{mazumdar:jde}{article}{
   author={Mazumdar, Saikat},
   title={GJMS-type operators on a compact Riemannian manifold: best
   constants and Coron-type solutions},
   journal={J. Differential Equations},
   volume={261},
   date={2016},
   number={9},
   pages={4997--5034},
   }

\bib{mazumdar:cpaa}{article}{
 author={Mazumdar, Saikat},
 title={Struwe's decomposition for a polyharmonic operator on a compact Riemannian manifold with or without boundary},
 journal={Communications on Pure and Applied Analysis},
 volume={16},
 number={1},
 pages={311--330},
 date={2017},
}

 

\bib{mitidieri}{article}{
   author={Mitidieri, Enzo},
   title={A simple approach to Hardy inequalities},
    journal={Mat. Zametki},
   volume={67},
   date={2000},
   number={4},
   pages={563--572},
   translation={
      journal={Math. Notes},
      volume={67},
      date={2000},
      number={3-4},
      pages={479--486},
    }
   }
 
\bib{premoselli}{article}{
   author={Premoselli, Bruno},
   title={A priori estimates for finite-energy sign-changing blowing-up
   solutions of critical elliptic equations},
   journal={Int. Math. Res. Not. IMRN},
   date={2024},
   number={6},
   pages={5212--5273},
}


\bib{PR}{article}{
   author={Premoselli, Bruno},
   author={Robert, Fr\'ed\'eric},
   title={One-bubble nodal blow-up for asymptotically critical stationary Schr\"odinger-type equations},
  journal = {Journal of Functional Analysis},
volume = {288},
number = {6},
pages = {110808},
year = {2025},
}



\bib{premo-vetois-jmpa}{article}{
   author={Premoselli, Bruno},
   author={V\'etois, J\'er\^ome},
   title={Stability and instability results for sign-changing solutions to
   second-order critical elliptic equations},
   journal={J. Math. Pures Appl. (9)},
   volume={167},
   date={2022},
   pages={257--293},
}


\bib{premo-vetois-aim}{article}{
   author={Premoselli, Bruno},
   author={V\'etois, J\'er\^ome},
   title={Sign-changing blow-up for the Yamabe equation at the lowest energy
   level},
   journal={Adv. Math.},
   volume={410},
   date={2022},
   pages={Paper No. 108769, 50},
}

\bib{premo-vetois-eigen}{article}{
   author={Premoselli, Bruno},
   author={V\'etois, J\'er\^ome},
   title={Nonexistence of minimizers for the second conformal eigenvalue near the round sphere in low dimensions},
   note={https://arxiv.org/abs/2408.07823},
   }

 

 \bib{robert:gjms}{misc}{
 author={Robert, Fr\'ed\'eric},
 review={arXiv:2501.00531},
 title={Localization of bubbling for high order nonlinear equations},
 date={2024},
}

 

\bib{RV:jdg}{article}{
   author={Robert, Fr\'ed\'eric},
   author={V\'etois, J\'er\^ome},
   title={Examples of non-isolated blow-up for perturbations of the scalar
   curvature equation on non-locally conformally flat manifolds},
   journal={J. Differential Geom.},
   volume={98},
   date={2014},
   number={2},
   pages={349--356},
}

 \bib{schoen-montecatini}{article}{
   author={Schoen, Richard M.},
   title={Variational theory for the total scalar curvature functional for
   Riemannian metrics and related topics},
   conference={
      title={Topics in calculus of variations},
      address={Montecatini Terme},
      date={1987},
   },
   book={
      series={Lecture Notes in Math.},
      volume={1365},
      publisher={Springer, Berlin},
   },
   date={1989},
   pages={120--154},
  
}
\bib{schoen1988}{article}{
   author={Schoen, Richard M.},
   title={Notes from graduate lectures in Stanford University},
   note={https://sites.math.washington.edu/~pollack/research/Schoen-1988-notes.html},
   date={1988},
   }
   
   
   \bib{sz}{article}{
   author={Schoen, Richard M.},
   author={Zhang, Dong},
   title={Prescribed scalar curvature on the $n$-sphere},
   journal={Calc. Var. Partial Differential Equations},
   volume={4},
   date={1996},
   number={1},
   pages={1--25},
  
}

 \bib{struwe:1984}{article}{
 author={Struwe, Michael},
  title={A global compactness result for elliptic boundary value problems involving limiting nonlinearities},
 journal={Mathematische Zeitschrift},
 volume={187},
 pages={511--517},
 date={1984},
 publisher={Springer, Berlin/Heidelberg},
 eprint={https://eudml.org/doc/173508},
}


\bib{vdv}{article}{
   author={Van der Vorst, R. C. A. M.},
   title={Best constant for the embedding of the space $H^2\cap
   H^1_0(\Omega)$ into $L^{2N/(N-4)}(\Omega)$},
   journal={Differential Integral Equations},
   volume={6},
   date={1993},
   number={2},
   pages={259--276},
  }
  
\bib{weixu}{article}{
   author={Wei, Juncheng},
   author={Xu, Xingwang},
   title={Classification of solutions of higher order conformally invariant
   equations},
   journal={Math. Ann.},
   volume={313},
   date={1999},
   number={2},
   pages={207--228},
   }

\bib{wz}{article}{
 author={Wei, Juncheng},
 author={Zhao, Chunyi},
 title={Non-compactness of the prescribed {{\(Q\)}}-curvature problem in large dimensions},
 journal={Calculus of Variations and Partial Differential Equations},
 volume={46},
 number={1-2},
 pages={123--164},
 date={2013},
 publisher={Springer, Berlin/Heidelberg},
}

\end{thebibliography}
\end{document}